\documentclass[11pt]{article}
\usepackage{graphicx,color,amsfonts,amsmath,amssymb,enumerate}
\usepackage{url,fancyhdr,indentfirst}
\usepackage[margin=1.25in]{geometry}
   \usepackage{bigints}
\usepackage{amsthm,comment,bm}
    \usepackage[colorlinks=true,citecolor=blue]{hyperref}
\usepackage[numbers]{natbib}
\usepackage{dsfont}
\usepackage{mathrsfs}
\usepackage{tikz-qtree}
\usepackage{xr}
\usepackage{threeparttable}
\usepackage[title]{appendix}
\usepackage{amssymb}
\usepackage{graphicx}
\usepackage{amsmath}
\usepackage{commath}
\usepackage{bm}
\usepackage{tikz}
\usepackage{bbm}

\usepackage{caption}
\usepackage{pgfplots}
\usepackage{subcaption}
\usepgfplotslibrary{fillbetween}
\usetikzlibrary{patterns}
\usepackage{color}
\usepackage{dsfont} 
  
\newcommand{\n}[1]{\left\lVert#1\right\rVert}
\usetikzlibrary{decorations.pathreplacing}

\usepackage{tkz-graph}
\usepackage{pgfplots}
\usetikzlibrary{calc,patterns}

\usepackage{pgfplots}
\usepackage{subcaption}
\usepgfplotslibrary{fillbetween}
\usetikzlibrary{patterns}
\usepackage{amssymb}
\usepackage{graphicx}
\usepackage{amsmath}
\usepackage{bm}
\usepackage{tikz}
\usepackage{color}
\usepackage{caption}
  
\usetikzlibrary{decorations.pathreplacing}

\usepackage{tkz-graph}
\usepackage{pgfplots}
\usetikzlibrary{calc,patterns}

\def\d{\mathrm{d}}

\newcommand{\var}{\mathrm{Var}}

\newcommand{\E}{\mathbb{E}}
\newcommand{\R}{\mathbb{R}}

\newcommand{\Z}{\mathbb{Z}}
\newcommand{\bT}{\mathbb{T}}
\newcommand{\C}{\mathbb{C}}

\newcommand{\M}{\mathcal{M}}

\newcommand{\hmc}{\mathrm{HMC}}
\newcommand{\gmc}{\mathrm{GMC}}

\newcommand{\N}{\mathbb{N}}
\renewcommand{\P}{\mathbb{P}}
\newcommand{\Q}{\mathbb{Q}}
\newcommand{\bn}{\mathbf{n}}
\renewcommand{\bm}{\mathbf{m}}

\newcommand{\T}{\mathcal T}
\newcommand{\A}{\mathcal A}

\newcommand{\cP}{\mathcal P}
\newcommand{\cL}{\mathcal L}

\newcommand{\ddd}{\overset{\text{d}}{=}}
\newcommand{\bone}{ {\mathbbm{1}} }

\renewcommand{\ge}{\geqslant}
\renewcommand{\le}{\leqslant}
\renewcommand{\geq}{\geqslant}
\renewcommand{\leq}{\leqslant}

\newcommand{\ee}{\varepsilon}

\theoremstyle{plain}

\newtheorem{theorem}{Theorem}[section]
\newtheorem{corollary}[theorem]{Corollary}
\newtheorem{lemma}[theorem]{Lemma}
\newtheorem{proposition}[theorem]{Proposition}

\newtheorem{definition}[theorem]{Definition}

\theoremstyle{remark}
\newtheorem{remark}{Remark}

\theoremstyle{definition}

\newtheorem{question}{Question}
\renewcommand{\cite}{\citet}
\renewcommand{\cdots}{\dots}

\newcommand{\ceil}[1]{\lceil #1\rceil}

\newcommand{\floor}[1]{\lfloor #1\rfloor}

\usepackage{tikz}
\usetikzlibrary{arrows.meta,positioning}

\usepackage[onehalfspacing]{setspace}

\allowdisplaybreaks
\pgfplotsset{compat=1.18}
\usepackage{authblk}
\usepackage{blindtext}
\begin{document}

\date{}

\title{Universality and Phase Transitions in Low Moments of Secular Coefficients of Critical Holomorphic Multiplicative Chaos}

\author{Haotian Gu\thanks{Department of Mathematics, Duke University, USA. Email: haotian.gu@duke.edu}\hspace{1cm} Zhenyuan Zhang\thanks{Department of Mathematics, Stanford University, USA. Email: zzy@stanford.edu}}

\maketitle

\begin{abstract}
We investigate the low moments $\mathbb{E}[|A_N|^{2q}],\, 0<q\leq 1$ of {secular coefficients} $A_N$ of the {critical non-Gaussian holomorphic multiplicative chaos}, i.e.~coefficients of $z^N$ in the power series expansion of $\exp(\sum_{k=1}^\infty X_kz^k/\sqrt{k})$, where $\{X_k\}_{k\geq 1}$ are i.i.d.~rotationally invariant unit variance complex random variables. Inspired by Harper's remarkable result on random multiplicative functions, Soundararajan and Zaman recently showed that if each $X_k$ is standard complex Gaussian, $A_N$ features better-than-square-root cancellation: $\mathbb{E}[|A_N|^2]=1$ and $\mathbb{E}[|A_N|^{2q}]\asymp (\log N)^{-q/2}$ for fixed $q\in(0,1)$ as $N\to\infty$.
We show that this asymptotics holds universally if $\mathbb{E}[e^{\gamma|X_k|}]<\infty$ for some $\gamma>2q$. As a consequence, we establish the universality for the tightness of the normalized secular coefficients $A_N(\log(1+N))^{1/4}$, generalizing a result of Najnudel, Paquette, and Simm. Another corollary is the almost sure regularity of some critical non-Gaussian holomorphic chaos in appropriate Sobolev spaces.
Moreover, we characterize the asymptotics of $\mathbb{E}[|A_N|^{2q}]$ for $|X_k|$ following a stretched exponential distribution with an arbitrary scale parameter, which exhibits a completely different behavior and underlying mechanism from the Gaussian universality regime. 
As a result, we unveil a double-layer phase transition around the critical case of exponential tails. 
Our proofs combine Harper's robust approach with a careful analysis of the (possibly random) leading terms in the monomial decomposition of $A_N$.
\end{abstract}


\tableofcontents

\section{Introduction}
For polynomials $\phi$ on the closed unit disc $\overline{\mathbb{D}}$ and a fixed constant $\vartheta>0$, the \textit{holomorphic multiplicative chaos $\hmc_\vartheta$} is defined as a random distribution (or functional) \begin{align}\label{defn-HMC}
    (\hmc_\vartheta,\phi):=\lim_{r\uparrow 1}\frac{1}{2\pi}\int_{-\pi}^{\pi}\exp\bigg(\sqrt{\vartheta}\sum_{k=1}^{\infty}\frac{X_k}{\sqrt{k}}(re^{i\theta})^k\bigg)\overline{\phi(r\theta)}\,\d\theta,
\end{align} where $\{X_k\}_{k\geq 1}$ is a sequence of i.i.d.~rotationally invariant unit variance (``standardized") complex random variables. The existence and uniqueness of the limit is straightforward for trigonometric polynomials since it is determined by the \textit{secular coefficients}  \begin{align}\label{eq:basic eq}
    A_N:= (\hmc_\vartheta,\theta\mapsto e^{iN\theta})=[z^N]\exp\bigg(\sqrt{\vartheta}\sum_{k=1}^{\infty}\frac{X_k}{\sqrt{k}}z^k\bigg),\quad N\in\N,
\end{align} where $[z^N]f(z)$ denotes the coefficient of $z^N$ in the power series expansion of $f$ around $z=0$, i.e., \[\exp\bigg(\sqrt{\vartheta}\sum_{k=1}^\infty\frac{X_k}{\sqrt{k}}z^k\bigg) = \sum_{N=0}^\infty A_N z^N.\] 
For instance, $A_2=\vartheta X_1^2/2+X_2\sqrt{\vartheta/2}$.

The case of a general $\vartheta$ where $X_k$ is standard complex Gaussian has been investigated by Najnudel, Paquette, and Simm \citep{najnudel2023secular} and more recently Najnudel, Paquette, Simm, and Truong \citep{najnudel2025fourier}, where the random measure \eqref{defn-HMC} is defined in (1.11) therein as a distributional limit of the characteristic polynomial of \textit{Circular Beta Ensemble} (C$\beta$E) inside the unit disc ($\beta=2/\vartheta$) and belongs to the Sobolev space $H^{s}$ almost surely for any $s<s(\vartheta)$, see Theorem 1.4 therein. The critical case of $\vartheta=1$ is of particular interest, due to its connections to the \textit{Circular Unitary Ensemble} (CUE) and random multiplicative functions. Moreover, the behavior of secular coefficients is more subtle in the critical case. We summarize the relevant literature in Section \ref{sec:intro-conn-rmt}.

In this work, our main focus will be the critical case $\vartheta=1$. 
If each $X_k$ is standard complex Gaussian, $\hmc_1$ first appeared in Appendix C of Saksman and Webb \citep{saksman2020riemann} and was shown to be $H^s$ for any $s<-1/2$ a.s.~while also a.s.~not $H^s$ for any $s>-1/2$ \citep{najnudel2023secular}. 
Recently, the work of Soundararajan and Zaman \citep{soundararajan2022model} studied the \textit{low moments} $\E[|A_N|^{2q}],\,q\in(0,1]$ of secular coefficients of $\hmc_1$ as a model problem for multiplicative chaos in number theory. 
By exploiting the connections to Harper's remarkable result \citep{harper2020moments} on ``better-than-square-root" cancellation for random multiplicative functions, they proved the following.\begin{theorem}[Theorem 2.1 of \citep{soundararajan2022model}]\label{thm:SZ-main-thm}
   For standard complex Gaussian variables $\{X_k\}_{k\geq 1}$ and $\vartheta=1$, uniformly in $q\in(0,1]$ and $N\geq 1$, \[\E[|A_N|^{2q}]\asymp \left(\frac{1}{1+(1-q)\sqrt{\log N}}\right)^{q}.\] 
\end{theorem}
We also remark that Corollary 1.1 of \citep{soundararajan2022model} also showed the equality $\E[|A_N|^2]=1$. As a consequence of Theorem \ref{thm:SZ-main-thm}, $A_N$ is typically smaller than what one expects from the central limit theorem, i.e.~applying H\"{o}lder's inequality to its second moment, indicating a complicated limiting behavior for secular coefficients at $\vartheta=1$. 
Parallel work of \citep{najnudel2023secular} also established the sharp tightness of secular coefficients.

 \begin{theorem}[Theorem 1.11 of \citep{najnudel2023secular}]\label{thm:NPS}  For standard complex Gaussian variables $\{X_k\}_{k\geq 1}$ and $\vartheta=1$, both the families $\{A_N/(\log(1+N))^{-1/4}\}_{N\in\N}$ and $\{(\log(1+N))^{-1/4}/A_N\}_{N\in\N}$ are tight.     
 \end{theorem}

Note that Theorem \ref{thm:SZ-main-thm} implies the first part of Theorem \ref{thm:NPS}. 
It is worth noticing that all aforementioned studies focused on the Gaussian setting. A natural question would be to examine the universality of current results and what happens beyond that regime.
In this article, we initiate the study of non-Gaussian holomorphic multiplicative chaos and investigate the low moments $\E[|A_N|^{2q}],~q\in[0,1]$ when $|X_k|$ follows a generic standardized distribution. 
A rather surprising phenomenon is that, even for the simplest setting of $q=1$, the second moments of $A_N$ may no longer remain of constant order as in the Gaussian case. For example, our results imply that $\E[|A_N|^2]$ grows exponentially if each $|X_k|\sim \mathrm{Exp}(1)$, polynomially if $|X_k|\sim \mathrm{Exp}(2)$, and remains of constant order if $|X_k|\sim \mathrm{Exp}(3)$.\footnote{Strictly speaking, these ``exponential distributions" need to be normalized to have unit variance, e.g.~their tails have the form $\P(|X_k|>t)=\exp(-\gamma(t-c_\gamma))\wedge 1$ for $\gamma=1,2,3$.}

Our main result consists of two parts. First, Theorem \ref{thm:main} below establishes a sharp criterion for \textit{Gaussian universality regime} of the asymptotics of $\E[|A_N|^{2q}],~q\in[0,1]$. Second, in Theorem \ref{thm:main2} below we completely characterize a subtle \emph{double-layer phase transition} associated with the tail of $|X_k|$; see Figure \ref{fig:phases} below for the phase diagram. As a result, we fully establish the asymptotics for low moments $\E[|A_N|^{2q}]$ with non-Gaussian inputs $\P(|X_k|>t)\sim\exp(-\gamma t^p)$ for all triples $(q,\gamma,p)\in[0,1]\times\R^{+}\times\R^{+}$.
Our main theorems generalize Theorem \ref{thm:SZ-main-thm} and show a new phenomenon of \emph{(super-)exponential cancellation} (that is,~$(\E[|A_N|])^2/\E[|A_N|^2]$ decays (super-)exponentially as $N\to\infty$) emerges as the tail of $|X_k|$ becomes heavier.

We establish universality by utilizing the connection between secular coefficients and the total mass of critical multiplicative chaos, which resembles a delicate joint effect of all inputs $X_k$'s. The double-layered phase transition, on the other hand, is a product of an intricate interplay between the dominance of $X_1$ (which occurs when the tail of $|X_1|$ is heavy) and a similar yet different joint effect as in the universality phase. 
We expand our discussions in Section \ref{sec:tech-challenge} and a more detailed description of our approach can be found in Section \ref{sec:proof strategy}.

Several corollaries are in place. First, we obtain a \emph{universality result} for the tightness of normalized secular coefficients of non-Gaussian holomorphic chaos (Corollary \ref{coro}), extending Theorem \ref{thm:NPS}. 
Second, we derive both a universal and a non-universal regularity result for some critical non-Gaussian holomorphic chaos (Corollary \ref{coro-reg}), partially extending the regularity of Gaussian $\hmc_1$.

Finally, based on these new observations, we propose a list of future research directions of broad interest, ranging from random matrix theory, and multiplicative chaos, to probabilistic number theory.
These questions will be discussed in detail in Section \ref{sec:intro-conn-rmt}.

\subsection{Statement of main results}
Informally speaking, our main results capture the asymptotics of all $(2q)$-moments ($q\in(0,1]$) of $A_N$ for $X_k$ satisfying any \textit{stretched exponential tail}, i.e.~$\P(|X_k|>t)\sim\exp(-t^p)$ for all $p\in(0,\infty)$. For example, $p=2$ resembles the Gaussian case studied in Theorems \ref{thm:SZ-main-thm} and \ref{thm:NPS}.

We now set the stage for formalizing our theorems. Let $\{\tau_k\}$ be i.i.d.~uniformly distributed on $[-\pi,\pi)$ and $\{R_k\}$ be i.i.d.~real random variables with $\E[|R_k|^2]=1$ and independent of $\{\tau_k\}$. Let $X_k=e^{i\tau_k}R_k$. We may assume that $R_k$ is symmetric by a standard symmetrization procedure. We introduce a few short-hand notations for the cases of interest, covering (roughly speaking) i) $p\in(1,\infty)$ case, ii) $p=1$ case, and iii) $p\in(0,1)$ case:
  \begin{itemize}
    \item ($q$-UNIV) For a given $q\in(0,1]$, $R_k$ has a finite $\gamma$-exponential moment for some $\gamma>2q$. That is,
$\E[e^{\gamma|R_k|}]$ is finite for some $\gamma>2q.$
    \item (EXP) $R_k$ is a two-sided shifted \emph{exponential} random variable with unit variance. That is, $\P(|R_k|\geq u)=\exp(-\gamma(u-c_\gamma))\wedge 1$ for $u\geq 0$, where $\gamma\in(0,2q]$ and $c_\gamma:=\log({\gamma^2}/{2})/\gamma$;
    \item (SE) $R_k$ follows a (symmetric) \textit{stretched exponential} distribution with exponent $p\in(0,1)$, i.e.~$\P(|R_k|\geq u)=\exp(-(u/c_p)^{p})$ for $u\geq 0$, where $c_p:=(2\Gamma(2/p)/p)^{-1/2}$. 
\end{itemize}
Here, we have implicitly imposed the constraints on the law of $R_k$ that $R_k$ is symmetric and $\E[|R_k|^2]=1$, which explains the shift $c_\gamma$ and the scaling parameter $c_p$. 


The cases (EXP) and (SE) are introduced as prototypes, where we do not attempt to achieve the greatest generality by considering the largest class of laws. Indeed, our proofs only rely on the absolute moment asymptotics of $X_1$ (along with rotational symmetry and the unit variance property).
Our main results are as follows. We refer to Section \ref{para:notations} for asymptotic notations; in particular, we allow the asymptotic constants to depend on $q$ and the distribution of $X_k$.

\begin{theorem}[Primary phase transition]
\label{thm:main}
Let $\vartheta=1$.

\noindent(i) Fix any $q\in(0,1]$.  Suppose that condition ($q$-UNIV) holds for the i.i.d.~input $\{X_k=e^{i\tau_k}R_k\}$. We have \begin{align}
    \E[|A_N|^{2q}]\asymp \left(\frac{1}{1+(1-q)\sqrt{\log N}}\right)^q\asymp\begin{cases}
        (\log N)^{-q/2}&\text{ if }q\in(0,1);\\
        1&\text{ if }q=1.
    \end{cases}\label{eq:asymp LT}
\end{align}
(ii) Fix any $q>0$.  Suppose that (SE) holds for the i.i.d.~input $\{X_k=e^{i\tau_k}R_k\}$. It holds that 
\begin{align}
    \E[|A_N|^{2q}]=(1+o(1)) (2\pi)^{{1}/{2}-q}\sqrt{\frac{2q}{p}} \left(\frac{2qc_p^p}{pe^{1-p}}\right)^{2qN/p}N^{1/2-q+2qN(1/p-1)}.\label{eq:asymp HT}
\end{align}Note that $2qc_p^p/(pe^{1-p})<1$ for all $0<p,q<1$.
\end{theorem}

 The asymptotics \eqref{eq:asymp LT} shows that the better-than-square-root cancellation phenomenon can still be quantitatively analyzed for non-Gaussian inputs in \eqref{eq:basic eq} under a suitable light-tailed assumption. This considerably extends the results in the Gaussian case (cf., Theorem \ref{thm:SZ-main-thm}). The result of \eqref{eq:asymp HT} in the heavy-tailed case indicates that the Gaussian-type behavior is limited to light-tailed distributions and does not apply in general to heavier tails (say, heavier than the exponential distribution). This gives rise to a phase transition in the thickness of the tail of $|R_k|$, leading naturally to the question of the \emph{primary criticality}: case (EXP). The following result completely characterizes the behavior in this critical phase, showcasing a second phase transition within this critical phase.

\begin{theorem}[Secondary phase transition]\label{thm:main2}
Fix $q\in(0,1]$ and let $\vartheta=1$. Suppose that (EXP) holds for the i.i.d.~input $\{X_k=e^{i\tau_k}R_k\}$. It holds that if $\gamma<2q$, \begin{align}
    \E[|A_N|^{2q}]\asymp 
    N^{1/2-q}\left(\frac{2q}{\gamma}\right)^{2qN};\label{eq:asymp SE}
\end{align}
 if $\gamma=2q$, 
\begin{align}\label{eq:gamma=2q result}
     \E[|A_N|^{2q}]\asymp \frac{N^{1-q+q^2/2}}{(1+(1-q)\sqrt{\log N})^q}\asymp\begin{cases}
        N^{1-q+q^2/2}(\log N)^{-q/2}&\text{ if }q\in(0,1);\\
        \sqrt{N}&\text{ if }q=1.
    \end{cases}
\end{align}
    and if $\gamma>2q$,
   \begin{align}\label{eq:gamma>2q result}
     \E[|A_N|^{2q}]\asymp \left(\frac{1}{1+(1-q)\sqrt{\log N}}\right)^q\asymp\begin{cases}
        (\log N)^{-q/2}&\text{ if }q\in(0,1);\\
        1&\text{ if }q=1.
    \end{cases}
\end{align}
\end{theorem}
Indeed, \eqref{eq:gamma>2q result} follows directly from Theorem \ref{thm:main} (i), and is included here for the sake of completeness. Asymptotic constants in \eqref{eq:asymp SE} and \eqref{eq:gamma>2q result} may depend on $\gamma$ which characterizes the distribution of $X_k$.
 The critical phase (EXP) showcases a further phase transition. There exists a critical moment exponent, which coincides with the parameter of the (shifted) exponential distribution $|R_k|$, which governs the asymptotics of the moments: Gaussian-type asymptotics of Theorem \ref{thm:SZ-main-thm} below such exponent ($\gamma>2q$) and exponential growth of moments above the exponent ($\gamma<2q$). At the \emph{secondary criticality}, where the moment exponent coincides with the parameter of the exponential distribution ($\gamma=2q$), the moment exhibits a regularly varying growth in $N$. 

To visualize the double-layer phase transition, we may restrict the class ($q$-UNIV) to laws satisfying 
\begin{align}
    \P(|X_k|>t)\asymp \exp(-\gamma t^p)\label{eq:repn}
\end{align}
as $t\to\infty$ for some $\gamma>0$ and $p\geq 1$. In this setting, the phases of (EXP), (SE), and restricted ($q$-UNIV) have the alternative representation \eqref{eq:repn} for some $\gamma,p>0$, and the pair $(\gamma,p)$ uniquely determines which phase the distribution of $R_k$ belongs to. Figure \ref{fig:phases} shows the phase diagram that illustrates our main results. Note that $\gamma$ is only effective in the phase diagram when $p=1$.

\begin{figure}[htbp]
    \centering
    \includegraphics[width=0.63\textwidth]{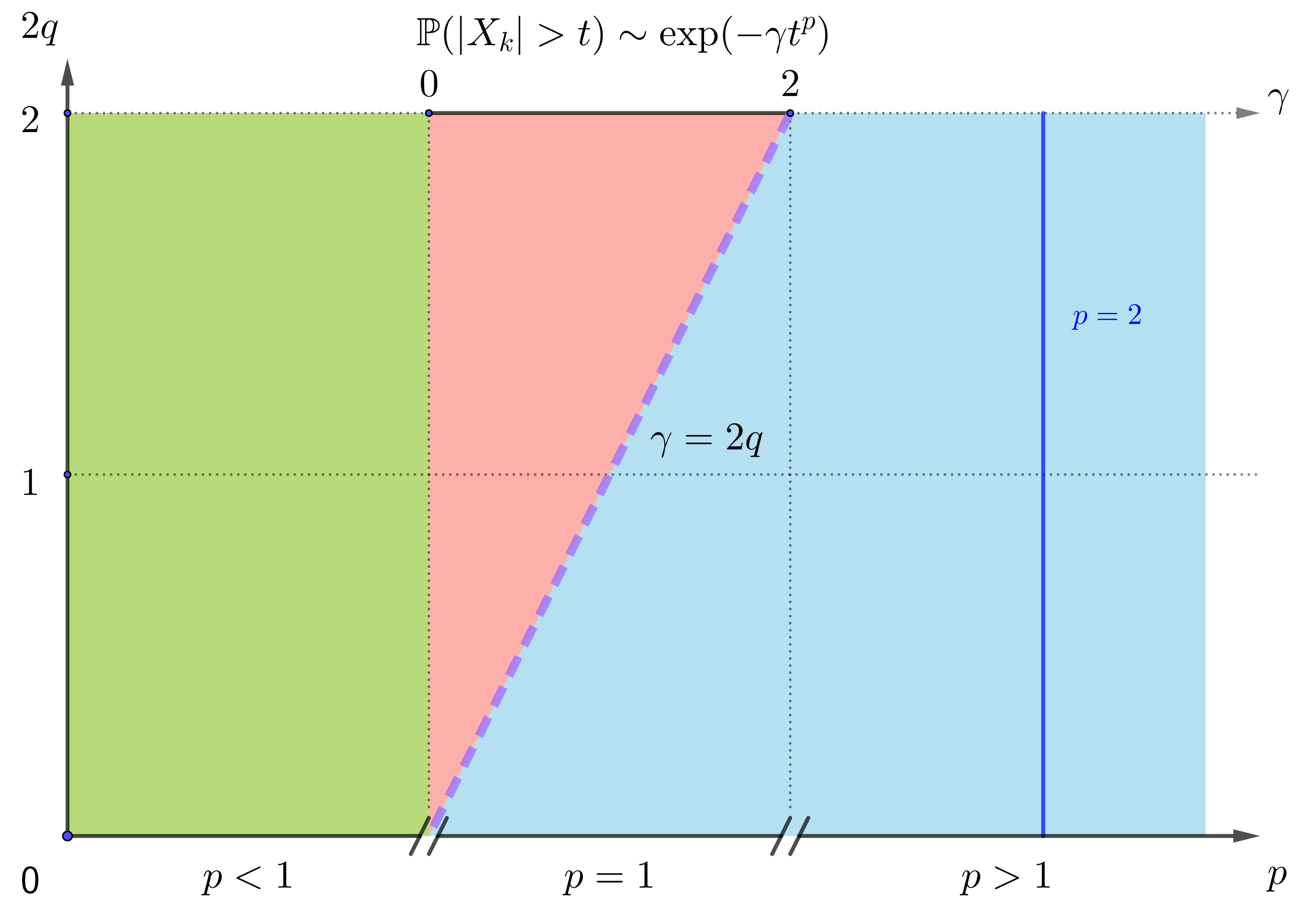}
    \caption{An illustration of the phase transitions of low moments for secular coefficients depending on the law of $|X_1|$ and the moment exponent $(2q)$. The tail satisfying $-\log\P(|X_1|>t)\asymp t^p$ is heavier on the left side and lighter on the right. The light blue regime refers to the universality phase, i.e.~($q$-UNIV) case, which considerably extends the dark blue line that represents the Gaussian case in previous studies. The green and red regimes correspond respectively to the (SE) case and (EXP) case with $\gamma<2q$. The purple dashed line refers to the secondary critical phase with $\P(|X_1|>t)\asymp\exp(-\gamma t)$ where $\gamma=2q$.}
    \label{fig:phases}
\end{figure}

\sloppy As a consequence of Theorem \ref{thm:main}, we find the typical order of magnitude of $A_N$ by considering the tightness of the family $\{A_N/w_N\}_{N\in\N}$ for some deterministic sequence $w_N$, partially generalizing Theorem \ref{thm:NPS}. Our result only requires the existence of some exponential moment; see the proof in Section \ref{subsection-coro} for details. 
\begin{corollary}\label{coro}
  Let $\vartheta=1$. Suppose that there exists $\ee>0$ such that $\E[e^{\ee|X_k|}]<\infty$ for the i.i.d.~input $\{X_k=e^{i\tau_k}R_k\}$. The followings hold: \begin{enumerate}[(i)]
        \item for $w_N=(\log (1+N))^{-1/4}$, $\{A_N/w_N\}_{N\in\N}$ is tight;
        \item for $w_N=o((\log (1+N))^{-1/4})$, $\{A_N/w_N\}_{N\in\N}$ is not tight.
    \end{enumerate}
\end{corollary}
Unfortunately, proving the tightness of $\{(\log (1+N))^{-1/4}/A_N\}_{N\in\N}$ remains a difficult task based on existing techniques, mainly due to the absence of distributional convergence of the total mass of critical non-Gaussian multiplicative chaos. The convergence for critical Gaussian multiplicative chaos was shown in \citep{junnila2017uniqueness}.

We can also obtain the regularity of some critical non-Gaussian $\hmc$ from Theorem \ref{thm:main}. The proof and a remark on irregularity are deferred to Section \ref{sec:reg}. 

\begin{corollary}\label{coro-reg}
 Let $\vartheta=1$.   If there is some $\gamma>2$ such that $\E[e^{\gamma|X_k|}]<\infty$, then the corresponding $\hmc_1$ is in $H^s$ almost surely for any $s<-1/2$. Moreover, if $X_k$ satisfies (EXP) with $\gamma\in(0,2]$, i.e.~$\P(|X_k|\geq u)\asymp e^{-\gamma u}$, then the corresponding $\hmc_1$ is in $H^s$ almost surely for any $s<-1/\gamma$.
\end{corollary}

We also mention that the super-exponential growth of the second moment of secular coefficients in the (SE) case suggests an extremely irregular behavior, in the sense that $\hmc_1$ may not even be a tempered distribution almost surely.

\subsection{Technical challenges}\label{sec:tech-challenge}

In what follows, we attempt to unveil the intricacies of low moments behind the distinct phases, without involving too much technicality. The approach of \citep{soundararajan2022model} in the Gaussian case is to connect $\E[|A_N|^{2q}]$ to low moments of the total mass of a suitably constructed (finite) multiplicative chaos, of the form
\begin{align}
    \E\bigg[\bigg(\int_{-\pi}^{\pi}\Big|\prod_{k=1}^K\exp\Big(\frac{X_kr^ke^{ik\theta}}{\sqrt{k}}\Big)\Big|^2\d\theta \bigg)^q\bigg]\quad\text{ where }\quad K\in\N,~q\in(0,1],\text{ and }r\in[e^{-1/K},e^{1/K}].\label{h}
\end{align}
However, as the tail of $|X_k|$ becomes heavier, the contribution of the random variable $X_1$ in \eqref{h} becomes more prominent and will need to be isolated from other variables $\{X_k\}_{k\geq 2}$. In the cases of (SE) and (EXP) with $\gamma<2q$, our proofs depend on a delicate but direct analysis of such domination effect of $X_1$, without delving into multiplicative chaos. The secondary critical case ((EXP) with $\gamma=2q$) turns out more intricate and requires both the domination effect and the analysis of critical multiplicative chaos. In this scenario, even speculating the correct asymptotics is a non-trivial task.


Our approach is to first condition on $|X_1|$ (or equivalently, $|R_1|$) \emph{before} connecting $\E[|A_N|^{2q}]$ to multiplicative chaos. This reduces our problem into studying the low moments of a \emph{randomly weighted} total mass of a multiplicative chaos, of the form
\begin{align}
    \E\bigg[\bigg(\int_{-\pi}^{\pi}\Big|U(\theta)\times\prod_{k=2}^K\exp\Big(\frac{X_kr^ke^{ik\theta}}{\sqrt{k}}\Big)\Big|^2\d\theta \bigg)^q\bigg]\quad\text{ where }\quad K\in\N,~q\in(0,1],~r\in[e^{-1/K},e^{1/K}],\label{U}
\end{align}
and the (random) weight function $U(\theta)$ resembles a discrete Fourier transform of a random function sufficiently close to a scaled Gaussian density, evaluated at the random frequency $\theta+\tau_1$, where $\tau_1$ is uniformly distributed on $[-\pi,\pi]$.
Intuitively, $U(\theta)$ arises from the remaining randomness of $\tau_1$ after conditioning on $|R_1|$. To study the magnitude of \eqref{U}, we need both a careful analysis of $U(\theta)$ and a \emph{uniform} version of the estimates of the \emph{partial mass} of the multiplicative chaos,
$$ \E\bigg[\bigg(\int_I\Big|\prod_{k=1}^K\exp\Big(\frac{X_kr^ke^{ik\theta}}{\sqrt{k}}\Big)\Big|^2\d\theta \bigg)^q\bigg],$$
where $I$ is an interval in $[-\pi,\pi]$ whose length is of a much smaller order. That is, the integration range in \eqref{h} is restricted to an asymptotically smaller interval, and the difficulty lies in obtaining uniformity in all those intervals of suitable length. A technical highlight while tackling such difficulty is performing generic chaining under a proper change of measure.

A more detailed description of our approach can be found in Section \ref{sec:main ideas}, where we also discuss intuitively how the regularly varying formula in \eqref{eq:gamma=2q result} arises.

\subsection{Related works and outlook}\label{sec:intro-conn-rmt}

In this section, we review a few connections of holomorphic multiplicative chaos to random matrix theory and number theory. Along the way, we also discuss a few further directions of interest that are beyond the scope of this paper. 

\paragraph{Connections to random matrix theory and (complex) Gaussian multiplicative chaos.}

\sloppy $\hmc_\vartheta$ with Gaussian inputs $\{X_k\}_{k\geq 1}$ is closely related to random matrix theory, especially to the characteristic polynomial of C$\beta$E with $\vartheta=2/\beta$. A joint distribution of $N$ points with parameter $\beta>0$ is said to be C$\beta$E if \[\mathrm{C}\beta\mathrm{E}_N(\theta_1,\cdots,\theta_N)\propto \bone_{\{\forall j\in\{1,\dots,N\},\,\theta_j\in[-\pi,\pi)\}}\prod_{1\leq j<k\leq N}|e^{i\theta_j}-e^{i\theta_k}|^\beta,\] and for $\beta=1,2,4$, it is also known as circular orthogonal/unitary/symplectic ensembles (COE/CUE/CSE), respectively. For $\beta=2$ (CUE), it is the distribution of the eigenvalues of a Haar-distributed unitary matrix, and for $\beta\not=2$, \citep{killip2004matrix} constructed explicit matrix models $U_N^{(\beta)}$ whose eigenvalues have law C$\beta$E. Therefore, we can define and study the characteristic polynomial \[\chi_N^{(\beta)}(z)=\det(I-zU_N^{(\beta)})=\sum_{n=0}^N a_n^{(N,\beta)}z^n\] on the unit circle $\bT=\{z\in\C:|z|=1\}$. We may drop the superscripts if $\beta=2$.

Secular coefficients $a_{n}^{(N)}$ of $\chi_N$ were first explicitly studied, to our best knowledge, by Haake, Kus, and Sommers \citep{haake1996secular}, where they delved into various theoretical and numerical properties of secular coefficients for CUE, including an explicit formula of the second moment. Later, Diaconis and Gamburd \citep{diaconis2004random} computed $(2k)$-th absolute moments of secular coefficients $a_n^{(N)}$ of CUE in terms of \textit{magic squares}, which is the number of $N\times N$ square matrices with nonnegative integer entries summing up to $n$ in each row and column. 
For a fixed $n$, the convergence result of $a_n^{(N,\beta)}$ is established via the formula connecting it to the first $n$ power traces $T_k:=\text{Tr}((U_N^{(\beta)})^k)$: \[a_n^{(N,\beta)} = \frac{1}{n!}\det\begin{pmatrix} T_1 & 1 & 0 & \cdots & 0 \\ T_2 & T_1 & 2 & \cdots & 0 \\ \vdots & \vdots &\vdots &\ddots & \vdots \\ T_{n-1} & T_{n-2} & T_{n-3} & \cdots & n-1 \\ T_n & T_{n-1} & T_{n-2} & \cdots & T_1\end{pmatrix},\] along with the convergence result for power traces by Diaconis and Shahshahani \citep{diaconis1994eigenvalues} for $\beta=2$, and Jiang and Matsumoto \citep{jiang2015moments} for a general $\beta>0$ that jointly, \[\{T_k\}_{k=1}^n\implies \sqrt{\frac{2}{\beta}}\big\{\sqrt{k}\mathcal{N}^{\C}_k\big\}_{k=1}^n,~ N\to\infty,\] where $\{N^{\C}_k\}_{k=1}^n$ are i.i.d.~standard complex Gaussian variables. 

A natural question is whether we can find connections between other ensembles of random matrices and non-Gaussian $\hmc_\vartheta$, mirroring Theorem 1.3 of \citep{najnudel2023secular}. Some inspiring examples with non-Gaussian (but weakly dependent) coefficients are the sparse Bernoulli matrix studied by Coste \citep{coste2023sparse} and sums of random permutation matrices by Coste, Lambert, and Zhu \citep{coste2022characteristic}, which can both be considered as random regular digraph models. As a simplified model in \citep{coste2023sparse}, let $B_N$ be an $N\times N$ matrix with i.i.d.~Bernoulli($1/n$) entries, then $\det(I_N-zB_N)$ converges weakly in the open unit disk $\mathbb{D}$ to the so-called \textit{Poissonian holomorphic chaos} $F(z)=\exp(-\sum_{k=1}^{\infty}X_kz^k/\sqrt{k})$ with $X_k=\sum_{\ell\mid k}\ell Y_{\ell}/\sqrt{k}$, where $\{Y_{\ell}\}_{\ell=1}^{\infty}$ is a family of independent Poisson variables with $\E[Y_\ell]=1/\ell$. Moreover, a similar convergence in law result (and towards a similar Poissonian chaos) was established in \citep{coste2022characteristic} if $B_n$ is replaced by the average of a fixed number of i.i.d.~random permutation matrices; see Theorem 2.2 therein for details.
\begin{question}
    Find other ensembles of random matrices $V_N^{(\vartheta)}$ such that its characteristic polynomial converges (in some sense) to $\exp(\sqrt{\vartheta}\sum_{k\geq 1}X_k z^k/\sqrt{k})$ uniformly in $|z|\leq r$, for any $r\in(0,1)$.
\end{question}

Yet another connection exists between secular coefficients of random matrices (or $\hmc$) and the total mass of the \textit{Gaussian multiplicative chaos} ($\gmc$).
A $\gmc_\vartheta$ of parameter $\vartheta\in(0,1]$ can be defined as the limiting random measure (see e.g.~Appendix B of \citep{najnudel2023secular}) \[\gmc_{\vartheta}(\d\theta):=\lim_{r\to 1}(1-r^2)^{\vartheta}(-\log(1-r^2))^{\frac{1}{2}\bone_{\vartheta=1}}|e^{\sqrt{\vartheta}\sum_{k=1}^\infty\frac{\mathcal{N}^{\C}_k}{\sqrt{k}}(re^{i\theta})^k}|^2\d\theta,\] where the convergence holds in $L^q$ for any $q\in(0,1)$. Its \textit{total mass} $\mathcal{M}_\vartheta$ is then defined as $\mathcal{M}_\vartheta:=\frac{1}{2\pi}\int_0^{2\pi}\gmc_{\vartheta}(\d\theta)$ in the in-probability, weak-$\ast$ convergence sense.
Recently, Najnudel, Paquette, and Simm \citep{najnudel2023secular} and Najnudel, Paquette, Simm, and Vu \citep{najnudel2025fourier} established the distributional convergence of $a_n^{(N,\beta)}/\sqrt{\E[(a_n^{(N,\beta)})^2]}$ as $n, N\to\infty$ jointly for $\beta>2$ ($0<\vartheta<1$), and expressed the limiting distribution as $\sqrt{\mathcal{{M}_\vartheta}}$ times an independent standard complex normal. 
See also \citep{garban2023harmonic} for the distributional convergence of Fourier coefficients of $\gmc_{\vartheta}$ on the unit circle for $0<\vartheta<1/\sqrt{2}$.
Moreover, \citep{najnudel2023secular} also established tightness for a general $\beta>0$. As an example, for $\beta=2$ they showed that $\{(\log n)^{1/4}a_n^{(N)}: N\geq 2n\}$ and $\{(\log n)^{-1/4}/a_n^{(N)}:N\geq N_0(n)\}$ are both tight (for some $N_0(n)$ growing faster than $n\sqrt{\log n}\,(\log\log n)$). Although normalization and limiting behavior are not yet clear in the critical case ($\vartheta=1$), the question of universality and phase transition could be asked, reminiscent of our Theorem \ref{thm:main} and \ref{thm:main2}.
\begin{question}
    When $\vartheta=1$, does $A_N(\log N)^{1/4}$ converge in distribution to the same limiting random variable for any $X_k$ as in Corollary \ref{coro}? If the tail of $|X_k|$ becomes heavier and reaches the (SE) phase, is there a phase transition in the limiting behaviors as well?
\end{question}

The convergence result of secular coefficients stated above also indicates an impressive fact, that a single $A_N$ (when $N$ is large) already contains some information of the total mass $\mathcal{M}_\vartheta$. In fact, \citep{najnudel2023secular} used this connection and existing results on the convergence of the critical $\gmc$ to prove the irregularity of $\hmc_1$ and the sharp tightness of $\{(\log(1+N))^{-1/4}/A_N\}$. To study these properties in a non-Gaussian setting, one has to understand the critical non-Gaussian multiplicative chaos.
\begin{question}
    Prove the convergence in $L^q$ for all $q\in(0,1)$ for\[\lim_{r\to 1}(1-r^2)(-\log(1-r^2))^{\frac{1}{2}}|e^{\sum_{k=1}^\infty\frac{X_k}{\sqrt{k}}(re^{i\theta})^k}|^2\d\theta\] when $\{X_k\}_{k\in\N}$ are i.i.d.~non-Gaussian, and apply it to study the irregularity of $\hmc$ and sharp tightness of $A_N$ in more general settings.
\end{question}

Another promising future direction is to explore the connection between complex Gaussian multiplicative chaos (CGMC) and $\hmc$. Roughly speaking, the complex Gaussian field defining CGMC has independent real and imaginary parts, while $\hmc$ has uncorrelated ones, which might suggest some commonalities in between. We refer interested readers to Section 1.6 of \citep{najnudel2023secular} for a more detailed discussion on this possible connection, and to \citep{lacoin2022universality,lacoin2015complex} for references on CGMC.

Finally, we summarize some related results on characteristic polynomials of random matrices. The celebrated work of Hughes, Keating, and O'Connell \citep{hughes2001characteristic} introduced the log-correlated Gaussian field $G^{\C}(z)=\sum_{k=1}^\infty\mathcal{N}_k^{\C}z^k/\sqrt{k}$ as the limiting random distribution of the log-characteristic polynomial of CUE. For general surveys on log-correlated Gaussian fields, see \citep{duplantier2017log}. Meanwhile, convergence results towards $\gmc$ were established in \citep{ chhaibi2019circle,nikula2020multiplicative,webb2015characteristic} for characteristic polynomial of C$\beta$E on the unit circle. For general surveys on $\gmc$, see \citep{powell2020critical, rhodes2014gaussian}.
We also point interested readers to the survey by Bailey and Keating \citep{bailey2022maxima} and the most recent result by Paquette and Zeitouni \citep{paquette2022extremal} and references therein for studies towards a conjecture in \citep{fyodorov2012freezing, fyodorov2014freezing} by Fyodorov, Hiary, and Keating on the maximum of the characteristic polynomial of CUE.


\paragraph{Connections to number theory.} 
A \emph{random multiplicative function} is a completely multiplicative function $f:\N\to\C$, i.e.~$f(mn)=f(m)f(n)$ for $m,n\in\N$, where $f(p)$ are i.i.d.~random variables for primes $p$.
An important application is modeling arithmetic functions, such as Dirichlet characters
(Steinhaus case, with $f(p)$ uniformly distributed on the unit circle $\mathbb{T}$) and the M\"{o}bius function (Rademacher case, with $f(p)$ uniformly distributed on $\{\pm 1\}$ and supported on square-free numbers). 
The recent celebrated work of Harper \citep{harper2020moments} investigated the better-than-square-root cancellation phenomenon of random multiplicative functions through computing low moments of the partial sums, thus resolving Helson's conjecture \citep{helson2010hankel}. More precisely, \citep{harper2020moments} showed that 
$$\E\bigg[\Big|\sum_{n\leq x}f(n)\Big|\bigg]\asymp\frac{\sqrt{x}}{(\log\log x)^{1/4}},$$
while $\E[|\sum_{n\leq x}f(n)|^2]\asymp x$. 
Harper's approach elegantly connects random multiplicative functions and critical multiplicative chaos, which we summarize in Section \ref{sec:main ideas}. 
Similar ideas are also exploited in the studies of deterministic multiplicative functions of interest in number theory, such as better-than-square-root cancellation for typical non-principal Dirichlet characters \citep{harper2023typical}.

The recent work of Soundararajan and Zaman \citep{soundararajan2022model} proposed that the secular coefficients $A_N$ of Gaussian $\hmc_1$ constitute a model that describes the mathematical structure of (Steinhaus) random multiplicative functions. As commented in \citep{soundararajan2022model}, the same model also serves as the function field counterpart of Harper's result \citep{harper2020moments}, as follows. 
Let $\M_n$ denote the set of monic polynomials of degree $n$ over the ring $\mathbb{F}_q[t]$ where $q$ is a prime power and $\mathbb{F}_q$ is a finite field with $q$ elements. Consider a random multiplicative function $f$ on $\mathbb{F}_q[t]$ defined analogously, with irreducible monic polynomials as ``primes". It follows that if $\bar{A}_n=q^{-n/2}\sum_{F\in\M_n}f(F)$, then 
\begin{align}
    \sum_{n=0}^\infty \bar{A}_nz^n=\exp\bigg(\sum_{k=1}^\infty \frac{\bar{X}_k}{\sqrt{k}}z^k\bigg)\quad\text{ where }\quad \bar{X}_k=\frac{\sqrt{k}}{q^{k/2}}\sum_{\substack{P \text{ irred.}\\ \mathrm{deg}(P)\mid k\\ r=k/\mathrm{deg}(P)}}\frac{f(P)^r}{r}.\label{eq:connection?}
\end{align}
Therefore, the partial sums $\bar{A}_n$ (in the limit case of $q\to\infty$) mirror the secular coefficients $A_n$ defined in \eqref{eq:basic eq}. We refer to the introduction of \citep{soundararajan2022model} for related discussions around this analogy. 


We speculate that secular coefficients arising from non-Gaussian chaos may connect to certain structured sparse partial sums of random multiplicative functions with heavy-tailed inputs $\{f(P)\}$.
For instance, by considering partial sums of $f$ on smooth polynomials (that factor into primes of small degrees, say $\leq d$), the last sum of \eqref{eq:connection?} consists of a uniformly bounded number of summands (consisting of irreducible polynomials of degrees $\leq d$). In this case, heavy-tailed inputs $\{f(P)\}$ may dominate the central limit effect, and weak dependence within $\{\bar{X}_k\}$ is expected as the power $r$ increases with $k$. 
We refer to \citep{soundararajan2022central,xu2023better} for studies on sparse partial sums of random multiplicative functions.
\begin{question}Find a formal analogous model of some random multiplicative function for secular coefficients arising from non-Gaussian chaos.
\end{question}

The secular coefficients $\{A_N\}_{N\geq 0}$ (with Gaussian inputs $\{X_k\}_{k\geq 1}$) capture other aspects of random multiplicative functions as well, such as almost sure fluctuations. For instance, \citep{caich2023almost2} established the almost sure upper bound $|\sum_{n\leq x}f(n)|\ll \sqrt{x}(\log\log x)^{3/4+\ee}$ for random multiplicative functions $f$, which mirrors the result $|A_N|\ll (\log  N)^{3/4+\ee}$ in \citep{caich2023almost} for secular coefficients.  
In addition, for any function $V(x)\to\infty$, \citep{harper2023almost} proved the almost sure existence of large values $x$ satisfying $|\sum_{n\leq x}f(n)|\geq \sqrt{x}(\log\log x)^{1/4}/V(x)$. The parallel for secular coefficients was established by \citep{gerspach2022almost}:  there exist almost surely large values of $N$ such that $|A_N|\geq (\log N)^{1/4}/V(N)$. We leave the investigation of the almost sure fluctuations of secular coefficients with non-Gaussian inputs to future research. {Let us also mention the recent works  \citep{gorodetsky2024martingale,hardy2025distribution} on limit theorems of partial sums of random multiplicative functions, of which we are not aware of an analogue for secular coefficients.}
\begin{question}
    Establish almost sure upper and lower bounds for secular coefficients arising from non-Gaussian chaos. Characterize a phase transition as the tail of $|X_1|$ becomes heavier.
\end{question}

Let us also mention the work of \citep{aggarwal2022conjectural} that numerically computes the secular coefficients in polynomial time, supporting conjectures on finer asymptotics for their low moments. 
Naturally, one would also ask if the asymptotics in the ($q$-UNIV) and (EXP) phases can be improved to $1+o(1)$ asymptotics. 
\begin{question}
    Find precise asymptotics of \eqref{eq:asymp LT}, \eqref{eq:asymp SE},  and \eqref{eq:gamma=2q result}.
\end{question}
Moreover, \citep{aggarwal2022conjectural} also conjectured a similar behavior of low moments of secular coefficients for \emph{real} standard Gaussian or Rademacher inputs $\{X_k\}_{k\geq 1}$. While we do not directly resolve their conjecture, we shall illustrate in Remark \ref{rem:p} below that the asymptotics \eqref{eq:asymp HT} and \eqref{eq:asymp SE} hold also in the real case, with essentially the same proof. However, this does not apply to \eqref{eq:gamma=2q result}. 
\begin{question}
    Do \eqref{eq:asymp LT} and \eqref{eq:gamma=2q result} of Theorem \ref{thm:main2} hold if one replaces the complex inputs $\{X_k\}_{k\geq 1}$ by their real part $\{R_k\}_{k\geq 1}$?
\end{question}
Finally, in addition to random multiplicative functions, critical multiplicative chaos also connects to 
the distribution of values of the Riemann zeta function on the critical line. According to the FHK conjecture, the local maxima of $\log|\zeta(1/2+it)|$ deviate from what one would predict from Selberg's central limit theorem, due to the log-correlated structure of the zeta values. We refer to \citep{arguin2020fyodorov,arguin2023fyodorov,bailey2022maxima,soundararajan2022distribution} and the references therein for the interplay between the Riemann zeta function, multiplicative chaos, and log-correlated fields.

\paragraph{Organization.}
The rest of this paper is organized as follows. Section \ref{sec:proof strategy} is devoted to the intuition of the phase transition regime and the main ideas of the proofs. In Sections \ref{sec:supercritical}--\ref{sec:critical}, we consider respectively the ($q$-UNIV), (SE), and (EXP) phases. 
Appendix \ref{sec:ubdef} is devoted to some technical computations regarding Laplace functionals of $\{X_k\}_{k\geq 1}$ under distinct probability measures. Appendix \ref{sec:deferred proofs} collects some deterministic calculations.

\paragraph{Notation.}\label{para:notations} For quantities or functions $A,B$, We use Vinogradov's symbol $A\ll B$ (or $A=O(B)$) to denote $|A|\le CB$ with some constant $C>0$ that depends only on the distribution of $R_k$ (equivalently, $X_k$) and the moment exponent $q$. Write $A\asymp B$ if $A\ll B\ll A$. If $A(N)/B(N)\to 0$ we write $A(N)=o(B(N))$. We will denote by $L>0$ a universal constant that \emph{may not be the same on each occurrence}, where the same applies for $C>0$. Vectors are typically denoted by bold symbols in this paper. 
Denote by $\Re z$ the real part of a complex number $z$. When dealing with events (or expectations) involving $\{(R_k,\tau_k)\}_{k\ge 1}$, we will use $\P$ (or $\E$) to denote the original probability measure. We will also use $\P$ to denote the probability measure of any Gaussian random variable (or vector) with a specified mean and variance (or covariance matrix). 

\section{Intuition behind phase transitions and proof strategy}\label{sec:proof strategy}
Fix $\vartheta=1$.
To understand the mechanism behind the phase transition phenomenon for low moments of secular coefficients $A_N$, let us first rewrite it in a more tractable form in terms of partitions. By a \emph{partition} $\lambda$ we mean a non-increasing sequence of integers (parts) $\lambda_1\ge\lambda_2\ge\cdots$ with $\lambda_n=0$ from some $n$ onward. For a partition $\lambda$ and $k\in\N$, we denote by $m_k(\lambda)$ the number of parts in $\lambda$ that are equal to $k$, and $|\lambda|$ the sum of the parts in $\lambda$. For example, the all-one partition $\lambda^*:=(1,\cdots,1)$ has $|\lambda^*|=n$ and $m_k(\lambda^*)=n\bone_{\{k=1\}}$. Let $\cP_N=\{\lambda:|\lambda|=N\}$. 
For a partition $\lambda\in\cP_N$, define
\begin{align}
    a(\lambda):=\prod_{k\geq 1} \left(\frac{X_k}{\sqrt{k}}\right)^{m_k}\frac{1}{m_k!}.\label{eq:al expression}
\end{align}
It follows from \eqref{eq:basic eq} (with $\vartheta=1$) and \eqref{eq:al expression} that 
\[\exp\Bigg(\sum_{k=1}^\infty\frac{X_k}{\sqrt{k}}z^k\Bigg) = \sum_{\lambda\in\cP_N} a(\lambda)z^{|\lambda|},\]
which can be seen by applying Taylor expansion to each $\exp(X_kz^k/\sqrt{k})$. 
In particular, we arrive at the tractable form as a sum of monomials in $\{X_k\}_{k\geq 1}$:
\begin{align}
    A_N=\sum_{\lambda\in\cP_N}a(\lambda).\label{eq:AN expression}
\end{align}
We note the orthogonality relation that $\E[a(\lambda)\overline{a(\lambda')}]=0$ for $\lambda\neq\lambda'$.

\subsection{A phase transition of the domination regime in the monomial decomposition}
In different phases, the quantity $\E[|A_N|^{2q}]$ is dominated by different interactions among parts in the sum over partitions of $N$. This can be intuitively seen by comparing the magnitudes of 
\begin{align}
    \E[|a(\lambda)|^{2q}]=\prod_{k\geq 1} \frac{\E[|X_k|^{2qm_k}]}{k^{qm_k}(m_k!)^{2q}},\quad\lambda\in\cP_N,\label{eq:alambda moments}
\end{align}
using independence, moment asymptotics for $|X_k|$, and  Stirling's approximation (see Chapter 3 of \citep{artin2015gamma})
\begin{align}
     1<(2\pi)^{-1/2}x^{1/2-x}e^x\Gamma(x)<e^{1/(12x)},~x\geq 1.\label{eq:gamma}
 \end{align}
 To be more specific, in the (SE) case, by a change of variable, \begin{align}
    \begin{split}
        \E[|X_k|^{2qm_k}]&= \int_0^\infty e^{-(u^{1/(2qm_k)}/c_p)^p} \d u\\
    &=c_p^{2qm_k}\frac{2qm_k}{p}\int_0^\infty e^{-v}v^{\frac{2qm_k}{p}-1}\d v= c_p^{2qm_k}\frac{2qm_k}{p}\Gamma\Big(\frac{2qm_k}{p}\Big).
    \end{split}\label{u}
\end{align} A finer computation by estimating the Gamma function in \eqref{u} using \eqref{eq:gamma} reveals that $\E[|a(\lambda)|^{2q}]$ grows fastest in $N$ when $\lambda=\lambda^*=(1,\cdots,1)$ (i.e.~$m_1(\lambda^*)=N$), and is significantly larger than the $(2q)$-th moments of $a(\lambda')$ at any other $\lambda'\in\cP_N$ for $N$ large. Therefore, the moments of $A_N$ almost only result from $a(\lambda^*)$, which contains a single random variable $X_1$.
The same arguments apply to even heavier tails than (SE).

On the other hand, as a prototypical example in the ($q$-UNIV) case, we consider the setting where $X_k$ is a standard complex Gaussian variable. Computing the absolute moments of Gaussian variables (see e.g.~\citep{winkelbauer2012moments}) gives 
\begin{align*}
    \E[|X_k|^{2qm_k}]\asymp 2^{qm_k}\Gamma\Big(qm_k+\frac 12\Big)=o\big(k^{qm_k}(m_k!)^{2q}\big),
\end{align*}
indicating that every $\E[|a(\lambda)|^{2q}]$ is of a vanishing order and therefore the main contribution to the moments of $A_N$ does not arise from a single partition. Instead, \citep{soundararajan2022model} found that, as inspired by Harper's remarkable paper on low moments of partial sum of random multiplicative functions \citep{harper2020moments}, the main contribution to the moments of $A_N$ (in the complex Gaussian case) comes from those partitions $\lambda$ with a large part, i.e.~$\lambda_1\geq N/(\log N)^C$ for some large constant $C>0$, indicating an intricate interplay among all $X_k$. 

As the exponent $p$ of $\P(|X_k|>t)\asymp \exp(-ct^p)$ decreases from above $1$ (e.g.~Gaussian) to below $1$ (stretched exponential), the contribution from $X_1$ becomes more prominent and experiences a phase transition at $p=1$. At the critical phase of exponential distributions where $\P(|X_k|>t)\asymp \exp(-\gamma t)$, more elaborate dependence on variables $X_k$ emerges. Fix a moment exponent $q\in(0,1]$. If the exponent $\gamma>2q$, the tail of $X_k$ decays fast enough to suppress the growth of each single $\E[|a(\lambda)|^{2q}]$. As will be shown in Section \ref{sec:supercritical}, the mechanism of Gaussian variables (i.e.~($q$-UNIV) case) remains true. If $\gamma<2q$, on the other hand, the tail decays slowly enough to \emph{partially} restore the dominance of $a(\lambda^*)$ among all partitions. Now the dominant parts comprise not only the all-one partition $\lambda^*$ but also the partitions $\lambda$ with almost all ones, i.e.~$m_1(\lambda)>N-C_*$ for some large constant $C_*>0$. This hints at a second phase transition in the behavior of $A_N$ and the interplay structure among the i.i.d.~inputs $\{X_k\}_{k\geq 1}$ at $\gamma=2q$.

At the secondary criticality where $\P(|X_k|>t)\asymp\exp(-2qt)$, a blend of the two aforementioned scenarios influences the $2q$-th moment of $A_N$. The dominating terms now appear randomly and depend on the value of $|R_1|=|X_1|$. Roughly speaking, on the event that $|R_1|$ is close to a fixed $x_1\in[0,N]$, the primary contribution to $\E[|A_N|^{2q}]$ stems from the $a(\lambda)$ with $m_1(\lambda)$ close to $ x_1$. After conditioning on $R_1$ and fixing $m_1(\lambda)$, the rest terms can be described in terms of partitions without ones and will behave similarly as in the ($q$-UNIV) case.

\subsection{Main ideas of the proofs}
\label{sec:main ideas}
In the following, we delve into further details of the different cases described by our main results, and illustrate the main ideas of the proofs.

\paragraph{(SE) phase.} The orders of the moments $\E[|X_k|^{2qm_k}]$
are given by \eqref{u}. 
Using \eqref{eq:gamma} and that $0<p<1$, one observes that
the quantity \eqref{eq:alambda moments} with $\lambda=\lambda^*=(1,\dots,1)$ (i.e., $m_1(\lambda^*)=N$) is of an order larger than that with any other $\lambda\in\cP_N$ as $N\to\infty$. This can be directly quantified by Minkowski's inequality for $q>1/2$ and concavity for $q\leq 1/2$. 

\paragraph{($q$-UNIV) phase, universality.} 
We adapt the proof of \citep{soundararajan2022model} to derive the proposed universality result of \eqref{eq:asymp LT}, utilizing also the robustness of the multiplicative chaos approach in computing low moments of random multiplicative functions in \citep{harper2020moments}, together with several new observations and technical improvements. The first observation is that low moments of $A_N$ concentrate around partitions such that not all parts are small, similar to the complex Gaussian model. 
Using our assumption $\E[e^{\gamma|R_k|}]<\infty$ and Markov's inequality, we have $\E[|X_k|^{2qm_k}]\leq {\Gamma(2qm_k+1)}{\gamma^{-2qm_k}}.$ 
 Inserting in \eqref{eq:alambda moments}, it follows from \eqref{eq:gamma} that
\begin{align}
\E[|a(\lambda)|^{2q}]\leq \prod_{k\geq 1}\frac{C(2q/\gamma)^{2qm_k}m_k^{1/2-q}}{k^{qm_k}}.\label{eq:moment bound 1}
\end{align}
Since $\gamma>2q$, we expect that the contribution from $a(\lambda)$ is small unless $m_k$ is in general very small (for instance, $\lambda=(N)$).
Indeed, Proposition \ref{SZ-prop3.1} below states that it suffices to consider only partitions $\lambda$ with $\lambda_1\geq \sqrt{N}/C(q)$ for some large constant $C(q)$. 

Next, recalling that $A_N$ can be viewed as a ``Fourier coefficient" of the holomorphic multiplicative chaos, we may apply Parseval's identity to connect the sum over $\lambda$ to the moments of the \emph{total mass} of a (finite) multiplicative chaos, which is of the form
\begin{align}
    \E\left[\left(\int_{-\pi}^{\pi}|F_{K}(re^{i\theta})|^2\d\theta \right)^q\right]\quad\text{ where }\quad K\in\N,~F_K(z):=\exp\Big(\sum_{k=1}^K\frac{X_k}{\sqrt{k}}z^k\Big),\text{ and }r\in[e^{-1/K},e^{1/K}].\label{eq:chaos}
\end{align}
However, a technical difficulty arises: since we only assumed $(2q+\ee)$-th exponential moment exists, certain Laplace functionals of $X_k/\sqrt{k}$ may not be well-defined for a small $k$. Therefore, it is necessary to eliminate the dependence on $X_k$ for small $k$. Fortunately, this is possible by the rotational invariance of $X_k$ and using $\gamma>2q$.
Indeed, we show in Section \ref{sec:LT reduction} that it suffices to study the \emph{truncated} secular coefficients 
\begin{align}
    A_{N,M_*}:=\sum_{\substack{\lambda\in\cP_N\\ m_1=\dots=m_{M_*-1}=0}}a(\lambda)=\sum_{\substack{\lambda\in\cP_N\\ m_1=\dots=m_{M_*-1}=0}}\prod_{k\geq M_*}\left(\frac{X_k}{\sqrt{k}}\right)^{m_k}\frac{1}{m_k!}\label{eq:ANM}
\end{align}
of the \emph{truncated} chaos
\begin{align}
    F_{K,M_*}(z):=\exp\Big(\sum_{k=M_*}^K\frac{X_k}{\sqrt{k}}z^k\Big).\label{eq:FK}
\end{align}
Here, $M_*$ is a large constant that depends only on $q$ and the law of $X_1$.\footnote{One can show that if $X_1$ is compactly supported, $M_*=1$ is feasible.} 
After truncation at $M_*$, we can apply Parseval's identity to convert the moments of $A_{N,M_*}$ to that of the total mass of $|F_{N,M_*}|^2$ over the (scaled) unit circle; see Propositions \ref{SZ-prop3.1} and \ref{SZ-Prop-8.1} below. 

The non-trivial part of \eqref{eq:asymp LT} is that the low moments of $A_N$ (and hence of the total mass of $|F_{N,M_*}|^2$ by the aforementioned connection) are asymptotically smaller than those predicted by the second moment (a.k.a.~better-than-square-root cancellation). The fundamental reason is that on certain unlikely events, some values of $|F_{N,M_*}|^2$ become exceptionally large and dominate the second moment. 
To intuitively see this, let us consider the following random field 
\begin{align}
    G(z):=\log F_{N,M_*}(z)=\sum_{k=M_*}^N \frac{X_k}{\sqrt{k}}z^k,~ z\in\mathbb{T},\label{eq:Gz}
\end{align}
where $\mathbb{T}$ denotes the unit circle, together with a \textit{branching random walk} (BRW) analogue which also guides the intuition in (EXP) $\gamma=2q$ phase.\footnote{A BRW process with $d$ offsprings, $L$ generations, and standard Gaussian increments is a Gaussian process on a $d$-ary tree of depth $L$, assigning i.i.d.~standard Gaussian variables to each edge of the tree. The value assigned to each leaf is then the sum of independent Gaussian variables along the shortest path from the root to that leaf.}
For more information on the analogy between BRW and other log-correlated fields, see Arguin's survey \citep{arguin2016extrema} and references therein.

Since the variance of each summand in \eqref{eq:Gz} is of size $\asymp{1}/{k}$, it is natural to consider the partial sum over an exponentially growing interval, i.e.~$\sum_{k\in[e^{n-1},e^n)}X_kz^k/\sqrt{k}$ and rewrite the field as \begin{align}
    G(z)\approx \sum_{n=\floor{\log M_*}}^{\floor{\log N}} Y_n(z):=\sum_{n=\floor{\log M_*}}^{\floor{\log N}}\Bigg(\sum_{k=e^{n-1}}^{e^n-1}\frac{X_k}{\sqrt{k}}z^k\Bigg), ~z\in\mathbb{T}.\label{eq:Yndef}
\end{align} 
The first observation is that for each $z\in\mathbb{T}$, the distribution of each \textit{increment} $Y_n(z)$ approaches a standard complex Gaussian as $n$ increases. Next, if we consider two points $z,z'\in\mathbb{T}$ with $|z-z'|$ close to $ e^{-m}$, then for $n\leq m$, one expects $Y_n(z)$ to strongly correlate with $Y_n(z')$. On the other hand, for $n$ much larger than $m$, the increments $Y_n(z)$ and $Y_n(z')$ should be sufficiently decorrelated and are asymptotically independent as $n\to\infty$. From these observations, the random field $G(z)$ on $\mathbb{T}$ can be viewed as a BRW with ``$e$ offsprings" and $\floor{\log N}$ generations, with standard complex Gaussian increments $Y_n(z)$ at each step. The latest common ancestor of $z,z'\in\mathbb{T}$ is of generation $-\log|z-z'|$. In other words, the values $G(z)$ and $G(z')$ share a common part of 
\begin{align}
    \sum_{n=\floor{\log M_*}}^{-\log|z-z'|}Y_n(z)\approx\sum_{n=\floor{\log M_*}}^{-\log|z-z'|}Y_n(z')\label{eq:equal    }
\end{align}and the rest are approximately independent. In particular, if $|z-z'|=o(1/N)$, $G(z)$ is almost indistinguishable from $G(z')$.

Consequently, an exceptionally large value of \eqref{eq:equal    } will affect a spectrum of $z$ of length $\asymp e^{-m}$, leading to an exceptionally large second moment. To circumvent such an issue, the works of \citep{harper2020moments} and \citep{soundararajan2022model} borrowed a barrier argument from BRW (see e.g.~\citep{shi2015branching}), using a ballot event to discard the (unlikely) large values of \eqref{eq:equal    }. Their first step is to view the term $|F_{N,M_*}|^2$
as a Girsanov-type change of measure, under which $Y_n$ has a non-zero expectation $\mu_n$. 
Next, they set up  the barrier event
\[\mathcal{G}(A;N):=\bigg\{\forall \log M_*<n\leq \log N,~\forall z\in\mathbb{T},~\sum_{j=\floor{\log M_*}}^{n}(Y_j(z)-\mu_j)\leq A\bigg\}\] and show that $\mathcal{G}(A;N)^c$ holds with probability exponentially decaying in $A$. On the event $\mathcal{G}(A;N)$, one writes 
\[\E[\bone_{\mathcal{G}(A;N)}|F_{N,M_*}|^2]=\E[|F_{N,M_*}|^{2}]\,\mathbb{Q}(\mathcal{G}(A;N)),\] where $\d\mathbb{Q}/\d\P\asymp |F_{N,M_*}|^2$  and the ballot-type probability $\mathbb{Q}(\mathcal{G}(A;N))$ accounts for the $(\log N)^{-q/2}$ correction term arising in \eqref{eq:asymp HT}. 

The technical difficulty in carrying over the above arguments to the non-Gaussian case is twofold. First, we need to replace the precise computation of $\E[|F_{N,M_*}|^{2}]$ in the Gaussian case with a slightly more involved asymptotic computation using Taylor's expansion. This will be detailed in Appendix \ref{sec:ubdef}. Second, it is harder to quantify the dependence structure within $G(z)$ and $G(z')$ for $z,z'\in\mathbb{T}$ as discussed near \eqref{eq:equal    }. To overcome this issue, we apply the slicing argument of \citep{harper2020moments} and a two-dimensional Berry-Esseen estimate to analyze quantitatively the dependence structure via normal approximation. 
However, Harper's original proof relies on the double-exponential growth of the number of summands defining $Y_n$, which leads to handy convergence results. This approach unfortunately does not suffice as our $Y_n$ only consists of an exponential number of summands. 
Instead, we apply yet another change of measure to re-center the vector $(Y_n(z),Y_n(z'))$ for each pair of $(z,z')$, which guarantees a good enough quantitative approximation.

\paragraph{(EXP) phase with $\gamma<2q$.} 
We compute directly that
\begin{align}
    \E[|X_k|^{2qm_k}]=2\int_{c_\gamma}^\infty u^{2qm_k} \gamma e^{-\gamma(u-c_\gamma)}\d u\asymp\gamma^{-2qm_k}\Gamma(2qm_k+1),\label{eq:moments of SE}
\end{align}
which amounts to (using \eqref{eq:gamma})
$$\E[|a(\lambda)|^{2q}]\asymp \prod_{k\geq 1}\frac{\Gamma(2qm_k+1)}{\gamma^{2qm_k}k^{qm_k}(m_k!)^{2q}}\asymp \prod_{k\geq 1} m_k^{1/2-q}\Big(\frac{2q}{\gamma\sqrt{k}}\Big)^{2qm_k}.$$
Since $\gamma<2q$, large values of $m_k$ are favorable. In other words,  
partitions $\lambda\in\cP_N$ with $m_1(\lambda)$ close to $ N$ dominate. On the other hand, for those $\lambda$, the magnitudes of $\E[|a(\lambda)|^{2q}]$ are comparable (contrary to the case of (SE)), and hence a na\"{i}ve Minkowski's inequality argument (and concavity when $q<1/2$) similar to the (SE) phase suffices for the upper bound but does not conclude the lower bound.

Observe that for the all-one partition $\lambda^*$, $a(\lambda^*)=X_1^N/N!$. It follows that 
\begin{align}
    \E[|a(\lambda^*)|^{2q}]=\E\Bigg[\Big|\frac{X_1^N}{N!}\Big|^{2q}\Bigg]\asymp \frac{1}{(N!)^{2q}}\int_0^\infty x^{2qN}e^{-\gamma x}\d x.\label{eq:gamma argument}
\end{align}
The contribution to the latter integral stems mainly from the range where $x$ is close to $ 2qN/\gamma$. This motivates us to restrict the expectation $\E[|A_N|^{2q}]$ to the event $|R_1|=|X_1|\approx 2qN/\gamma$. By restricting to such an event, we gain better control of the dominance of $\lambda\in\cP_N$ with $m_1(\lambda)$ close to $ N$. 

We first use Minkowski's inequality (and concavity when $q<1/2$) to exclude the partitions $\lambda\in\cP_N$ with $m_1(\lambda)\leq N-C_*$ for some large constant $C_*>0$. For $\lambda$ such that $m_1(\lambda)> N-C_*$, we may extract a common large power of $|X_1|$, that of $|X_1|^{N-C_*}$. Let us restrict to the event $|X_1|=2qN/\gamma+O(\sqrt{N})$. For the rest power of $|X_1|^{m_1(\lambda)-(N-C_*)}$, we may substitute $|X_1|$ with $2qN/\gamma$ with a negligible error term. This leads to tractable control on  
\begin{align}
    \E\Bigg[\Big||X_1|^{N-C_*}\sum_{\substack{\lambda\in\cP_N\\ m_1(\lambda)\geq N-C_*}}{\widetilde{a}(\lambda)}\Big|^{2q}\bone_{\{|X_1|\approx 2qN/\gamma\}}\Bigg],\label{eq:control1}
\end{align}
where 
$$\widetilde{a}(\lambda):=\frac{e^{i\tau_1m_1}|2qN/\gamma|^{m_1-(N-C_*)}}{m_1!}\prod_{k\geq 2} \left(\frac{X_k}{\sqrt{k}}\right)^{m_k}\frac{1}{m_k!}.$$
The finite sum in \eqref{eq:control1}, which is independent of $R_1$, can be separated from the expectation and approximated in $L^2$ (for $C_*$ large) by an infinite series that one can offer a lower bound in probability. 

The above arguments can also be adapted to the case of real inputs $\{X_k\}_{k\geq 1}$; see Remark \ref{rem:p}.

\paragraph{(EXP) phase with $\gamma=2q$.}  
A similar computation as in the case $\gamma<2q$ using \eqref{eq:moments of SE} yields
$$\E[|a(\lambda)|^{2q}]\asymp \prod_{k\geq 1}\frac{\Gamma(2qm_k+1)}{\gamma^{2qm_k}k^{qm_k}(m_k!)^{2q}}\asymp \prod_{k\geq 1} m_k^{1/2-q}k^{-qm_k}.$$
Compared to the above phases, we observe a new phenomenon here: the main contribution in \eqref{eq:AN expression} stems randomly from a spectrum of $\lambda$ in a  different way from the other cases. The dominating term cannot be described deterministically and is governed by the random variable $|R_1|=|X_1|$. For $x_1\in[0,N]$, conditioned on $|R_1|=x_1$, we will show that the sum in \eqref{eq:AN expression} contributes mostly when $m_1(\lambda)$ is close to $ x_1$. The contributions from different values of $|R_1|$ can be roughly described by a slowly varying function in $|R_1|$. If both values of $|R_1|$ and $m_1$ are fixed, the situation is close to ($q$-UNIV): among partitions of $N-m_1$ without ones, those with a large component dominate. 

We sketch the ideas for the lower bound of \eqref{eq:gamma=2q result}, which is the harder part. For $m\in[N/6,N/3]\cap\Z$, we consider the disjoint events that $|R_1|\in[m,m+1)$. On such an event, we show that the main contribution to 
\begin{align}
    \E\Bigg[\bigg|\sum_{\substack{\lambda\in\cP_N}}\prod_{k\geq 1}\left(\frac{X_k}{\sqrt{k}}\right)^{m_k}\frac{1}{m_k!}\bigg|^{2q}\bone_{\{|X_1|\in[m,m+1)\}}\Bigg]\label{m}
\end{align}
stems from the sum over $\lambda$ such that $m_1(\lambda)=m+O(N^{9/10})$, by arguing similarly as \eqref{eq:gamma argument}. Conditioning on $|R_1|=x_1$, such quantity can be further rewritten into 
\begin{align}
    \E\Bigg[\bigg|\sum_{\substack{\lambda\in\cP_N\\ |m_1(\lambda)-m|\leq N^{9/10}}}u(m_1)\prod_{k\geq 2}\left(\frac{X_k}{\sqrt{k}}\right)^{m_k}\frac{1}{m_k!}\bigg|^{2q}\Bigg],\label{g}
\end{align}
where 
\begin{align}
    u(j)=e^{ij\tau_1}\,\frac{|x_1|^{j-m}}{j!/m!}\label{eq:ujuj}
\end{align} and $x_1\in [m,m+1)$. Note that we condition only on $|R_1|$, and $\tau_1$ is uniformly distributed on $[-\pi,\pi]$ independent of anything else.  
The quantity \eqref{g} can be viewed as a generalization of $\E[|A_N|^{2q}]$ (which is essentially $u(j)\equiv 1$, cf.~\eqref{eq:al expression} and \eqref{eq:AN expression}). To derive its asymptotics, we need to apply a finer version of the multiplicative chaos approach in the ($q$-UNIV) case. First, applying Parseval's identity allows us to reduce our problem into studying the asymptotics of the low moments of a \emph{randomly weighted mass} of truncated chaos,
\begin{align}
    \E\left[\left(\int_{-\pi}^{\pi}|\widetilde{F}_{K,M_*,m}(re^{i\theta})|^2\d\theta \right)^q\right],\label{b}
\end{align}
where
$$\widetilde{F}_{K,M_*,m}(z):=F_{K,M_*}(z)\times\Big(\sum_{j:|j-m|\leq N^{9/10}}u(j)z^j\Big),$$
cf.~\eqref{eq:chaos} and \eqref{eq:FK}. 
For $r$ close enough to $1$, the second term on the right-hand side with $z=re^{i\theta}$ can be seen as roughly the discrete Fourier transform of the function $j\mapsto r^mm^{j-m}\Gamma(m)/\Gamma(j)$ at frequency $\tau_1+\theta$, where one expects that roughly,
$$\bigg|\sum_{j:|j-m|\leq N^{9/10}}u(j)(re^{i\theta})^j\bigg|\asymp \min\left\{\frac{r^m}{|\tau_1+\theta|},\sqrt{N}r^m\right\}.$$
Therefore, it is natural to decompose the integral in \eqref{b} depending on the magnitude of $|\theta+\tau_1|$, and reduce the problem to estimating (uniformly) the low moments of the \emph{partial} mass of the truncated chaos,
$$\E\Bigg[\bigg(\int_{|\theta|\leq\frac{1}{K_*}}|{F}_{K,M_*}(re^{i\theta})|^2\d\theta \bigg)^q\Bigg],$$where $K_*\ll\sqrt{N}$.  
This is central to the proof and will be the goal of Section \ref{sec:partial mass}.  
As an example, we briefly illustrate the ideas behind estimating
\begin{align}
    \E\Bigg[\bigg(\int_{|\theta|\leq\frac{1}{\sqrt{N}}}|{F}_{K,M_*}(re^{i\theta})|^2\d\theta \bigg)^q\Bigg], \label{p}
\end{align} using the aforementioned branching random walk analogue. The integration over $\{\theta:|\theta|\leq1/\sqrt{N}\}$ can be translated as considering only those points $z'\in\mathbb{T}$ having the latest common ancestor with $z=0$ of generation no earlier than $(\log N)/2$, which form a sub-tree of the whole family. This suggests that we should 1) consider a decomposition
\begin{align}
    F_{K,M_*}(z)=\overline{F}_{K,M_*}(z)\times \underline{F}_{K,M_*}(z):=\exp\Big(\sum_{k=M_*}^{\sqrt{N}}\frac{X_k}{\sqrt{k}}z^k\Big)\times \exp\Big(\sum^K_{k=\sqrt{N}}\frac{X_k}{\sqrt{k}}z^k\Big).\label{eq:decomp ex}
\end{align} so that if $|\theta|,|\theta'|\leq 1/\sqrt{N}$, uniformly $\overline{F}_{K,M_*}(re^{i\theta})\asymp \overline{F}_{K,M_*}(re^{i\theta'})$; and 2) the integration of $\underline{F}_{K,M_*}$ over $\{|\theta|\leq1/\sqrt{N}\}$ should parallel that of $F_{K,M_*}$ over $\mathbb{T}$, since each sub-tree of a branching random walk can be considered as independently a new BRW with fewer generations. In other words, the term $\overline{F}_{K,M_*}(re^{i\theta})$ shares roughly the same value (up to multiplicative constants that are bounded in probability) for all $\theta$ with $|\theta|\leq 1/\sqrt{N}$. This can be quantified using the generic chaining technique, which gives tail bounds of
$$\sup_{|\theta|\leq 1/\sqrt{N}}\Big|\sum_{k=M_*}^{\sqrt{N}}\frac{r^k}{\sqrt{k}}\Re(X_ke^{ik\theta}-X_k)\Big|$$ under a suitable change of measure, 
allowing us to pull the term $|\overline{F}_{K,M_*}(re^{i\theta})|^2$ out from the integral in  \eqref{p}. With a change of variable $\theta\mapsto\pi\sqrt{N}\theta$, the remaining term
$$\E\Bigg[\bigg(\int_{|\theta|\leq\frac{1}{\sqrt{N}}}|\underline{F}_{K,M_*}(re^{i\theta})|^2\d\theta \bigg)^q\Bigg]\asymp N^{-q/2}\E\Bigg[\bigg(\int_{-\pi}^\pi|\underline{F}_{K,M_*}(re^{i\theta/(\pi\sqrt{N})})|^2\d\theta \bigg)^q\Bigg]$$
shares the same log-correlated structure as the ($q$-UNIV) case, and can be analyzed in a similar way. For details we refer to Proposition \ref{prop:christmas}, which establishes that roughly,
 $$ \E\left[\left(\int_{-\pi}^{\pi}|\widetilde{F}_{K,M_*,m}(re^{i\theta})|^2\d\theta \right)^q\right]\asymp N^{q^2/2}\left(\frac{r^{2m}K}{1+(1-q)\sqrt{\log K}}\right)^q.$$
Moreover, restricting to the event $|X_1|\in[m,m+1)$ in \eqref{m} loses a factor of $m^{-q}$, which can be intuitively verified by the following property of the Gamma function:
\begin{align*}
    \int_m^{m+1}e^{-x}x^{m}\d x\asymp m^{-1/2}\int_0^\infty e^{-x}x^m\d x.
\end{align*}
Altogether, we obtain 
$$\E[|A_N|^{2q}]\gg \sum_{m=N/6}^{N/3}\frac{m^{-q}m^{q^2/2}}{(1+(1-q)\sqrt{\log(N-m)})^{q}}\asymp \frac{N^{1-q+q^2/2}}{(1+(1-q)\sqrt{\log N})^q},$$
 leading to the lower bound of \eqref{eq:gamma=2q result}. 
 
 The upper bound follows a similar approach: condition on $|R_1|$, construct \eqref{g} for all $m\in\N$, use the multiplicative chaos upper bound for the ($q$-UNIV) case, and finally apply Minkowski's inequality ($q\geq 1/2$) or concavity ($q<1/2$).

\section{The universality phase}\label{sec:supercritical}

\subsection{Reducing the proof to Proposition \texorpdfstring{\ref{prop:superC-asymp-truncated-moments}}{}}
\label{sec:LT reduction}

To prove the asymptotics for $\E[|A_N|^{2q}]$ under the ($q$-UNIV) condition, we first show that, given any constant $M_*>0$, removing summands $a(\lambda)$ in $A_N$ with $m_i(\lambda)=0$ for all $i<M_*$ from \eqref{eq:AN expression} costs at most a constant factor depending only on the distribution of $X_k$ (equivalently $R_k$), $q$, and $M_*$. This will be applied later with $M_*$ picked in terms of the distribution of $X_k$ and $q$ only. The same idea will be recycled later a few times, for instance, for the (EXP) case with $\gamma=2q$ in Section \ref{sec:str}.

Recall \eqref{eq:ANM}. 
 We use the following proposition on the low moments of $A_{N,M_*}$ with a suitable $M_*$ to deduce the asymptotics of low moments of $A_N$. We will prove the upper and lower bound parts of this proposition in Sections \ref{sec:ub} and \ref{sec:lb} respectively.

\begin{proposition}\label{prop:superC-asymp-truncated-moments}
Fix an integer $M_*$ larger than some constant depending only on the distribution of $X_k$. For any large $N$ and any $q\in(0,1]$, under ($q$-UNIV) we have \begin{equation*}
        \E[|A_{N,M_*}|^{2q}]\asymp \left(\frac{1}{1+(1-q)\sqrt{\log N}}\right)^q\asymp\begin{cases}
        (\log N)^{-q/2}&\text{ if }q\in(0,1);\\
        1&\text{ if }q=1,
    \end{cases}
    \end{equation*} where the implied constants may depend on the distribution of $X_k$ (equivalently $R_k$), $q$, and $M_*$, but not on $N$. Moreover, one can take $M_*=1$ if $X_k$ (or equivalently, $|R_k|$) is compactly supported.
\end{proposition}\begin{remark}
    The truncation parameter $M_*$ is chosen so that every exponential moment regarding $X_k$ used in the proof of Proposition \ref{prop:superC-asymp-truncated-moments} is finite for any $k\geq M_*$. It is possible to track down this threshold in terms of the constant $k_0$ defined in Lemma \ref{SZ-lemma2.2}, but we omit the work here. 
\end{remark}\begin{proof}[Deducing equation \eqref{eq:asymp LT} of Theorem \ref{thm:main} from Proposition \ref{prop:superC-asymp-truncated-moments}]
If $|X_k|$ is compactly supported, Proposition \ref{prop:superC-asymp-truncated-moments} directly implies \eqref{eq:asymp LT} of Theorem \ref{thm:main}. Therefore, in the rest of the proof, we may assume that $|X_k|$ is not compactly supported. 
    We first establish the upper bound of $\E[|A_N|^{2q}]$. 
Define  
\begin{align}
    \cP_{N,M_*}^{<}:= \bigg\{(m_i)_{1\le i< M_*}: \forall 1\le i<M_*,~ m_i\in\{0,\cdots,N\} \text{ and }\sum_{1\le i<M_*}im_i< N\bigg\}.\label{eq:P def}
\end{align}
For $\bm\in \cP_{N,M_*}^{<}$, we define
$$A_{\bm;N,M_*}:=\sum_{\substack{\lambda\in\cP_N\\ (m_1,\dots,m_{M_*-1})=\bm}}a(\lambda).$$ In particular, $A_{\mathbf{0}; N,M_*}=A_{N,M_*}$.
    For $q\le 1/2$, using concavity we have \begin{align*}
    \E[|A_N|^{2q}]&\leq \sum_{\bm\in\cP_{N,M_*}^{<}}\E[|A_{\bm;N,M_*}|^{2q}]+\sum_{\lambda:\lambda_1<M_*}\E[|a(\lambda)|^{2q}]\\
    &\ll \sum_{\bm\in\cP_{N,M_*}^{<}}\E\left[\big|A_{N-\sum_{j=1}^{M_*-1}jm_j,M_*}\big|^{2q}\right]\prod_{k=1}^{M_*-1}\E\left[\left|\left(\frac{X_k}{\sqrt{k}}\right)^{m_k}\frac{1}{m_k!}\right|^{2q}\right]\\
    &\hspace{2cm}+ \sum_{\lambda:\lambda_1<M_*}\prod_{k=1}^{M_*-1}\E\left[\left|\left(\frac{X_k}{\sqrt{k}}\right)^{m_k}\frac{1}{m_k!}\right|^{2q}\right].
\end{align*}
By Proposition \ref{prop:superC-asymp-truncated-moments} and \eqref{eq:moment bound 1} and since $M_*$ is fixed, the first sum is bounded by
\begin{align*}
    &\hspace{0.5cm}\sum_{\bm\in\cP_{N,M_*}^{<}}\E\left[\big|A_{N-\sum_{j=1}^{M_*-1}jm_j,M_*}\big|^{2q}\right]\prod_{k=1}^{M_*-1}\E\left[\left|\left(\frac{X_k}{\sqrt{k}}\right)^{m_k}\frac{1}{m_k!}\right|^{2q}\right]\\
    &\ll \sum_{\bm\in\cP_{N,M_*}^{<}}\frac{1}{(1+(1-q)\sqrt{\log(N-\sum_{j=1}^{M_*-1}jm_j)})^{q}}\prod_{k=1}^{ M_*-1}Cm_k^{1/2-q}\left(\frac{2q}{\gamma \sqrt{k}}\right)^{2qm_k}\\
    &\ll \sum_{\bm\in\cP_{N,M_*}^{<}}\frac{1}{(1+(1-q)\sqrt{\log(N-\sum_{j=1}^{M_*-1}jm_j)})^{q}}\prod_{k=1}^{ M_*-1}\left(\frac{1}{e^\delta \sqrt{k}}\right)^{2qm_k},
\end{align*}
where $\delta>0$ depends on $\gamma,q$. Next, since $M_*$ is fixed,
\begin{align*}
    \sum_{\lambda:\lambda_1<M_*}\prod_{k=1}^{M_*-1}\E\left[\left|\left(\frac{X_k}{\sqrt{k}}\right)^{m_k}\frac{1}{m_k!}\right|^{2q}\right]&\ll \sum_{\lambda:\lambda_1<M_*}\prod_{k=1}^{M_*-1}Cm_k^{1/2-q}\left(\frac{2q}{\gamma \sqrt{k}}\right)^{2qm_k}\ll \sum_{\lambda:\lambda_1<M_*}\prod_{k=1}^{M_*-1}e^{-\delta m_k}\ll \frac{1}{N},
\end{align*}
where we used in the last step that $\sum_{k<M_*}m_k\ge N/M_*$ for $\lambda\in\cP_N$ with $\lambda_1<M_*$, and the bound $|\cP_N|\ll e^{\pi\sqrt{2/3}\sqrt{N}}$ (see \citep{andrews1998theory}). We thus arrive at
$$\E[|A_N|^{2q}]\ll \sum_{\bm\in\cP_{N,M_*}^{<}}\frac{1}{(1+(1-q)\sqrt{\log(N-\sum_{j=1}^{M_*-1}jm_j)})^{q}}\prod_{k=1}^{ M_*-1}\left(\frac{1}{e^\delta \sqrt{k}}\right)^{2qm_k}+\frac{1}{N}.$$
Let us divide the above sum over $\bm\in\cP_{N,M_*}^{<}$ into two parts depending on whether $\sum_{j=1}^{M_*-1}jm_j\leq N/2$ or not. First,
\begin{align*}&\hspace{0.5cm}\sum_{\substack{\bm\in\cP_{N,M_*}^{<}\\ \sum_{j=1}^{M_*-1}jm_j\leq N/2}}\frac{1}{(1+(1-q)\sqrt{\log(N-\sum_{j=1}^{M_*-1}jm_j)})^{q}}\prod_{k=1}^{M_*-1}\left(\frac{1}{e^\delta \sqrt{k}}\right)^{2qm_k}\\
    &\ll \frac{1}{(1+(1-q)\sqrt{\log N})^q}\sum_{\substack{\bm\in\cP_{N,M_*}^{<}\\ \sum_{j=1}^{M_*-1}jm_j\leq N/2}}\prod_{k=1}^{M_*-1}\left(\frac{1}{e^\delta \sqrt{k}}\right)^{2qm_k}\\
    &\leq \frac{1}{(1+(1-q)\sqrt{\log N})^q} \prod_{k=1}^{M_*-1}\sum_{m_k=0}^\infty \left(\frac{1}{e^\delta \sqrt{k}}\right)^{2qm_k}\\
    &\ll \frac{1}{(1+(1-q)\sqrt{\log N})^q}.
\end{align*}
Second, if $\sum_{j=1}^{M_*-1}jm_j> N/2$, then $\sum_{j=1}^{M_*-1}m_j> N/(2M_*)$. This implies
\begin{align}
\begin{split}
    \sum_{\substack{\bm\in\cP_{N,M_*}^{<}\\ \sum_{j=1}^{M_*-1}jm_j> N/2}}&\frac{1}{(1+(1-q)\sqrt{\log(N-\sum_{j=1}^{M_*-1}jm_j)})^{q}}\prod_{k=1}^{M_*-1}\left(\frac{1}{e^\delta \sqrt{k}}\right)^{2qm_k}\\
    &\ll \sum_{\substack{\bm\in\cP_{N,M_*}^{<}\\ \sum_{j=1}^{M_*-1}m_j> N/(2M_*)}}\prod_{k=1}^{M_*-1}\left(\frac{1}{e^\delta \sqrt{k}}\right)^{2qm_k}\leq (N+1)^{M_*}\exp\left(-\frac{2\delta qN}{2M_*}\right) \ll \frac{1}{N}.
\end{split}\label{eq:second case}
\end{align}
Altogether, these yield $\E[|A_N|^{2q}]\ll (1+(1-q)\sqrt{\log N})^{-q}$ for $q\le 1/2$. The case $q\in(1/2,1]$ is similar by looking at $\E[|A_N|^{2q}]^{1/(2q)}$ and using Minkowski's inequality instead of concavity.

Next, we show a matching lower bound on $\E[|A_N|^{2q}]$ by using that of $\E[|A_{N,M_*}|^{2q}]$.
The first observation is that we can, with the cost of a constant factor depending on $L$, remove those $a(\lambda)$ such that $0<m_1(\lambda)<2^L$ for some large integer $L$ to be determined later. Denote by $\mathcal{A}_{k}\subset\cP_N$ the set of partitions of $N$ such that $2^k|m_1(\lambda)$, for $k\geq 0$. In particular, $\mathcal A_0=\cP_N$. Using $X_1\ddd -X_1$, we get 
\[A_N=\sum_{\lambda\in\mathcal{A}_1}a(\lambda)+\sum_{\lambda\in\mathcal{A}_0\setminus\A_1}a(\lambda)\ddd \sum_{\lambda\in\mathcal{A}_1}a(\lambda)-\sum_{\lambda\in\mathcal{A}_0\setminus\A_1}a(\lambda),\] and therefore, \begin{align}
\begin{split}
    \E[|A_N|^{2q}] &= \frac{1}{2}\Bigg(\E\Bigg[\Big|\sum_{\lambda\in\mathcal{A}_1}a(\lambda)+\sum_{\lambda\in\mathcal{A}_0\setminus\A_1}a(\lambda)\Big|^{2q}+\Big|\sum_{\lambda\in\mathcal{A}_1}a(\lambda)-\sum_{\lambda\in\mathcal{A}_0\setminus\A_1}a(\lambda)\Big|^{2q}\Bigg]\Bigg)\geq \frac{1}{2}\E\Bigg[\Big|\sum_{\lambda\in\mathcal{A}_1}a(\lambda)\Big|^{2q}\Bigg],
\end{split}
\label{eq:symm}
\end{align}
where we used that $\max(|w-z|,|w+z|)\geq|w|$ for complex numbers $z,w$. 
Further, to get rid of $\lambda\in \mathcal{A}_{k-1}\setminus\A_k$, we use $X_1\ddd e^{i\pi/2^{k-1}}X_1$ and apply the same argument. More precisely, for each $k\geq 1$, using $X_1\ddd e^{i\pi/2^{k-1}}X_1$, we have
$$\sum_{\lambda\in\mathcal{A}_{k}}a(\lambda)+\sum_{\lambda\in\mathcal{A}_{k-1}\setminus\A_k}a(\lambda)\ddd \sum_{\lambda\in\mathcal{A}_k}a(\lambda)-\sum_{\lambda\in\mathcal{A}_{k-1}\setminus\A_k}a(\lambda),$$
which implies 
$$\E\Bigg[\Big|\sum_{\lambda\in\mathcal{A}_{k-1}}a(\lambda)\Big|^{2q}\Bigg]\geq \frac{1}{2}\E\Bigg[\Big|\sum_{\lambda\in\mathcal{A}_k}a(\lambda)\Big|^{2q}\Bigg].$$
By induction, finally we have as claimed above \[\E[|A_N|^{2q}]\ge 2^{-L}
\E\Bigg[\Big|\sum_{\lambda\in\mathcal{A}_L}a(\lambda)\Big|^{2q}\Bigg].\] Continuing this procedure for $m_2,\cdots,m_{M_*-1}$, we get \[\E[|A_N|^{2q}]\geq 2^{-M_*L}
\E\Bigg[\Big|\sum_{\lambda\in\mathcal{B}_{M_*,L}}a(\lambda)\Big|^{2q}\Bigg]=:2^{-M_*L}\E[|\widehat{A}_N|^{2q}],\] where 
$$\mathcal{B}_{M_*,L}=\left\{\lambda\in\cP_N: \,\forall 1\le k<M_*,\,2^L\mid m_k(\lambda) \right\}\subseteq\left\{\lambda\in\cP_N: \,\forall 1\le k<M_*,\,m_k(\lambda)\not=0\implies m_k(\lambda)\ge 2^L\right\}.$$
It thus remains to show that
$${\E}[|\widehat{A}_N|^{2q}]\gg \left(\frac{1}{1+(1-q)\sqrt{\log N}}\right)^q,$$
where the implied constant does not depend on $L$. 
Using the triangle inequality and either concavity (for $q\leq 1/2$) or conditional Minkowski's inequality (for $q\geq 1/2)$ on the decomposition
$$\widehat{A}_N=\sum_{\lambda\in\mathcal{B}_{M_*,L}}a(\lambda)=A_{N,M_*}+\sum_{\substack{\lambda\in\mathcal{B}_{M_*,L}\\ \exists j\in[1,M_*),~m_j\neq 0}}a(\lambda)$$
and Proposition \ref{prop:superC-asymp-truncated-moments}, 
 it then suffices to prove that given a fixed $M_*$, for $L$ large enough,
\begin{align*}
    {\E}\Bigg[\bigg|\sum_{\substack{\lambda\in\mathcal{B}_{M_*,L}\\ \exists j\in[1,M_*),~m_j\neq 0}}a(\lambda)\bigg|^{2q} \Bigg]\leq \frac{1}{2}\E[|A_{N,M_*}|^{2q}].
\end{align*}

Since we assume the unit variance condition $\E[|X_1|^2]=\E[|R_1|^2]=1$, we may assume that there is some $\ee_0\in(0,1)$ such that $\P(|R_1|<\ee_0)>0$ (otherwise, $|R_1|\equiv 1$, so $|X_1|$ is compactly supported, which we assumed from the beginning that is not the case). The strategy is then to restrict to the event that $|X_j|=|R_j|<\ee_0$ for each $1\leq j\leq M_*-1$, that is, with $\widehat{\E}$ denoting the conditional expectation on the event $\{|R_j|<\ee_0,~1\leq j\leq M_*-1\}$,  \begin{align*}
    \E\Bigg[\bigg|\sum_{\substack{\lambda\in\mathcal{B}_{M_*,L}\\ \exists j\in[1,M_*),~m_j\neq 0}}a(\lambda)\bigg|^{2q} \Bigg]&\gg \E\Bigg[\bigg|\sum_{\substack{\lambda\in\mathcal{B}_{M_*,L}\\ \exists j\in[1,M_*),~m_j\neq 0}}a(\lambda)\bigg|^{2q} \, \mid\, |R_j|<\ee_0,~1\leq j\leq M_*-1\Bigg]\\
    &=\widehat{\E}\Bigg[\bigg|\sum_{\substack{\lambda\in\mathcal{B}_{M_*,L}\\ \exists j\in[1,M_*),~m_j\neq 0}}a(\lambda)\bigg|^{2q} \Bigg],
\end{align*}where the implied constant does not depend on $L$. 
 It then suffices to prove that given a fixed $M_*$, for $L$ large enough,
\begin{align}
    \widehat{\E}\Bigg[\bigg|\sum_{\substack{\lambda\in\mathcal{B}_{M_*,L}\\ \exists j\in[1,M_*),~m_j\neq 0}}a(\lambda)\bigg|^{2q} \Bigg]\leq \frac{1}{C}\E[|A_{N,M_*}|^{2q}]\label{eq:Cbound}
\end{align}
where $C>0$ is a fixed large constant that does not depend on $L$. 

Let $C_0>0$ be the implied constant in the upper bound part of Proposition \ref{prop:superC-asymp-truncated-moments}. Let $\cP_{N,M_*}^{<,L}$ be the subset of $\cP_{N,M_*}^{<}$ with $m_k(\lambda)\ge 2^L$ for all $1\leq k<M_*$ such that $m_k(\lambda)\not=0$. We have for $q\leq 1/2$ (and similarly using Minkowski's inequality for $q>1/2$),
\begin{align*}
    \widehat{\E}\Bigg[\bigg|\sum_{\substack{\lambda\in\mathcal{B}_{M_*,L}\\ \exists j\in[1,M_*),~m_j\neq 0}}a(\lambda)\bigg|^{2q} \Bigg]&=\widehat{\E}\Bigg[\bigg|\sum_{\substack{\lambda\in\mathcal{B}_{M_*,L}\\ \exists j\in[1,M_*),~m_j\neq 0}}\prod_{k\geq 1}\left(\frac{X_k}{\sqrt{k}}\right)^{m_k}\frac{1}{m_k!}\bigg|^{2q}\Bigg]\\
    &\leq \sum_{\substack{\bm\in\cP_{N,M_*}^{<,L}\\ \bm\neq\mathbf{0}}}\widehat{\E}\Bigg[\bigg|\sum_{\substack{\lambda\in\cP_N\\ (m_1,\dots,m_{M_*-1})=\bm}}\prod_{k\geq 1}\left(\frac{X_k}{\sqrt{k}}\right)^{m_k}\frac{1}{m_k!}\bigg|^{2q}\Bigg]\\
    &=\sum_{\substack{\bm\in\cP_{N,M_*}^{<,L}\\ \bm\neq\mathbf{0}}}\prod_{k=1}^{M_*-1}\frac{\widehat{\E}[|X_k|^{2qm_k}]}{k^{qm_k}(m_k!)^{2q}}\E\Bigg[\bigg|\sum_{\substack{\lambda\in\cP_N\\ (m_1,\dots,m_{M_*-1})=\bm}}\prod_{k\geq M_*}\left(\frac{X_k}{\sqrt{k}}\right)^{m_k}\frac{1}{m_k!}\bigg|^{2q}\Bigg]\\
    &\leq \sum_{\substack{\bm\in\cP_{N,M_*}^{<,L}\\ \bm\neq\mathbf{0}}}\frac{C_0}{(1+(1-q)\sqrt{\log(N-\sum_{j=1}^{M_*-1}jm_j)})^{q}}\prod_{k=1}^{M_*-1}\frac{\ee_0^{2qm_k}}{k^{qm_k}}.
\end{align*}
The sum over $\bm$ with $\sum_{j=1}^{M_*-1}jm_j> N/2$ can be controlled similarly as in \eqref{eq:second case}. The rest of the terms are bounded by
\begin{align*}
    \sum_{\substack{\bm\in\cP_{N,M_*}^{<,L}\\ \bm\neq\mathbf{0}\\ \sum_{j=1}^{M_*-1}jm_j\leq N/2}}\frac{C_0}{(1+(1-q)\sqrt{\log(N-\sum_{j=1}^{M_*-1}jm_j)})^{q}}\prod_{k=1}^{M_*-1}\frac{\ee_0^{2qm_k}}{k^{qm_k}}&\leq \frac{2C_0}{(1+(1-q)\sqrt{\log N})^q}\sum_{\substack{\bm\in\cP_{N,M_*}^{<,L}\\ \bm\neq\mathbf{0}}}\prod_{k=1}^{M_*-1}{\ee_0^{2qm_k}}.
\end{align*}
Note that 
\begin{align*}
  \sum_{\substack{\bm\in\cP_{N,M_*}^{<,L}\\ \bm\neq\mathbf{0}}}\prod_{k=1}^{M_*-1}{\ee_0^{2qm_k}}&\leq \prod_{k=1}^{M_*-1}\bigg(1+\sum_{m_k=2^L}^{\infty}{\ee_0^{2qm_k}}\bigg)-1\ll \frac{M_*\ee_0^{2^{L+1}q}}{1-\ee_0^{2q}}.
\end{align*}
By picking $L$ large enough, we obtain 
\begin{align*}
    \E\Bigg[\bigg|\sum_{\substack{\lambda\in\mathcal{B}_{M_*,L}\\ \exists j\in[1,M_*),~m_j\neq 0}}\prod_{k\geq 1}\left(\frac{X_k}{\sqrt{k}}\right)^{m_k}\frac{1}{m_k!}\bigg|^{2q}\mid |R_j|<\ee_0,~1\leq j\leq M_*-1\Bigg]&\leq \frac{1}{C_0C(1+(1-q)\sqrt{\log N})^q}\\
    &\leq \frac{1}{C}\E[|A_{N,M_*}|^{2q}],
\end{align*}
thus proving \eqref{eq:Cbound} and hence the desired lower bound.
\end{proof}

\subsection{Setting the stage for proving Proposition \texorpdfstring{\ref{prop:superC-asymp-truncated-moments}}{}}\label{sec:prop2.1setup}

We now proceed to the proof of Proposition \ref{prop:superC-asymp-truncated-moments}. In this short section, we prepare ourselves with a few notations. First, motivated by our discussions on page \pageref{eq:equal    }, we define two probability measures as follows.

    For any $K,M,r$ satisfying $\log M_*\le M<\log K$ and $ e^{-1/K}\leq r\leq e^{1/K}$, define the measure $\Q^{(1)}_{r,M,K}$ by
    \begin{align}
        \frac{\d\Q^{(1)}_{r,M,K}}{\d\P} := \frac{\exp(2\sum_{k=e^M}^{K-1}\frac{r^k}{\sqrt{k}}R_k\cos(\tau_k))}{\E\left[\exp(2\sum_{k=e^M}^{K-1}\frac{r^k}{\sqrt{k}}R_k\cos(\tau_k))\right]}.\label{eq:Q1 def}
    \end{align}
    In addition, let $K_r$ be such that  $\log K_r$ is the largest integer with $K_r\le\min\{\frac{-1}{4\log r},K\}$, and $M=M(r,\theta)\in\N$ to be determined that depends on $r,\theta$ (see the definitions at the beginning of Sections \ref{sec:332} and \ref{sec:342}). For $e^{-1/K}\leq r\leq e^{1/K}$ and  $\theta\in[-\pi,\pi)$, define the measure $\Q^{(2)}_{r,M,K,\theta}$ by \begin{align}
        \frac{\d\Q^{(2)}_{r,M,K,\theta}}{\d\P} := \frac{\exp(2\sum_{m=M+1}^{\log K_r}(Z_0(m)+Z_\theta(m)))}{\E[\exp(2\sum_{m=M+1}^{\log K_r}(Z_0(m)+Z_\theta(m)))]},\label{eq:dq2}
    \end{align}
    where  for any $M<m\le\log K_r$ and $\theta\in[-\pi,\pi)$,
    \begin{align}\label{eq:ztheta}
    Z_\theta(m) &:=\Re \sum_{e^{m-1}\le k<e^m}\frac{X_kr^ke^{ik\theta}}{\sqrt{k}} =\sum_{e^{m-1}\le k<e^m}\frac{r^k}{\sqrt{k}}R_k\cos(\tau_k+k\theta).
    \end{align} When the values of $r,M,K,\theta$ are clear from the context, we will drop the subscripts and write instead $\Q^{(1)}$ and $\Q^{(2)}$.

    To further understand the sums $Y_n$ defined in \eqref{eq:Yndef} under the new measures, we define and compute using Lemma \ref{lemma-Q-computations} that
    \begin{align}
    \mu_k:=\E^{\Q^{(1)}}\left[\frac{r^k}{\sqrt{k}}R_k\cos(\tau_k)\right]=\frac{r^{2k}}{k}+O(k^{-3/2})\label{eq:muk def}
\end{align} and 
\begin{align}
    \nu_k=\nu_k(\theta):=\E^{\Q^{(2)}}\left[\frac{r^k}{\sqrt{k}}R_k\cos(\tau_k)\right] = \frac{r^{2k}}{k}+\frac{\cos(k\theta)r^{2k}}{k}+O(k^{-3/2}).\label{eq:nuk def}
\end{align}

We also refresh ourselves with the truncated multiplicative chaos. Recall \eqref{eq:ANM} and \eqref{eq:FK}. It holds that
\begin{align*}
    A_{N,M_*} = [z^N]\exp\Big(\sum_{k=M_*}^\infty\frac{X_k}{\sqrt{k}}z^k\Big).
\end{align*} Moreover, inserting $K=N/2$ into \eqref{eq:ANM}, we have \begin{align}
    F_{N/2,M_*}(z)=\exp\bigg(\sum_{k=M_*}^{ N/2}\frac{X_k}{\sqrt{k}}z^k\bigg)=\sum_{n=0}^\infty \bigg(\sum_{\substack{\lambda\in\cP_n\\ \lambda_1\leq N/2\\ \forall 1\leq k<M_*,\,m_k(\lambda)=0}}a(\lambda)\bigg)z^n.\label{eq:FA relation}
\end{align}
In the next two subsections, we establish the upper and lower bounds for Proposition \ref{prop:superC-asymp-truncated-moments} respectively.

\subsection{Upper bound of Proposition \ref{prop:superC-asymp-truncated-moments}}\label{sec:ub}

By H\"older's inequality, it suffices to prove the upper bound for $ 1/2\le q\le 1$. Proposition \ref{prop:superC-asymp-truncated-moments} then follows from the two propositions below.

\begin{proposition}\label{SZ-prop3.1}Suppose that $\E[e^{\gamma|R_1|}]<\infty$ for some $\gamma>2q$. 
For $ 1/2\le q\le 1$ and $N$ large enough, we have for some $C(q)>0$,\[\E[|A_{N,M_*}|^{2q}]^{1/(2q)}\ll \frac{1}{\sqrt{N}}\sum_{j=1}^J\E\left[\left(\frac{1}{2\pi}\int_{-\pi}^{\pi}|F_{N/2^j,M_*}(\exp(j/N+i\theta))|^2 \d\theta \right)^q \right]^{1/(2q)}+\frac{1}{N},\]
where $J=\ceil{\log(C(q)\sqrt{N})/\log 2}$. For $q\in(0,1/2]$, we have
$$\E[|A_{N,M_*}|^{2q}]\ll \frac{1}{\sqrt{N}}\sum_{j=1}^J\E\left[\left(\frac{1}{2\pi}\int_{-\pi}^{\pi}|F_{N/2^j,M_*}(\exp(j/N+i\theta))|^2 \d\theta \right)^q \right]+\frac{1}{N^{2q}}.$$
\end{proposition}

\begin{proof}
The main idea resembles that of \citep[Proposition 3.1]{soundararajan2022model}: separate the sum over $\lambda\in \cP_N$ according to the values of $\lambda_1$, control the contributions from partitions $\lambda$ with a small $\lambda_1$ (and hence some $m_k(\lambda)$ is large), and apply Parseval's identity to the rest $\lambda$.

Suppose first that $q\geq 1/2$. 
Denote by $\cP_{N,M_*}^0$ the subset of $\cP_N$ satisfying $m_1(\lambda)=\cdots=m_{M_*-1}(\lambda)=0$. Recall \eqref{eq:al expression} and \eqref{eq:AN expression}.  
Suppose that  $\E[e^{\gamma|R_1|}]<\infty$ for some $\gamma>2q$. Then $\E[|R_1|^\ell]\ll \gamma^{-\ell}\Gamma(\ell+1)$ for $\ell\geq 0$. 
 Therefore, 
 \begin{align}
     \E[|a(\lambda)|^{2q}]=\prod_{k:m_k>0}\frac{\E[|R_k|^{2qm_k}]}{(m_k!)^{2q}k^{qm_k}}\ll \prod_{k:m_k>0}\frac{C\gamma^{-2qm_k}\Gamma(2qm_k+1)}{(m_k!)^{2q}k^{qm_k}}\ll \prod_{k:m_k>0} \frac{C({2q}/{\gamma})^{2qm_k}}{k^{qm_k}},\label{eq:a lambda moments}
 \end{align} where we have  used \eqref{eq:gamma}. 
 By Minkowski's inequality, 
 \begin{align}
     \E\Bigg[\bigg|\sum_{\substack{\lambda\in\cP_{N,M_*}^0\\ \lambda_1\leq N/2^J}}a(\lambda)\bigg|^{2q}\Bigg]^{1/(2q)}&=\E\Bigg[\bigg|\sum_{\substack{\lambda\in\cP_{N,M_*}^0\\ \lambda_1\leq \sqrt{N}/C(q)}}a(\lambda)\bigg|^{2q}\Bigg]^{1/(2q)}\nonumber \\
     &\le \sum_{\substack{\lambda\in\cP_{N,M_*}^0\\ \lambda_1\leq \sqrt{N}/C(q)}} \E[|a(\lambda)|^{2q}]^{1/(2q)}\nonumber\\
     &\le \sum_{\substack{\lambda\in\cP_{N,M_*}^0\\ \lambda_1\leq \sqrt{N}/C(q)}}\prod_{k:m_k>0} C\left(\frac{2q}{\gamma}\right)^{m_k}\nonumber\\
     &\le |\cP_N| C^{\sqrt{N}/C(q)}\left(\frac{2q}{\gamma}\right)^{\min_\lambda\sum_{k=M_*}^{\sqrt{N}/C(q)}m_k},
     \label{eq:holder p hat}
 \end{align}where the minimum in \eqref{eq:holder p hat} is taken over $\lambda\in\cP_{N,M_*}^0$ such that $\lambda_1\leq \sqrt{N}/C(q)$. Note that for partitions $\lambda\in \cP_{N,M_*}^0$ with the largest component at most $\sqrt{N}/C(q)$, we have \begin{equation}
    \sum_{k=M_*}^{\sqrt{N}/C(q)}m_k\geq \frac{C(q)}{\sqrt{N}} \sum_{k=M_*}^{\sqrt{N}/C(q)}km_k= C(q)\sqrt{N}.\label{eq:o}
 \end{equation} Inserting back to \eqref{eq:holder p hat}, together with the classical bound $|\cP_N|\ll e^{\pi\sqrt{2/3}\sqrt{N}}$, we have
$$\E\Bigg[\bigg|\sum_{\substack{\lambda\in \cP_{N,M_*}^0\\ \lambda_1\leq N/2^J}}a(\lambda)\bigg|^{2q}\Bigg]^{1/(2q)}\ll \left(\frac{2q}{\gamma}\right)^{C(q)\sqrt{N}/2}\ll\frac{1}{N},$$ where we take $C(q)$ large enough in terms of the fixed constant $C$ and $q,\gamma$.
    The rest of the arguments follow similarly as in Section 4 of \citep{soundararajan2022model}, which we sketch below for completeness. By Minkowski's inequality,
\begin{align}
    \E[|A_{N,M_*}|^{2q}]^{1/(2q)}\ll \sum_{j=1}^J\E\Bigg[\bigg|\sum_{\substack{\lambda\in\cP_{N,M_*}^0\\ N/2^j<\lambda_1\leq N/2^{j-1}}}a(\lambda)\bigg|^{2q}\Bigg]^{1/(2q)}+\frac{1}{N}.\label{eq:SZ3.1Minkowski}
\end{align}
    For a fixed $j\in\{1,\dots,J\}$, we may decompose
    $$\sum_{\substack{\lambda\in\cP_{N,M_*}^0\\ N/2^j<\lambda_1\leq N/2^{j-1}}}a(\lambda)=\sum_{\substack{\rho,\sigma\\ |\rho|+|\sigma|=N\\ |\rho|>0}}a(\rho)a(\sigma),$$
    where the parts of $\rho$ lie in $(N/2^j,N/2^{j-1}]$ and of $\sigma$ lie in $[M_*,N/2^j]$ (here and later, we keep such constraints in the sums over $\rho$ or $\sigma$). 
    Let $\E_j$ denote the expectation in $\{X_k\}_{N/2^j<k\leq N/2^{j-1}}$.
    Note that if $\lambda\neq\lambda'$ are distinct partitions, then $\E[a(\lambda)\overline{a(\lambda')}]=0$ by independence of $\{X_k\}_{k\geq 1}$. 
    Therefore, by applying Jensen's inequality and expanding the square, we obtain
    \begin{align}
        \E_j\Bigg[\bigg|\sum_{\substack{\lambda\in\cP_{N,M_*}^0\\ N/2^j<\lambda_1\leq N/2^{j-1}}}a(\lambda)\bigg|^{2q}\Bigg]^{1/q}&\leq \sum_{N/2^j<n\leq N}\bigg|\sum_{\sigma\in \cP_{N-n,M_*}^0}a(\sigma)\bigg|^2\sum_{\rho\in\cP_n}\E_j\Big[|a(\rho)|^2\Big].\label{eq:SZ 4.5}
    \end{align}
    For a partition $\rho$ whose parts lie in $(N/2^j,N/2^{j-1}]$, we have
    $$\E_j\Big[|a(\rho)|^2\Big]\leq \prod_{N/2^j<k\leq N/2^{j-1}}\frac{C\gamma^{-2m_k}(2m_k)!}{(m_k!)^2 k^{m_k}}\leq \Big(\frac{C2^j}{N}\Big)^r,$$where $r$ is the number of parts in $\rho$. Using $r\leq 2^j\leq C(q)^2N/2^j$ for $1\leq j\leq J$, it follows that
\begin{align*}
    \sum_{\rho\in\cP_n}\E_j\Big[|a(\rho)|^2\Big]&\leq \sum_{2^{j-1}n/N\leq r<2^jn/N} \Big(\frac{C2^j}{N}\Big)^r\binom{\floor{N/2^j}+r}{r-1}\\
    &\leq  \frac{2^j}{N}\sum_{2^{j-1}n/N\leq r<2^jn/N}\frac{C^{r-1}}{(r-1)!}\ll \frac{1}{N}\exp\Big(\frac{2j(N-n)}{N}\Big).
\end{align*}
Inserting in \eqref{eq:SZ 4.5} leads to
\begin{align}\begin{split}
    \E_j\Bigg[\bigg|\sum_{\substack{\lambda\in\cP_{N,M_*}^0\\ N/2^j<\lambda_1\leq N/2^{j-1}}}a(\lambda)\bigg|^{2q}\Bigg]^{1/q}&\ll \frac{1}{N}\sum_{N/2^j<n\leq N}\bigg|\sum_{\sigma\in \cP_{N-n,M_*}^0}a(\sigma)\bigg|^2\exp\Big(\frac{2j(N-n)}{N}\Big)\\
    &\ll \frac{1}{2\pi N}\int_{-\pi}^{\pi}\Big|F_{N/2^j,M_*}\Big(\exp\Big(\frac{j}{N}+i\theta\Big)\Big)\Big|^2 \d\theta,
\end{split}\label{q}
\end{align}
where the last step follows from Parseval's identity applied to
$$F_{N/2^j,M_*}(z)=\sum_r \Big(\sum_{\sigma\in \cP_{r,M_*}^0}a(\sigma)\Big)z^r.$$
Taking expectations on the $q$-th powers of both sides of \eqref{q} completes the proof for $q\geq 1/2$. The case $q\in(0,1/2]$ follows along the same lines by using concavity instead of Minkowski's inequality.
\end{proof}

\begin{proposition}    \label{SZ-prop3.2}  Fix any $q\in(0,1]$. Suppose that $\E[e^{\gamma|R_1|}]<\infty$ for some $\gamma>2q$.\footnote{Strictly speaking, Proposition \ref{SZ-prop3.2} does not rely on $\gamma>2q$. Assuming only $\gamma>0$, the same conclusion holds if $M_*$ is large enough in terms of $\gamma$.} For any $K$ sufficiently large (in terms of $M_*$) and $1\le r\le e^{1/K}$, we have
\begin{align}
    \E\left[\left(\int_{-\pi}^{\pi}|F_{K,M_*}(re^{i\theta})|^2\d\theta \right)^q\right]\ll \left(\frac{K}{1+(1-q)\sqrt{\log K}}\right)^q\asymp\begin{cases}
        K^q(\log K)^{-q/2}&\text{ if }q\in(0,1);\\
        K^q&\text{ if }q=1,
    \end{cases}\label{eq:prop3.2}
\end{align}
where the implied constant depends on $q,\gamma$ but is independent of $K,r$.
\end{proposition}
Inserting Proposition \ref{SZ-prop3.2} into Proposition \ref{SZ-prop3.1} then yields the upper bound of Proposition \ref{prop:superC-asymp-truncated-moments}. 
To set up the proof of Proposition \ref{SZ-prop3.2} we need the following two-sided ballot estimate for a centered Gaussian random walk, which is due to \citep{harper2020moments}.

\begin{lemma}\label{lemma:ballot}
    There is a large universal constant $L_1>0$ such that the following holds. Consider a sequence of independent centered Gaussian random variables $\{G_n\}_{n\in\N}$ with variances between ${1}/{20}$ and $20$. Then uniformly for any functions $h(m),\,g(m)$ satisfying $|h(m)|\leq 10\log m$ and $g(m)\leq -L_1m$, and for $a,n$ large enough,
    \[\P\left(\forall 1\leq m\leq n,~g(m)\leq \sum_{j=1}^m G_j\leq \min\{a,L_1m\}+h(m)\right)\asymp\min\Big\{1,\frac{a}{\sqrt{n}}\Big\}.\]
    The same conclusion holds if we replace $\min\{a,L_1 m\}$ above by $a$. 
\end{lemma}

\begin{proof}
    The first claim is \citep[Probability Result 2]{harper2020moments}. For the second claim, the upper bound follows from \citep[Probability Result 1]{harper2020moments}, and the lower bound follows from the first claim.
\end{proof}

Recall the definition of $\mu_k$ from \eqref{eq:muk def}.

\begin{definition}\label{def:Gr}
Fix a large universal constant $L_1>20$ as in Lemma \ref{lemma:ballot}. Let $K$ be large enough (in terms of $M_*$) and $1\leq r\leq e^{1/K}$, and  suppose that $1\le A\le\sqrt{\log K}$. Define the event\[\mathcal{G}_r(A,\theta;K):=\Bigg\{\forall \log M_*\le n\le\log K,\,-A-L_1n\le \sum_{M_*\leq k< e^n}\left(\Re \frac{X_k r^k e^{ik\theta}}{\sqrt{k}}-\mu_k\right)\le A+10\log n\Bigg\}.\]
Define also the event $\mathcal{G}_r(A;K)$ that $\mathcal{G}_r(A,\theta;K)$ holds for all $\theta\in[-\pi,\pi)$.
\end{definition}

Proposition \ref{SZ-prop3.2} then follows from two propositions below.

\begin{proposition}\label{SZ-prop5.2}
    For any $K$ sufficiently large, $1\leq r\leq e^{1/K}$, and  $1\le A\le\sqrt{\log K}$, \[\P(\mathcal{G}_r(A;K)^c)\ll \exp(-A).\]
\end{proposition}

\begin{proposition}\label{SZ-prop5.3}
   For any $K$ sufficiently large, $1\leq r\leq e^{1/K}$, $\theta\in[-\pi,\pi)$, and  $1\le A\le\sqrt{\log K}$, \[\E[\bone_{\mathcal{G}_r(A,\theta;K)}|F_{K,M_*}(re^{i\theta})|^2]\ll\frac{AK}{\sqrt{\log K}},\]
    and therefore \[\E\left[\bone_{\mathcal{G}_r(A;K)}\int_{-\pi}^{\pi}|F_{K,M_*}(re^{i\theta})|^2\d\theta \right]\ll\frac{AK}{\sqrt{\log K}}.\]   
\end{proposition}

\begin{proof}[Deducing Proposition \ref{SZ-prop3.2} from Propositions \ref{SZ-prop5.2} and  \ref{SZ-prop5.3}.]  If $q=1$, the upper bound of \eqref{eq:prop3.2} is equivalent to $K^q$ (up to constant) and the desired claim follows from Lemma \ref{SZ-lemma2.2!}. 
If $q\in(0,1)$, it suffices to prove 
\[\E\left[\left(\int_{-\pi}^{\pi}|F_{K,M_*}(re^{i\theta})|^2\d\theta \right)^q\right]\ll \left(\frac{K}{\sqrt{\log K}}\right)^q.\] Next, we partition the whole probability space into the events  \[\mathcal{G}_r(1;K),~\mathcal{G}_r(2^J;K)^c,~\text{and }\mathcal{G}_r(2^j;K)\backslash \mathcal{G}_r(2^{j-1};K),\ 1\le j\le J:=\floor{\frac{1}{2}\log_2\log K}.\] Then we have by Proposition \ref{SZ-prop5.3} that\[\E\left[\bone_{\mathcal{G}_r(1;K)}\left(\int_{-\pi}^{\pi}|F_{K,M_*}(re^{i\theta})|^2\d\theta \right)^q\right]\le \left(\E\left[\bone_{\mathcal{G}_r(1;K)}\int_{-\pi}^{\pi}|F_{K,M_*}(re^{i\theta})|^2\d\theta \right]\right)^q\ll\left(\frac{K}{\sqrt{\log K}}\right)^q,\]
and by H\"older's inequality and Propositions \ref{SZ-prop5.2} and \ref{SZ-prop5.3},
\begin{align*}
    &\E\left[\bone_{\mathcal{G}_r(2^j;K)\backslash\mathcal{G}_r(2^{j-1};K)}\left(\int_{-\pi}^{\pi}|F_{K,M_*}(re^{i\theta})|^2\d\theta \right)^q\right]\\
    &\le (\P(\mathcal{G}_r(2^{j-1};K)^c))^{1-q}\left(\E\left[\bone_{\mathcal{G}_r(2^j;K)}\int_{-\pi}^{\pi}|F_{K,M_*}(re^{i\theta})|^2\d\theta \right]\right)^q\ll \Big(\frac{K}{\sqrt{\log K}}\Big)^{q}2^{jq}\exp(-(1-q)2^{j-1}).
\end{align*}
Note that
\[\sum_{j\geq 1}2^{jq}\exp(-(1-q)2^{j-1})\ll \sum_{1\leq j\leq -\log_2(1-q)}2^{jq}\exp(-(1-q)2^{j-1})\ll \left(\frac{1}{1-q}\right)^q.\]
Finally, by H\"older's inequality, Lemma \ref{SZ-lemma2.2!}, and Proposition \ref{SZ-prop5.2},
\begin{align*}
    \E\left[\bone_{\mathcal{G}_r(2^J;K)^c}\left(\int_{-\pi}^{\pi}|F_{K,M_*}(re^{i\theta})|^2\d\theta \right)^q\right]&\ll \left(\frac{K}{\sqrt{\log K}}\right)^q.
\end{align*}
Combining the above estimates concludes the proof.
\end{proof}

\subsubsection{Proof of Proposition \ref{SZ-prop5.2}}


By definition, if $\mathcal{G}_r(A;K)$ fails, then there must be some $\log M_*\le n\le\log K$ such that either \begin{align}
    \max_{\theta\in[-\pi,\pi)}\sum_{k=M_*}^{e^n-1}\left(\Re \frac{X_kr^ke^{ik\theta}}{\sqrt{k}}-\mu_k\right)> A+10\log n\label{eq:1st}
\end{align}
or \begin{align}
    \min_{\theta\in[-\pi,\pi)}\sum_{k=M_*}^{e^n-1}\left(\Re \frac{X_kr^ke^{ik\theta}}{\sqrt{k}}-\mu_k\right)< -A-L_1n.\label{eq:2nd}
\end{align}
Applying the union bound, we can bound the probability of the first event \eqref{eq:1st} by
\begin{align*}
    &\P\left(\exists \log M_*\le n\le\log K:\max_{\theta\in[-\pi,\pi)}\sum_{k=M_*}^{e^n-1}\left(\Re\frac{X_kr^ke^{ik\theta}}{\sqrt{k}}-\mu_k\right)>A+10\log n\right)\\ 
    &\hspace{3cm}\le \sum_{n= \log M_*}^{\log K}\P\left(\max_{\theta\in[-\pi,\pi)}\sum_{k=M_*}^{e^n-1}\left(\Re \frac{X_kr^ke^{ik\theta}}{\sqrt{k}}-\mu_k\right)> A+10\log n\right)=:\sum_{n=\log M_*}^{\log K}\mathcal{P}_n.
\end{align*} The quantity $\mathcal{P}_n$ is then controlled using a chaining argument. First, we discretize the interval $[-\pi,\pi)$ into $ne^n$ points. Define $\theta_j = 2\pi j/(ne^n),\,0\le j< ne^n$ (and identifying $\theta$ with $\theta-2\pi$). Then on the event in the definition of $\mathcal{P}_n$, it holds that either 
\begin{equation}
    \sum_{k=M_*}^{e^n-1}\left(\Re \frac{X_kr^ke^{ik\theta_j}}{\sqrt{k}}-\mu_k\right)\ge \frac{A}{2}+5\log n
    \label{SZ-Eq-6.2}
\end{equation}
for some $0\le j<ne^n$ or \[\Re \sum_{k=M_*}^{e^n-1}\frac{X_kr^k}{\sqrt{k}}(e^{ik\theta}-e^{ik\theta_j})=\Re \int_{\theta_j}^\theta\sum_{k=M_*}^{e^n-1}X_kr^k(i\sqrt{k}e^{iky})\d y\ge \frac{A}{2}+5\log n\]
for some $0\le j<ne^n$ and some $\theta\in[\theta_j,\theta_{j+1})$. In particular, the second case implies
\begin{equation}
    \int_{\theta_j}^{\theta_{j+1}}\Big|\sum_{k=M_*}^{e^n-1}X_kr^k\sqrt{k}e^{iky}\Big|\,\d y \ge \frac{A}{2}+5\log n
    \label{SZ-Eq-6.3}
\end{equation}
for some $0\le j< ne^n$. Therefore, by the rotational invariance of $X_k$, it holds that
\begin{equation*}
    \mathcal{P}_n\le ne^n(\mathcal{P}_n'+\mathcal{P}_n''),
\end{equation*}
where $\mathcal{P}'_n$ (resp.~$\mathcal{P}_n''$) is the probability that \eqref{SZ-Eq-6.2} (resp.~\eqref{SZ-Eq-6.3}) holds with $j=0$.

For $\mathcal{P}_n'$, using Markov's inequality and Lemma \ref{SZ-lemma2.2!} we have 
\begin{align*}
    \mathcal{P}_n' &\le \exp\left(-2\Big(\frac{A}{2}+5\log n+\sum_{k=M_*}^{e^n-1}\mu_k\Big)\right)\E\left[\exp\left(2\sum_{k=M_*}^{e^n-1}\frac{r^k}{\sqrt{k}}R_k\cos(\tau_k)\right)\right]\\
    &\ll \exp(-A-10\log n-2n)e^n= \frac{e^{-n-A}}{n^{10}}.
\end{align*}
Next, we estimate $\mathcal{P}_n''$. Applying in order Jensen's inequality and  Markov's inequality, we have for $\beta:=e^{-n}$,
\begin{align*}
    \mathcal{P}_n'' &\le \P\left(\frac{1}{\theta_1}\int_0^{\theta_1}\exp\Big(\beta\Big|\sum_{k=M_*}^{e^n-1}X_kr^k\sqrt{k}e^{iky}\Big|\Big)\d y\ge \exp\Big(\frac{A+10\log n}{2\theta_1}\beta\Big)\right)\\
    &\leq \exp\left(-\frac{A+10\log n}{2\theta_1}\beta\right)\E\left[\frac{1}{\theta_1}\int_0^{\theta_1}\exp\Big(\beta\Big|\sum_{k=M_*}^{e^n-1}X_kr^k\sqrt{k}e^{iky}\Big|\Big)\d y\right]\\
    &\le \exp\left(-\frac{A+10\log n}{2\theta_1}\beta\right)\E\left[\exp\Big(\beta\Big|\sum_{k=M_*}^{e^n-1}X_kr^k\sqrt{k}\Big|\Big)\right]\\
    &\ll \exp\left(-\frac{A+10\log n}{2\theta_1}\beta\right)\E\left[\exp\Big(2\beta\Big|\sum_{k=M_*}^{e^n-1}R_k\cos(\tau_k)r^k\sqrt{k}\Big|\Big)\right],
\end{align*}
where in the last step we have used that by separating the real and imaginary parts and the fact that the law of $X_k$ is rotational invariant:
\begin{align*}
    \E\left[\exp\Big(\beta\Big|\sum_{k=M_*}^{e^n-1}X_kr^k\sqrt{k}\Big|\Big)\right]&\leq \E\left[\exp\Big(\beta\Big|\sum_{k=M_*}^{e^n-1}R_k\cos(\tau_k)r^k\sqrt{k}\Big|+\beta\Big|\sum_{k=M_*}^{e^n-1}R_k\sin(\tau_k)r^k\sqrt{k}\Big|\Big)\right]\\
    &\leq \E\left[\exp\Big(2\beta\Big|\sum_{k=M_*}^{e^n-1}R_k\cos(\tau_k)r^k\sqrt{k}\Big|\Big)\right]+\E\left[\exp\Big(2\beta\Big|\sum_{k=M_*}^{e^n-1}R_k\sin(\tau_k)r^k\sqrt{k}\Big|\Big)\right]\\
    &\ll \E\left[\exp\Big(2\beta\Big|\sum_{k=M_*}^{e^n-1}R_k\cos(\tau_k)r^k\sqrt{k}\Big|\Big)\right],
\end{align*}
where in the second inequality we used $ab\leq a^2+b^2$. 
Since $R_k$ are i.i.d.~sub-exponential, so is $\sum_{k=M_*}^{e^n-1}R_k\cos(\tau_k)r^k\sqrt{k}$. Then by Bernstein's inequality, there is some  constant $C>0$ depending on the distribution of $R_1$ such that  \[\P\left(\Big|\sum_{k=M_*}^{e^n-1}R_k\cos(\tau_k)r^k\sqrt{k}\Big|\ge u\right)\ll \exp\left(-\frac{1}{C}\min\{\frac{u^2}{e^{2n}},\frac{u}{e^{n/2}}\}\right),\] which implies \[\E\left[\exp\Bigg(2\beta\Big|\sum_{k=M_*}^{e^n-1}R_k\cos(\tau_k)r^k\sqrt{k}\Big|\Bigg)\right]\ll \exp(C\beta^2e^{2n})+\frac{1}{\frac{1}{2\beta Ce^{n/2}}-1}.\]
Finally, plugging in $\beta=e^{-n}$ we get the desired bound  $\mathcal{P}_n''\ll e^{-A-n}/n^{10}$ (which is loose but sufficient). Combining this with the bound of $\mathcal{P}_n'$, we get $\mathcal{P}_n\ll e^{-A}/n^9$, and therefore \[\P\left(\forall 1\le n\le\log K:\max_{\theta\in[-\pi,\pi)}\sum_{k=M_*}^{e^n-1}\left(\Re\frac{X_kr^ke^{ik\theta}}{\sqrt{k}}-\mu_k\right)>A+10\log n\right)\ll e^{-A}.\]
By replacing $R_k$ with $-R_k$ and also $\mu_k$ with $-\mu_k$, the probability of the lower tail event \eqref{eq:2nd} can be bounded by the same argument. This concludes the proof of Proposition \ref{SZ-prop5.2}.

\subsubsection{Proof of Proposition \ref{SZ-prop5.3}}\label{sec:332}

Let $M =M(A)= \max\{2\sqrt{A},20,\log M_*\}$, and define \[A_\theta(M):=\sum_{k=M_*}^{e^M-1}\left(\Re \frac{X_kr^ke^{ik\theta}}{\sqrt{k}}-\mu_k\right).\] Define the event \[\mathcal{E}_r(A,\theta;M):=\{-A-L_1M\le A_\theta(M)\le A+10\log M\},\]
and for $B>0$, define the event 
\begin{align}
    \mathcal{L}_r(B,\theta;K):=\bigg\{\forall M< n\le \log K,~-B-L_1n\le\sum_{e^M\le k<e^n}\bigg(\Re \frac{X_kr^ke^{ik\theta}}{\sqrt{k}}-\mu_k\bigg)\le B+10\log n\bigg\}.\label{eq:LR}
\end{align}
 Then we can replace the event $\mathcal{G}_r(A,\theta;K)$ by the less restricted event $\mathcal{E}_r(A,\theta;M)\cap\mathcal{L}_r(2A+L_1M,\theta;K)$, where we used $10\log x\le 2x$ when $x\ge 20$. By rotational symmetry and independence,
\begin{align*}\E[\bone_{\mathcal{G}_r(A,\theta;K)}|F_{K,M_*}(re^{i\theta})|^2]    &=\E[\bone_{\mathcal{G}_r(A,0;K)}|F_{K,M_*}(r)|^2]\\
    &\le \E\left[\bone_{\mathcal{E}_r(A,0;M)}\exp\Big(2\Re \sum_{k=M_*}^{e^M-1}\frac{X_kr^k}{\sqrt{k}}\Big)\bone_{\mathcal{L}_r(2A+L_1M,0;K)}\exp\Big(2\Re \sum_{k=e^M}^{K-1}\frac{X_kr^k}{\sqrt{k}}\Big)\right]\\
    &= \E\left[\bone_{\mathcal{E}_r(A,0;M)}\exp\Big(2\Re \sum_{k=M_*}^{e^M-1}\frac{X_kr^k}{\sqrt{k}}\Big)\right] \E\left[\bone_{\mathcal{L}_r(2A+L_1M,0;K)}\exp\Big(2\Re \sum_{k=e^M}^{K-1}\frac{X_kr^k}{\sqrt{k}}\Big)\right].
\end{align*}

\begin{proposition}
    \label{Cond-Expectation}
    With the notations above and let $B:=2A+L_1M$, we have 
    \begin{align}
        \E\left[\bone_{\mathcal{L}_r(B,0;K)}\exp\Bigg(2\Re\sum_{k=e^M}^{K-1}\frac{X_kr^k}{\sqrt{k}}\Bigg)\right]\ll \frac{K}{e^M}\frac{B}{\sqrt{\log(K/e^M)}}.\label{eq:Eub}
    \end{align}
\end{proposition}

\begin{proof}[Deducing Proposition \ref{SZ-prop5.3} from Proposition \ref{Cond-Expectation}] Note that $B=2A+L_1M=O(\sqrt{\log K})$. By definition of $M$ and Proposition \ref{Cond-Expectation}, we have
\begin{align*}
    \E[\bone_{\mathcal{G}_r(A,\theta;K)}|F_{K,M_*}(re^{i\theta})|^2] &\ll \frac{K(A+M)}{e^M\sqrt{\log(K/e^M)}}\E\left[\bone_{\mathcal{E}_r(A,0;M)}\exp\Big(2\Re\sum_{k=M_*}^{e^M-1}\frac{r^k}{\sqrt{k}}X_k\Big)\right]\\
    &\ll \frac{AK}{e^M\sqrt{\log K}}\times \E\left[\exp\Big(2\Re\sum_{k=M_*}^{e^M-1}\frac{r^k}{\sqrt{k}}X_k\Big)\right] \ll \frac{AK}{\sqrt{\log K}},
\end{align*}
where the last step is due to Lemma \ref{SZ-lemma2.2!}. This concludes the proof.
\end{proof}

\subsubsection{Proof of Proposition \ref{Cond-Expectation}}\label{sec:slicing}


Recall our definition of $\Q^{(1)}=\Q^{(1)}_{r,M,K}$ from \eqref{eq:Q1 def} and the event $\mathcal{L}_r(B,0;K)$ from \eqref{eq:LR}. The constant $M_*$ is chosen depending on the law of $R_1$ so that for all $k\geq M_*$, every Laplace functional below (i.e.~expectations of the form $\E[\exp(a_k R_k\cos(\tau_k))]$ for some $a_k$) is well defined, and the conditions in Lemmas \ref{SZ-lemma2.2} and \ref{lemma-Q-computations} hold. Using Lemma \ref{SZ-lemma2.2!}, the left-hand side of \eqref{eq:Eub} becomes \[\E\left[\exp\Bigg(\sum_{k=e^M}^{K-1}\frac{2r^k}{\sqrt{k}}R_k\cos(\tau_k)\Bigg)\right]\Q^{(1)}(\mathcal{L}_r(B,0;K))\asymp \frac{K}{e^M}\Q^{(1)}(\mathcal{L}_r(B,0;K)).\]
It then suffices to calculate the ballot-type probability $\Q^{(1)}(\mathcal{L}_r(B,0;K))$. Define \begin{align}
    Y_m:=\sum_{e^{m-1}\le k< e^m}\left(\frac{r^k}{\sqrt{k}}R_k\cos(\tau_k)-\mu_k\right),\label{eq:YM def}
\end{align} 
so each $Y_m$ is centered under $\Q^{(1)}$. 
Next, we approximate these random batches $Y_m$ with centered Gaussian variables with uniformly bounded variance, which allows us to use Lemma \ref{lemma:ballot} to compute the probability \[\Q^{(1)}\bigg(\forall M< n\le\log K,~ -B-L_1n\le\sum_{m=M+1}^n Y_m\le B+10\log n\bigg).\] 

\begin{lemma}\label{Lemma-1d-GaussianApprox}
For any $r_m$ such that $|r_m|\ll m^2$, there is some centered Gaussian random variable $N_m$ with variance $\sigma_m^2=\E^{\Q^{(1)}}[Y_m^2]=\frac{1}{2}+O(e^{-m/2})$ such that \[\Q^{(1)}(r_m\le Y_m\le r_m+m^{-4})=(1+O(m^{-2}))\P(r_m\le N_m\le r_m+m^{-4}).\]
\end{lemma}

\begin{proof}[Deducing Proposition \ref{Cond-Expectation} from Lemma \ref{Lemma-1d-GaussianApprox}] We use a slicing argument as in \citep{harper2020moments}. Note that on the event $\cL_r(B,\theta;K)$, it holds $-B-L_1(n-1)\le \sum_{m=M+1}^{n-1}Y_m\le B+10\log(n-1)$ and $-B-L_1n\le\sum_{m=M+1}^n Y_m\le B+10\log n$, then there must be \[|Y_n|\le 2B+L_1n+10\log n\ll A+n,\ M< n\le \log K.\]
Since $M\geq \max\{2\sqrt{A},20\}$, we know further that $|Y_n|\ll n^2$ uniformly in $M< n\le \log K$. Then there must be some $r_n\in\mathcal{R}_n:=\{r\in n^{-4}\Z:|r|\ll n^2\}$ such that $r_n\le Y_n< r_n+n^{-4}$. In particular, \begin{align} -B-L_1n-\sum_{m=M+1}^n\frac{1}{m^4}\le \sum_{m=M+1}^n r_m\le B+10\log n,~M< n\leq \log K.\label{eq:rncondition}\end{align}
Denote by $\mathcal{D}(M,K)$ the set of all possible vectors $(r_n)_{M< n\le\log K},~r_n\in\mathcal{R}_n$ that satisfy \eqref{eq:rncondition}, then Lemma \ref{Lemma-1d-GaussianApprox} yields
\begin{align*}
    &\hspace{0.5cm}\Q^{(1)}\left(\forall M< n\le\log K: -B-L_1n\le\sum_{m=M+1}^n Y_m\le B+10\log n\right)\\
    &\le \sum_{(r_n)\in\mathcal{D}(M,K)}\prod_{m=M+1}^{\log K}\Q^{(1)}\Big(r_m\le Y_m<r_m+\frac{1}{m^4}\Big)\\
    &= \sum_{(r_n)\in\mathcal{D}(M,K)}\prod_{m=M+1}^{\log K}(1+O(m^{-2}))\,\P\Big(r_m\le N_m<r_m+\frac{1}{m^4}\Big)\\
    &\ll \sum_{(r_n)\in\mathcal{D}(M,K)}\prod_{m=M+1}^{\log K}\P\Big(r_m\le N_m<r_m+\frac{1}{m^4}\Big)\\
    &\le \P\left(\forall M< n\le \log K: -B-L_1n-\sum_{m=M+1}^n\frac{1}{m^4}\le \sum_{m=M+1}^n N_m\le B+10\log n+\sum_{m=M+1}^n\frac{1}{m^4}\right)\\
    &\ll \frac{B+10\log M}{\sqrt{\log(K/e^M)}}\ll \frac{B}{\sqrt{\log(K/e^M)}},
\end{align*}
where the last line follows from   Lemma \ref{lemma:ballot} (applied with the shift $n\mapsto n-M$ and $a=B+10\log M$) and that $B=2A+L_1M\gg M$. This completes the proof of Proposition \ref{Cond-Expectation}.\end{proof}

\begin{proof}[Proof of Lemma \ref{Lemma-1d-GaussianApprox}] That $\sigma_m^2=\E^{\Q^{(1)}}[Y_m^2]={1}/{2}+O(e^{-m/2})$ follows from Lemmas \ref{lemma:basic} and \ref{lemma-Q-computations}. We first compute the characteristic function of $Y_m$ under $\Q$. Using Lemma 3.3.19 of \citep{durrett2019probability} (with $n=2$ therein) and Lemma \ref{lemma-Q-computations}, we have
\begin{align*}
    \E^{\Q^{(1)}}\left[\exp\left(it\Big(\frac{r^k}{\sqrt{k}}R_k\cos(\tau_k)-\mu_k\Big)\right)\right] &= 1-\frac{t^2}{2}\var^{{\Q}^{(1)}}\left[\frac{r^k}{\sqrt{k}}R_k\cos(\tau_k)\right]+D_k(t)\\
    &=1-\left(\frac{r^{2k}}{4k}+O(k^{-3/2})\right)t^2+D_k(t)
\end{align*}
with some $D_k(t)$ satisfying \[|D_k(t)|\ll |t|^3\E^{\Q^{(1)}}\left[\Big|\frac{r^k}{\sqrt{k}}R_k\cos(\tau_k)-\mu_k\Big|^3\right]\ll |t|^3 k^{-3/2},\]and similarly $|D'_k(t)|\ll |t|^2 k^{-3/2}$ for $|t|\leq \sqrt{k}/C$ for some large constant $C>0$ (that depends only on the law of $X_1$). Rewriting this in exponential form, we have \[\E^{\Q^{(1)}}\left[\exp\left(it(\frac{r^k}{\sqrt{k}}R_k\cos(\tau_k)-\mu_k)\right)\right] = \exp\left(-\frac{t^2}{2}\var^{{\Q}^{(1)}}\left[\frac{r^k}{\sqrt{k}}R_k\cos(\tau_k)\right]+T_k(t)\right)\]
with again $|T_k(t)|\ll |t|^3 k^{-3/2}$ and $|T'_k(t)|\ll |t|^2 k^{-3/2}$ for $|t|\leq \sqrt{k}/C$. By independence, we then have \[\E^{\Q^{(1)}}\left[\exp(itY_m)\right] =\exp\left(-\frac{t^2\sigma_m^2}{2}+S_m(t)\right)= \exp\left(-\frac{t^2}{2}\sum_{e^{m-1}\le k<e^m}\left(\frac{r^{2k}}{2k}+O(k^{-3/2})\right)+S_m(t)\right)\]
with $|S_m(t)|\ll |t|^3 e^{-m/2}$ and $|S'_m(t)|\ll |t|^2 e^{-m/2}$ for $|t|\leq e^{m/2}/C$. Let  $\lambda_m:=r_m/\sigma_m^2\ll m^2$ by our assumption that $|r_m|\ll m^2$. Define a new measure $\widetilde{\Q}^{(1)}$ by \[\frac{\d\widetilde{\Q}^{(1)}}{\d\Q^{(1)}} = \prod_{m=M+1}^{\log K}\frac{\exp(\lambda_m Y_m)}{\E^{\Q^{(1)}}[\exp(\lambda_m Y_m)]},\]
and the characteristic function of $Y_m$ under $\widetilde{\Q}^{(1)}$ is \[\E^{\widetilde{\Q}^{(1)}}\left[\exp(itY_m)\right] = \exp\left(ir_m t-\frac{\sigma_m^2}{2}t^2+S_m(t-i\lambda_m)-S_m(-i\lambda_m)\right).\]
The intuition of why $\widetilde{\Q}^{(1)}$ is needed is explained in the last paragraph of the ($q$-UNIV) phase in Section \ref{sec:main ideas}: compared to Harper's approach in \citep{harper2020moments}, we have fewer summands in $Y_m$ for the normal approximation, and a change of measure guarantees a better quantitative normal approximation. 
Note that when $Y_m\in[r_m,r_m+m^{-4}]$, \[|\lambda_mY_m-\lambda_mr_m|\le\frac{|r_m|}{\sigma_m^2 m^4}\ll m^{-2}\qquad\text{a.s.}\] Together with the estimates on $S_m(t)$, we have
\begin{align}\label{Ym-Q-Exp-Tilting}\begin{split}
    &\hspace{0.5cm}\Q^{(1)}(r_m\le Y_m\le r_m+m^{-4})\\
    &= \E^{\Q^{(1)}}\left[e^{\lambda_mY_m}\right]\E^{\widetilde{\Q}^{(1)}}\left[e^{-\lambda_mY_m}\bone_{\{r_m\le Y_m\le r_m+m^{-4}\}}\right] \\
    &= \exp\Big(\frac{\sigma_m^2}{2}\lambda_m^2+S_m(-i\lambda_m)\Big)\exp(-r_m\lambda_m+O(m^{-2}))\,\widetilde{\Q}^{(1)}(r_m\le Y_m\le r_m+m^{-4})\\
    &= (1+O(m^{-2}))\exp\Big(-\frac{r_m^2}{2\sigma_m^2}\Big)\,\widetilde{\Q}^{(1)}(r_m\le Y_m\le r_m+m^{-4}).
\end{split}
\end{align}
For each $m$, let $\widetilde{N}_m$ be a Gaussian  random variable with mean $r_m$ and variance $\sigma_m^2$ under $\P$, i.e., with characteristic function \[\E[\exp(it\widetilde{N}_m)] = \exp\Big(ir_mt-\frac{\sigma_m^2}{2}t^2\Big).\]Applying the Berry-Esseen bound (see equation (3.4.1) of \citep{durrett2019probability}), we obtain
\begin{align*}
  &|\widetilde{\Q}^{(1)}(r_m\le Y_m\le r_m+m^{-4})-\P(r_m\le \widetilde{N}_m\le r_m+m^{-4})|\ll \int_{-e^{m/9}}^{e^{m/9}}\frac{|\E^{\widetilde{\Q}^{(1)}}[e^{itY_m}]-\E[e^{it\widetilde{N}_m}]|}{|t|}\d t+e^{-m/9}.
\end{align*}
Further, since $|S_m(t)|\ll|t|^3e^{-m/2}\ll 1$,  we have the estimate \
\begin{align*}
    \Big|\E^{\widetilde{\Q}^{(1)}}[e^{itY_m}]-\E[e^{it\widetilde{N}_m}]\Big| &\le e^{-{\sigma_m^2}t^2/2}|\exp(S_m(t-i\lambda_m)-S_m(-i\lambda_m))-1|\\
    &\ll e^{-{\sigma_m^2}t^2/2}|S_m(t-i\lambda_m)-S_m(-i\lambda_m)|\\
    &\ll e^{-{\sigma_m^2}t^2/2}\int_0^{|t|}|S'_m(s-i\lambda_m)|\,\d s\\
    &\ll e^{-{\sigma_m^2}t^2/2}e^{-m/2}(|\lambda_m|^2|t|+|t|^3).
\end{align*}
Therefore,
\begin{align*}
  |\widetilde{\Q}^{(1)}(r_m\le Y_m\le r_m+m^{-4})-\P(r_m\le \widetilde{N}_m\le r_m+m^{-4})| &\ll e^{-m/2}\int_{-e^{m/9}}^{e^{m/9}}(|\lambda_m|^2+|t|^2)e^{-{\sigma_m^2}t^2/2}\d t+e^{-m/9}\\
  &\ll (1+|r_m|^2)e^{-m/2}+e^{-m/9}\ll e^{-m/9}
\end{align*}
for any $|r_m|\ll m^2$. Note that $\P(r_m\le \widetilde{N}_m\le r_m+m^{-4})\gg m^{-4}$, we have \[\widetilde{\Q}^{(1)}(r_m\le Y_m\le r_m+m^{-4}) = (1+O(e^{-m/10}))\P(r_m\le \widetilde{N}_m\le r_m+m^{-4}),\]
Finally, a standard Gaussian computation yields
\[\P(r_m\le \widetilde{N}_m\le r_m+m^{-4})=(1+O(m^{-2}))\exp\Big(\frac{r_m^2}{2\sigma_m^2}\Big)\P(r_m\le N_m\le r_m+m^{-4}).\]
Combining with \eqref{Ym-Q-Exp-Tilting} yields the proof.
\end{proof}

\subsection{Lower bound of Proposition \ref{prop:superC-asymp-truncated-moments}}\label{sec:lb}

In this section, we prove the lower bound of Proposition \ref{prop:superC-asymp-truncated-moments}.\footnote{The proof of the lower bound does not strictly rely on the ($q$-UNIV) condition, but the bound may not be tight for the other cases.} 
We first reduce the proof to the following two propositions. 

\begin{proposition} \label{SZ-Prop-8.1} 
 Fix $q\in(0,1]$. Suppose that $\E[e^{\gamma|R_1|}]<\infty$ for some $\gamma>2q$, i.e.~($q$-UNIV) holds. Then for any $0<r<1$, \[\E[|A_{N,M_*}|^{2q}]\gg \frac{1}{N^q}\left(\E\left[\left(\int_{-\pi}^{\pi}|F_{N/2,M_*}(re^{i\theta})|^2\d\theta \right)^q\right]-\E\left[r^{Nq}\left(\int_{-\pi}^{\pi}|F_{N/2,M_*}(e^{i\theta})|^2\d\theta  \right)^q\right]\right).\]
\end{proposition}

\begin{remark}
    Proposition \ref{SZ-Prop-8.1} mirrors Proposition 8.1 of \citep{soundararajan2022model}, which focuses on $q\in[1/2,1]$ instead of $q\in(0,1]$, together with a H\"{o}lder's inequality argument for $q\in(0,1/2)$, in order to obtain uniformity of the constants in $q$. Here, we do not attempt to have the asymptotic constants in \eqref{eq:asymp LT} independent of $q$.
\end{remark}

\begin{proof}
We mainly follow the arguments in Section 9 of \citep{soundararajan2022model}. First, it follows from the same symmetrization procedure as \eqref{eq:symm} that $\E[|A_{N,M_*}|^{2q}]\geq \E[|B_{N,M_*}|^{2q}]/2$, where 
$$B_{N,M_*}:=\sum_{N/2<n\leq N}\frac{X_n}{\sqrt{n}}A_{N-n,M_*}.$$
By Khintchine's inequality, in the form of Lemma 4.1 of \citep{ledoux1991probability},\footnote{The proof therein is stated for Rademacher random variables, but the same arguments work for sub-exponential random variables, using a standard concentration bound (e.g., Lemma 8.2.1 of \citep{talagrand2014upper}).
} we have 
$$\E\left[|B_{N,M_*}|^{2q}\mid \{A_{n,M_*}\}_{1\leq n\leq N/2}\right]\gg \bigg(\sum_{N/2<n\leq N}\frac{|A_{N-n,M_*}|^2}{n}\bigg)^q\gg \bigg(\frac{1}{N}\sum_{n<N/2}|A_{n,M_*}|^2\bigg)^q.$$
    Taking expectation yields
    $$\E[|A_{N,M_*}|^{2q}]\gg \E\bigg[\bigg(\frac{1}{N}\sum_{n<N/2}|A_{n,M_*}|^2\bigg)^q\bigg].$$
Recall from \eqref{eq:FA relation} that with 
$$\widetilde{A}_{n,N,M_*}:=\sum_{\substack{\lambda\in\cP_n\\ \lambda_1\leq N/2\\ \forall 1\leq k<M_*,\,m_k(\lambda)=0}}a(\lambda),$$
it holds that $F_{N/2,M_*}(z)=\sum_{n\geq 0}\widetilde{A}_{n,N,M_*}z^n$, and hence by Parseval's identity, for $r\in(0,1]$,
\begin{align}\begin{split}
    \sum_{n<N/2}|\widetilde{A}_{n,N,M_*}|^2&\geq \sum_{n=0}^\infty |\widetilde{A}_{n,N,M_*}|^2r^{2n}-r^N\sum_{n=0}^\infty |\widetilde{A}_{n,N,M_*}|^2\\
    &=\frac{1}{2\pi}\int_{-\pi}^\pi|F_{N/2,M_*}(re^{i\theta})|^2\d\theta-\frac{r^N}{2\pi}\int_{-\pi}^\pi|F_{N/2,M_*}(e^{i\theta})|^2\d\theta.
\end{split}\label{eq:pars1}
\end{align}
Moreover, by definition $\widetilde{A}_{n,N,M_*}={A}_{n,M_*}$ for $n< N/2$. The claim then follows by applying the inequality $|z+w|^q\leq|z|^q+|w|^q$ for $q\in[0,1]$ to \eqref{eq:pars1}, and taking expectation.
\end{proof}

\begin{proposition}\label{SZ-Prop-8.2}
  Fix any $q\in(0,1]$.  Suppose that ($q$-UNIV) holds. Let $K$ be large enough. For any $e^{-1/400}\le r<1$, we have \[\E\left[\left(\int_{-\pi}^{\pi}|F_{K,M_*}(re^{i\theta})|^2\d\theta \right)^q\right]\gg \left(\frac{K_r}{1+(1-q)\sqrt{\log K_r}}\right)^q\asymp\begin{cases}
      K_r^q  (\log K_r)^{-q/2}&\text{ if }q\in(0,1);\\
        K_r^q&\text{ if }q=1,
    \end{cases}\]
    where $\log K_r$ is the largest integer such that $K_r\le\min\{{-1}/({4\log r}),K\}$.
\end{proposition}

\begin{remark}
    Here, $K_r$ is the threshold below which the variances of $Y_m$ (defined in \eqref{eq:YM def}) are comparable (say, between ${1}/{20}$ and $20$). This appears necessary to make sense of the random walk analog of \eqref{eq:Yndef} proposed in the proof sketch.
\end{remark}

\begin{proof}[Deducing the lower bound of Proposition \ref{prop:superC-asymp-truncated-moments} from Propositions \ref{SZ-Prop-8.1} and \ref{SZ-Prop-8.2}]
Consider $r=e^{-C/N}$ for a large constant $C$ to be determined. Then Proposition \ref{SZ-Prop-8.2} gives that for $N$ large enough,
\[\frac{1}{N^q}\E\left[\left(\int_{-\pi}^{\pi}|F_{N/2,M_*}(re^{i\theta})|^2\d\theta \right)^q\right]\geq \frac{1}{C_1} \left(\frac{1}{C(1+(1-q)\sqrt{\log N})}\right)^q.\]
On the other hand, applying Proposition \ref{SZ-prop3.2} gives
\[\frac{1}{N^q}\E\left[r^{Nq}\left(\int_{-\pi}^{\pi}|F_{N/2,M_*}(e^{i\theta})|^2\d\theta  \right)^q\right]\leq C_2 \left(\frac{e^{-C}}{1+(1-q)\sqrt{\log N}}\right)^q.\]
Therefore, by picking $C>0$ large enough depending on the constants $C_1,C_2$, Proposition \ref{SZ-Prop-8.1} yields
\begin{align*}
    \E[|A_{N,M_*}|^{2q}]\gg \left(\frac{1}{1+(1-q)\sqrt{\log N}}\right)^q
\end{align*}for $N$ large enough. By adjusting constants suitably, the conclusion stands for all $N$.
\end{proof}

\subsubsection{Proof of Proposition \ref{SZ-Prop-8.2}}

For any (random) subset $\mathcal{L}$ of $[-\pi,\pi)$, we use H\"older's inequality to obtain
\begin{equation}
    \E\left[\left(\int_{-\pi}^{\pi}|F_{K,M_*}(re^{i\theta})|^2\d\theta \right)^q\right]\gg \frac{\left(\E\left[\int_{\mathcal{L}}|F_{K,M_*}(re^{i\theta})|^2\d\theta \right]\right)^{2-q}}{\left(\E\left[\left(\int_{\mathcal{L}}|F_{K,M_*}(re^{i\theta})|^2\d\theta \right)^2\right]\right)^{1-q}}.
    \label{SZ-Eq-10.1}
\end{equation}
We then carefully choose this random set $\mathcal{L}$ as inspired by \citep{soundararajan2022model}. Recall \eqref{eq:muk def}.

\begin{definition}\label{def:event L}
    Fix again a universal constant $L_1>20$ from Lemma \ref{lemma:ballot}. Let $A$ be a real number with $1\le A\le\sqrt{\log K_r}$. Define $\cL(\theta)=\cL(A,\theta;K)$ as the event that for each $\log M_*\le n\le \log K_r$, one has  \[-A-L_1n\le \sum_{k=M_*}^{e^n-1}\left(\Re \frac{X_kr^ke^{ik\theta}}{\sqrt{k}}-\mu_k\right)\le A-5\log n.\]
     Also, let $\cL=\cL(A;K)$ be the random subset of $\theta\in [-\pi,\pi)$ such that $\cL(\theta)$ holds.
\end{definition}

First, we give a lower bound of the numerator of \eqref{SZ-Eq-10.1}.

\begin{lemma}
\label{lemma:lbnumerator}
For any $1\leq A\leq \sqrt{\log K_r}$ and $e^{-1/400}\le r<1$, we have 
\begin{equation*}
    \E\left[\int_{\cL}|F_{K,M_*}(re^{i\theta})|^2\d\theta \right]\gg \frac{AK_r}{\sqrt{\log K_r}}.
\end{equation*}    
\end{lemma}

\begin{proof}
The proof is similar to the proof of Proposition \ref{SZ-prop5.3} in Section \ref{sec:ub}, so we only sketch the key steps. First, observe that by rotational symmetry,
\[\E\left[\int_{\mathcal{L}}|F_{K,M_*}(re^{i\theta})|^2\d\theta \right]=\E\left[\int_{-\pi}^{\pi}\bone_{\mathcal{L}(\theta)}|F_{K,M_*}(re^{i\theta})|^2\d\theta \right]=2\pi\, \E[\bone_{\mathcal{L}(0)}|F_{K,M_*}(r)|^2].\]
It follows from Lemmas \ref{lemma:basic} and \ref{SZ-lemma2.2!}, and definition of $K_r$ that 
\[\E[\bone_{\mathcal{L}(0)}|F_{K,M_*}(r)|^2]=\E\bigg[\exp\bigg(\sum_{M_*\leq k\le K}\frac{2r^k}{\sqrt{k}}R_k\cos(\tau_k)\bigg)\bigg]\Q^{(1)}(\mathcal{L}(0))\asymp  K_r\Q^{(1)}(\mathcal{L}(0)).\]
Next, recalling the definition of $Y_m$ in \eqref{eq:YM def}, we replace ${\mathcal{L}}(0)$ by the more restrictive event 
\[\widehat{\mathcal{L}}(0):=\bigg\{\forall \log M_*<n\le\log K_r,~-L_1n\le\sum_{m=\log M_*}^n Y_m\le \min\{A,L_1n\}-5\log n\bigg\}.\]
On the event $\widehat{\mathcal{L}}(0)$, we have
\[ \bigg|\sum_{e^m\leq k<e^{m+1}}\bigg(\Re \frac{X_kr^ke^{ik\theta}}{\sqrt{k}}-\mu_k\bigg)\bigg|\ll m,\ \log M_*\leq m\leq \log K_r.\]
It then follows from Lemma \ref{Lemma-1d-GaussianApprox} and a slicing argument as in Section \ref{sec:slicing} that
\begin{align*}
    &\hspace{0.5cm}\Q^{(1)}(\mathcal{L}(0))\\
    &\ge \Q^{(1)}\left(\forall \log M_*<n\le\log K_r:-L_1n\le\sum_{m=\log M_*}^n Y_m\le \min\{A,L_1n\}-5\log n\right)\\
    &\gg \P\left(\forall \log M_*< n\leq \log K_r:-L_1n+\sum_{m=\log M_*}^n m^{-4}\leq \sum_{m=\log M_*}^n N_m\leq \min\{A,L_1n\}-5\log n-\sum_{m=\log M_*}^n m^{-4} \right)\\
    &\gg \frac{A}{\sqrt{\log K_r}},
\end{align*}where the last step follows from Lemma \ref{lemma:ballot} and by adjusting the constant $L_1$ suitably, while noting that $M_*$ is a fixed constant. This completes the proof.
\end{proof}

For the upper bound of the denominator in \eqref{SZ-Eq-10.1}, we first expand the square to get\begin{align}\begin{split}
\E\left[\left(\int_{\mathcal{L}}|F_{K,M_*}(re^{i\theta})|^2\d\theta \right)^2\right] &= \E\left[\int_{-\pi}^{\pi}\int_{-\pi}^{\pi}\bone_{\cL(\theta_1)}|F_{K,M_*}(re^{i\theta_1})|^2\bone_{\cL(\theta_2)}|F_{K,M_*}(re^{i\theta_2})|^2\d\theta _1 \d\theta _2\right]\\
    &= \int_{-\pi}^{\pi}\E\left[\bone_{\cL(0)\cap\cL(\theta)}|F_{K,M_*}(r)|^2|F_{K,M_*}(re^{i\theta})|^2\right]\d\theta .
\end{split}\label{SZ-Eq-10.2}
\end{align}

\begin{proposition}
    \label{SZ-prop-10.2}
With notations as above, for any $\theta\in[-\pi,\pi)$ and $e^{-1/400}\le r<1$, we have  \[\E\left[\bone_{\cL(0)\cap\cL(\theta)}|F_{K,M_*}(r)|^2|F_{K,M_*}(re^{i\theta})|^2\right]\ll A^2e^{2A}\frac{K_r^2}{\log K_r}\frac{\min\{K_r,2\pi/|\theta|\}}{(\log\min\{K_r,2\pi/|\theta|\})^7}.\]
\end{proposition}

\begin{proof}[Deducing Proposition \ref{SZ-Prop-8.2} from Proposition \ref{SZ-prop-10.2}] As in \citep{soundararajan2022model}, applying equations \eqref{SZ-Eq-10.1}, \eqref{SZ-Eq-10.2}, Lemma \ref{lemma:lbnumerator}, and Proposition \ref{SZ-prop-10.2} with $A:=\sqrt{\log K_r}/(1+(1-q)\sqrt{\log K_r})$ completes the proof. \end{proof}

\subsubsection{Proof of Proposition \ref{SZ-prop-10.2}}\label{sec:342}

Let us define $M=M(r,\theta)$ to be the smallest integer such that $e^M\ge\max\{\min\{10^3/|\theta|,K_r/e\}, M_*\}$. Set \begin{align*}
    A_\theta(M) &:=\Re \sum_{k=M_*}^{e^M-1}\left(\frac{X_kr^ke^{ik\theta}}{\sqrt{k}}-\mu_k\right)= \sum_{k=M_*}^{e^M-1}\left(\frac{r^k}{\sqrt{k}}R_k\cos(\tau_k+k\theta)-\mu_k\right),\ \theta\in[-\pi,\pi).
\end{align*}
Similarly as in the proof of Proposition \ref{SZ-prop5.3}, we replace the event $\cL(0)\cap\cL(\theta)$ with a less restricted event $\widetilde{\cL}$, defined by the constraints that \begin{align}
    -A-L_1M\le A_0(M),A_\theta(M)\le A-5\log M, \label{eq:tildeL1}
\end{align}
and for any $M<n\le\log K_r$, \begin{align*}
   &-A-L_1n-\max\{A_0(M),A_\theta(M),0\}\\
   &\hspace{1cm}\le \sum_{e^M\le k<e^n}\left(\Re \frac{X_kr^k}{\sqrt{k}}-\mu_k\right),\sum_{e^M\le k<e^n}\left(\Re \frac{X_kr^ke^{ik\theta}}{\sqrt{k}}-\mu_k\right)\le A-\min\{A_0(M),A_\theta(M),0\}.
\end{align*}
Also recalling our definition \eqref{eq:ztheta}, and using $\mu_k={r^{2k}}/{k}+O(k^{-3/2})$, we get
\begin{align*}
    &\hspace{0.5cm}\E\left[\bone_{\cL(0)\cap\cL(\theta)}|F_{K,M_*}(r)|^2|F_{K,M_*}(re^{i\theta})|^2\right]\\
    &\ll\exp\Big(4\sum_{k=M_*}^{e^M-1}\frac{r^{2k}}{k}\Big)\E\Bigg[\bone_{\widetilde{\mathcal L}}\,e^{2A_0(M)+2A_\theta(M)}\prod_{m=M+1}^{\log K_r}e^{2Z_0(m)+2Z_\theta(m)}\exp\Big(2\sum_{k=K_r}^{K}\frac{r^k}{\sqrt{k}}\Re(X_k+X_ke^{ik\theta})\Big)\Bigg]\\
    &\ll e^{4M}\E\left[\bone_{\widetilde{\cL}}\,e^{2A_0(M)+2A_\theta(M)}\prod_{m=M+1}^{\log K_r}e^{2Z_0(m)+2Z_\theta(m)}\right]\E\Bigg[\exp\Big(2\sum_{k=K_r}^{K}\frac{r^k}{\sqrt{k}}\Re(X_k+X_ke^{ik\theta})\Big)\Bigg].
\end{align*}
Using the inequality $ab\leq a^2+b^2$, the rotational invariance of the law of $X_k$, Lemma \ref{SZ-lemma2.2} (i), and definition of $K_r$, we arrive at
\begin{align*}
    \E\Bigg[\exp\Big(2\sum_{k=K_r}^{K}\frac{r^k}{\sqrt{k}}\Re(X_k+X_ke^{ik\theta})\Big)\Bigg] &\leq \E\Bigg[\exp\Big(4\sum_{k=K_r}^{K}\frac{r^k}{\sqrt{k}}\Re X_k\Big)\Bigg]+\E\Bigg[\exp\Big(4\sum_{k=K_r}^{K}\frac{r^k}{\sqrt{k}}\Re (X_ke^{ik\theta})\Big)\Bigg]\\
    &\ll \E\Bigg[\exp\Big(4\sum_{k=K_r}^{K}\frac{r^k}{\sqrt{k}}\Re X_k\Big)\Bigg]\\
    &\ll \prod_{k=K_r}^{K}\left(1+\frac{2r^{2k}}{k}+O(k^{-3/2})\right)\ll 1.
\end{align*}
We conclude that\begin{equation}
    \E\left[\bone_{\cL(0)\cap\cL(\theta)}|F_{K,M_*}(r)|^2|F_{K,M_*}(re^{i\theta})|^2\right]\ll e^{4M}\E\left[\bone_{\widetilde{\cL}}\,e^{2A_0(M)+2A_\theta(M)}\prod_{m=M+1}^{\log K_r}e^{2Z_0(m)+2Z_\theta(m)}\right].\label{eq:?}
\end{equation}
We now state a two-dimensional version of Proposition \ref{Cond-Expectation}, which suffices for proving Proposition \ref{SZ-prop-10.2}.

\begin{proposition}\label{SZ-prop-11.1}
    Notations as above, and let $M'=\max\{M,A\}$ and $e^{-1/K}\leq r\leq e^{1/K}$ (recall that $K\geq K_r$). Given any $B,B'$ satisfying $B'\leq 0\leq B$ and $B,-B'\le LM'$ with some absolute constant $L>0$, define the event  
    \begin{align}
        \mathcal{E}:=\Bigg\{\forall M<n\le\log K_r,\,B'-L_1n\le \sum_{e^M\le k<e^n}\left(\Re \frac{X_kr^k}{\sqrt{k}}-\nu_k\right),\sum_{e^M\le k<e^n}\left(\Re \frac{X_kr^ke^{ik\theta}}{\sqrt{k}}-\nu_k\right)\le B\Bigg\}.\label{eq:event E}
    \end{align}
    Then \[\E\bigg[\bone_{\mathcal{E}}\prod_{m=M+1}^{\log K_r}\exp(2Z_0(m)+2Z_\theta(m))\bigg]\ll \frac{K_r^2}{e^{2M}}\left(\frac{M'}{\sqrt{1+\log(K_r/e^{M'})}}\right)^2.\]
\end{proposition}


\begin{proof}[Deducing Proposition \ref{SZ-prop-10.2} from Proposition \ref{SZ-prop-11.1}.]
\sloppy Our plan is to  condition on $\{X_k\}_{1\leq k<e^M}$ and insert Proposition \ref{SZ-prop-11.1} into \eqref{eq:?}. First we claim that, with a constant $L_0>0$ large enough, setting $B=A-\min\{A_0(M),A_\theta(M),0\}+L_0$ and $B'=-A-\max\{A_0(M),A_\theta(M),0\}-L_0$  in the definition of event $\mathcal{E}$ gives that $\widetilde{\cL}\subseteq\mathcal{E}$. Indeed, this is a consequence of \eqref{eq:A52} of Lemma \ref{lemma:basic 2}, with 
$$L_0=\sum_{m\geq M}\bigg|\sum_{e^{m-1}\leq k<e^m}(\mu_k-\nu_k)\bigg|\ll 1.$$
Moreover, we have by \eqref{eq:tildeL1}  that on the event $\widetilde{\cL}$, $B-B'\le L M'$ and $M'\ll A+M$ with some large enough absolute constant $L>0$. Together with \eqref{eq:?}, this yields that
\begin{align*}
    \E\left[\bone_{\cL(0)\cap\cL(\theta)}|F_{K,M_*}(r)|^2|F_{K,M_*}(re^{i\theta})|^2\right]
    &\ll \frac{K_r^2e^{2M}}{1+\log(K_r/e^{M'})}\E\left[\bone_{\widetilde{\cL}}\,e^{2(A_0(M)+A_\theta(M))}(A+M)^2\right]\\
    &\ll \frac{K_r^2e^{2M}}{1+\log(K_r/e^{M'})}\E\left[\bone_{\widetilde{\cL}}\,e^{2(A_0(M)+A_\theta(M))}A^2M^2\right]\\
    &\ll \frac{K_r^2e^{2M}}{1+\log(K_r/e^{M'})}\frac{e^{2A}A^2}{M^{8}}\E\left[\bone_{\{A_0(M)\leq A-5\log M\}}\,e^{2A_0(M)}\right],
\end{align*}
where in the last step we used the definition \eqref{eq:tildeL1}. On the other hand, Lemma \ref{SZ-lemma2.2} yields 
\begin{align}
    \E\left[\bone_{\{A_0(M)\leq A-5\log M\}}\,e^{2A_0(M)}\right]\leq \E[e^{2A_0(M)}]\ll e^{-M}.\label{eq:A0M}
\end{align}
Next, using $M'=\max\{M,A\}\leq \max\{M,\sqrt{\log K_r}\}$ we obtain $(1+\log(K_r/e^{M'}))M\gg\log K_r$.
Combining the above and using the definition of $M$ completes the proof.
\end{proof}

\subsubsection{Proof of Proposition \ref{SZ-prop-11.1}}

Suppose first that $\log(K_r/e^M)\le 10$. Using rotational symmetry, \eqref{eq:ztheta}, and Lemma \ref{SZ-lemma2.2} (i),
\begin{align*}
    \E\bigg[\bone_{\mathcal{E}}\prod_{m=M+1}^{\log K_r}\exp(2Z_0(m)+2Z_\theta(m))\bigg]&\ll \E\bigg[\prod_{m=M+1}^{\log K_r}\exp(4Z_0(m))\bigg]=\E\bigg[\exp\Big(4\Re\sum_{e^M\leq k< K_r}\frac{X_kr^k}{\sqrt{k}}\Big)\bigg]\ll 1.
\end{align*}
Now, if $K_r> e^{M+10}$, we can assume that $\theta$ satisfies $10^3/|\theta|\le K_r/e$ and $e^M|\theta|\ge 10^3$. 
Recall \eqref{eq:dq2}. In view of Lemma \ref{lemma:basic3},  it suffices to bound $\Q^{(2)}(\mathcal{E})$. Recall from \eqref{eq:nuk def} that \[\nu_k=\E^{\Q^{(2)}}\left[\Re \frac{X_kr^k}{\sqrt{k}}\right] = \frac{r^{2k}}{k}+\frac{\cos(k\theta)r^{2k}}{k}+O(k^{-3/2}),\]
and define \[X_{k,0}:=\Re \frac{X_kr^k}{\sqrt{k}}-\nu_k\qquad\text{ and }\qquad X_{k,\theta}:=\Re \frac{X_kr^ke^{ik\theta}}{\sqrt{k}}-\nu_k,\qquad k\in\N.\] 
We apply the same strategy as in the proof of Proposition \ref{Cond-Expectation}: approximate the batched sums of $X_{k,\theta}$ using Gaussians, and apply Lemma \ref{lemma:ballot}. We first use Lemma \ref{lemma-Q-computations} to compute the joint characteristic function of $(X_{k,0},X_{k,\theta})$  as 
\begin{align*}
&\hspace{0.5cm}\E^{\Q^{(2)}}[\exp(iuX_{k,0}+ivX_{k,\theta})]\\
&=1+\frac{-u^2\E^{\Q^{(2)}}[X_{k,0}^2]-v^2\E^{\Q^{(2)}}[X_{k,\theta}^2]-2uv \E^{\Q^{(2)}}[X_{k,0}X_{k,\theta}]}{2}+\sum_{j=3}^\infty\frac{1}{j!}\E^{\Q^{(2)}}[(iu X_{k,0}+ivX_{k,\theta})^j]\\
&=1-\frac{(u^2+v^2)r^{2k}}{4k}-\frac{uvr^{2k}\cos(k\theta)}{2k}+O\Big(\frac{(u^2+v^2)r^{3k}}{k^{3/2}}\Big)+\sum_{j=3}^\infty\frac{1}{j!}\E^{\Q^{(2)}}[(iu X_{k,0}+ivX_{k,\theta})^j]\\
&=: 1-\frac{(u^2+v^2)r^{2k}}{4k}-\frac{uvr^{2k}\cos(k\theta)}{2k}+D_{k}(u,v),
\end{align*}
and using Lemma 3.3.19 of \citep{durrett2019probability} and Lemma \ref{lemma-Q-computations}, we have for $u,v$ satisfying $|u|,|v|\leq \sqrt{k}/C$,
\begin{enumerate}[(a)]
    \item $|D_{k}(u,v)|\ll (u^2+v^2+|u|^3+|v|^3){r^{3k}}{k^{-3/2}}$;
    \item $|\frac{\partial D_{k}(u,v)}{\partial u}|,|\frac{\partial D_{k}(u,v)}{\partial v}|\ll (|u|+|v|+|u|^2+|v|^2){r^{3k}}{k^{-3/2}} $;
    \item $|\frac{\partial^2 D_{k}(u,v)}{\partial u\partial v}|\ll (1+|u|+|v|){r^{3k}}{k^{-3/2}}$.
\end{enumerate}
Before using independence to form a product of the characteristic functions over $k$, we need to transform our expression into an exponential form. We have
$$\E^{\Q^{(2)}}[\exp(iuX_{k,0}+ivX_{k,\theta})]=\exp\left(-\frac{(u^2+v^2)r^{2k}}{4k}-\frac{uvr^{2k}\cos(k\theta)}{2k}+T_k(u,v)\right),$$
where the above (a)--(c) hold with $D_k(u,v)$ replaced by $T_k(u,v)$.

Next, we group the random variables $X_{k,0},X_{k,\theta}$ and approximate their sums using Gaussian distributions. Define
\[Y_{m,0}=\sum_{e^{m-1}\leq k<e^m}X_{k,0}\ ~~\text{ and }~~\ Y_{m,\theta}=\sum_{e^{m-1}\leq k<e^m}X_{k,\theta},\qquad m\in\N.\]After summing over $e^{m-1}\leq k<e^m$,
\begin{align}
\E^{\Q^{(2)}}\left[\exp\left(iuY_{m,0}+ivY_{m,\theta}\right)\right]&=\exp\left(\sum_{e^{m-1}\leq k<e^m}\Big(-\frac{(u^2+v^2)r^{2k}}{4k}-\frac{uvr^{2k}\cos(k\theta)}{2k}\Big)+S_m(u,v)\right),\label{eq:xm0char}
\end{align}
where for $u,v$ satisfying $|u|,|v|\leq e^{m/2}/C$,
\begin{enumerate}[(a)]
    \item $|S_m(u,v)|\ll (u^2+v^2+|u|^3+|v|^3)e^{-m/2}\ll (1+|u|^3+|v|^3)e^{-m/2}$;
    \item $|\frac{\partial S_m(u,v)}{\partial u}|,|\frac{\partial S_m(u,v)}{\partial v}|\ll (|u|+|v|+|u|^2+|v|^2)e^{-m/2} $;
    \item $|\frac{\partial^2 S_m(u,v)}{\partial u\partial v}|\ll (1+|u|+|v|)e^{-m/2}$.
\end{enumerate}
Our Gaussian approximation will have the same covariance structure of $(Y_{m,0},Y_{m,\theta})$ under $\Q^{(2)}$, which we compute first. 
Define the covariance matrix \[\Sigma_m:=\begin{pmatrix}\sigma_m^2 & \rho_m\sigma_m^2 \\ \rho_m\sigma_m^2 & \sigma_m^2\end{pmatrix},\]
where $\sigma_m,\rho_m$ are defined such that $\Sigma_m$ is the covariance matrix of $(Y_{m,0},Y_{m,\theta})$ under $\Q^{(2)}$. It follows from Lemmas \ref{lemma-Q-computations} and \eqref{eq:A51} of \ref{lemma:basic 2} that for $m\in[M,\log K_r]$,  \begin{align}
    \sigma_m^2=\sum_{e^{m-1}\leq k<e^m}\left(\frac{r^{2k}}{2k}+O(k^{-3/2})\right)=\frac{1}{2}+O(e^{-m/2})\label{eq:sm2}
\end{align} and 
\begin{align}
    \rho_m\sigma_m^2=\sum_{e^{m-1}\leq k<e^m}\left(\frac{r^{2k}\cos(k\theta)}{2k}+O(k^{-3/2})\right)=O(e^{M-m}+e^{-m/2}).\label{eq:rm2}
\end{align}
In particular, for $m\in[M,\log K_r]$,
\begin{align}
    \rho_m\ll \frac{e^{M-m}+e^{-m/2}}{\frac{1}{2}+O(e^{-m/2})}\ll e^{M-m}+e^{-m/2}.\label{eq:rhom}
\end{align}
Let $\mathbf{N}_m :=(N_{m,1},N_{m,2})$ be a two-dimensional centered Gaussian vector with covariance matrix $\Sigma_m$. That is,
\begin{align}
    \E [e^{i\mathbf{x}\cdot{\mathbf{N}}_m }] = \exp\left(-\frac{1}{2}\mathbf{x}^\top\Sigma_m\mathbf{x}\right),\ \mathbf{x}\in\R^2.\label{eq:Nm1}
\end{align}

\begin{lemma}\label{2pt-GaussApprox}
    For any $|u_m|,|v_m|\ll m$, it holds that
\begin{align*}
&\hspace{0.5cm}\Q^{(2)}(u_m\leq Y_{m,0}\leq u_m+m^{-3},\ v_m\leq Y_{m,\theta}\leq v_m+m^{-3})\\
&\hspace{3cm}=(1+O(m^{-2}))\P(u_m\leq N_{m,1}\leq u_m+m^{-3},\ v_m\leq N_{m,2}\leq v_m+m^{-3}).
\end{align*}
\end{lemma}

\begin{proof}[Deducing Proposition \ref{SZ-prop-11.1} from Lemma \ref{2pt-GaussApprox}]We use a slicing argument  to bound $\Q^{(2)}(\mathcal{E})$. A direct argument as in the proof of Proposition \ref{Cond-Expectation} would not work properly because here, we do not have a lower bound of $M$ in terms of $A$ (cf.~the condition $M\geq 2\sqrt{A}$ in the setting of Section \ref{sec:332}), which implies that we do not have a bound similar to $|Y_m|\ll m^2,~m\in(M,\log K_r]$ therein. For this reason, we condition on the values $\{Y_{m,0},Y_{m,\theta}\}_{M<m\leq M'}$ and consider the ballot event with partial sums of $\{Y_{m,0},Y_{m,\theta}\}_{M'<m\leq \log K_r}$ where $M'=\max\{M,A\}$. Observe that by definition \eqref{eq:event E}, on the event $\mathcal{E}$, 
\[B'-L_1M'\leq\sum_{M<j\leq M'}Y_{j,0},\sum_{M<j\leq M'} Y_{j,\theta}\leq B,\] and $|Y_{m,0}|,|Y_{m,\theta}|\ll m$ for  $M'<m\leq \log K_r$, using $M'\gg  B-B'$. Therefore,
\begin{align}
    \Q^{(2)}(\mathcal{E})&=\E\Big[\Q^{(2)}(\mathcal{E}\mid \{Y_{m,0},Y_{m,\theta}\}_{M<m\leq M'})\Big]\nonumber\\
    &\leq \Q^{(2)}\Bigg(\forall M'<m\leq\log K_r,~B'-L_1m-B\leq \sum_{M'<j\leq m}Y_{j,0},\sum_{M'<j\leq m} Y_{j,\theta}\leq B-B'+L_1M'\nonumber\\
    &\hspace{5cm}\text{and}\quad\forall M'<m\leq \log K_r,~|Y_{m,0}|,|Y_{m,\theta}|\ll m\Bigg).\label{eq:QM'}
\end{align}
Given a set of values $\{Y_{m,0},Y_{m,\theta}\}_{M<m\leq \log K_r}$ satisfying the event in \eqref{eq:QM'}, there exist numbers $u_m,v_m\in \mathcal{S}_m:=\{t\in(1/m^3)\Z:|t|\ll m\},\, M'< m\leq \log K_r$ such that for all $M'< m\leq \log K_r$,
\begin{align*}
    u_m&\leq Y_{m,0}<u_m+m^{-3};\ v_m\leq Y_{m,\theta}<v_m+m^{-3},
\end{align*} and in particular,
\begin{align}
     B'-L_1m-B-\sum_{M'<j\leq m}m^{-3}&\leq \sum_{M'<j\leq m}u_j,\,\sum_{M'<j\leq m}v_j\leq B-B'+L_1M'.\label{eq:umvm2}
\end{align}Denote by $\mathcal{C}(M',K_r)$ the set of all possible vectors $(u_m,v_m)_{M'\leq m\leq \log K_r},~ u_m,v_m\in \mathcal{S}_m$ satisfying  \eqref{eq:umvm2}. Using Lemma \ref{2pt-GaussApprox}  and \eqref{eq:QM'}, we have 
\begin{align*}
     \Q^{(2)}(\mathcal{E})&\leq \Q^{(2)}\Bigg(\forall M'<m\leq\log K_r,~B'-L_1m-B\leq \sum_{M'<j\leq m}Y_{j,0},\sum_{M'<j\leq m} Y_{j,\theta}\leq B-B'+L_1M'\\
    &\hspace{3cm}\text{and }\forall M'<m\leq \log K_r,~|Y_{m,0}|,|Y_{m,\theta}|\ll m\Bigg)\\
    &\leq \sum_{(u_m,v_m)\in \mathcal{C}(M',K_r)}\Q^{(2)}\left(Y_{m,0}\in[u_m,u_m+\frac{1}{m^3}],\ Y_{m,\theta}\in[v_m,v_m+\frac{1}{m^3}]\text{ for all }M'<m\leq \log K_r\right)\\
    &\leq \sum_{(u_m,v_m)\in \mathcal{C}(M',K_r)}\prod_{M'<m\leq \log K_r}(1+O(m^{-2}))\\
    &\hspace{3cm}\P\left(N_{m,1}\in[u_m,u_m+\frac{1}{m^3}],\ N_{m,2}\in[v_m,v_m+\frac{1}{m^3}]\text{ for all }M'<m\leq \log K_r\right)\\
    &\ll     \P\Bigg(B'-L_1m-2\sum_{j=1}^m \frac{1}{j^3}-B\le \sum_{M'< \ell\leq m}N_{\ell,1},\ \sum_{M'< \ell\leq m}N_{\ell,2}\leq B+2\sum_{j=1}^m \frac{1}{j^3}-B'+L_1M'\\
    &\hspace{6cm}\text{ for all }M'<m\leq \log K_r\Bigg)\\
    &\leq \P\left(-Lm\le \sum_{M'< \ell\leq m}N_{\ell,1},\ \sum_{M'< \ell\leq m}N_{\ell,2}\leq LM'\text{ for all }M'<m\leq \log K_r\right),
\end{align*}where in the last step we used that $B,B'\ll M'<m$ and $L$ is some other absolute constant.

Before applying the Gaussian ballot theorem, we shall decorrelate each pair of random variables  $(N_{m,1},N_{m,2})$. 
Recall from \citep[Section 12]{soundararajan2022model} that for any Borel set $B\subseteq\R^2$,
$$\P((N_{m,1},N_{m,2})\in B)\leq\sqrt{\frac{1+|\rho_m|}{1-|\rho_m|}}\,\P((\widetilde{N}_{m,1},\widetilde{N}_{m,2})\in B),$$
where 
 $\widetilde{N}_{m,1},\widetilde{N}_{m,2}$ are i.i.d.~$N(0,\sigma_m^2(1+|\rho_m|))$ distributed. Using \eqref{eq:rhom}, it is straightforward to see that \begin{align}
     \prod_{M'<m\leq \log K_r}\sqrt{\frac{1+|\rho_m|}{1-|\rho_m|}}\ll 1.\label{eq:decor}
 \end{align}
By Lemma \ref{lemma:ballot}, we conclude that 
\begin{align*}
    \Q^{(2)}(\mathcal{E})&\ll\P\left(-Lm\le\sum_{M'< \ell\leq m}N_{\ell,1},\ \sum_{M'< \ell\leq m}N_{\ell,2}\leq LM'\text{ for }M'<m\leq \log K_r\right)\\
    &\ll \P\left(-Lm\le \sum_{M'< \ell\leq m}\widetilde{N}_{\ell,1},\ \sum_{M'< \ell\leq m}\widetilde{N}_{\ell,2}\leq LM'\text{ for }M'<m\leq \log K_r\right)\\
    &= \left(\P\left(-Lm\le\sum_{M'< \ell\leq m}\widetilde{N}_{\ell,1}\leq LM'\text{ for }M'<m\leq \log K_r\right)\right)^2\\
    &\ll \left(\frac{M'}{\sqrt{1+\log(K_r/e^{M'})}}\right)^2.
\end{align*}
This yields the desired result.\end{proof}

\begin{proof}[Proof of Lemma \ref{2pt-GaussApprox}]  Let $\mathbf{Y}_m :=(Y_{m,0},Y_{m,\theta})^\top,\mathbf{u}:=(u_m,v_m)^\top$, and $\boldsymbol{\lambda}:=\Sigma_m^{-1}\mathbf{u}$, where we recall that $\Sigma_m$ is the covariance matrix of $(Y_{m,0},Y_{m,\theta})$ under $\Q^{(2)}$. First, we apply an exponential tilt of the measure so that $(Y_{m,0},Y_{m,\theta})$ is centered at $(u_m,v_m)$. Define the tilted measure $\widetilde{\Q}^{(2)}$ by \[\frac{\d\widetilde{\Q}^{(2)}}{\d\mathbb{Q}^{(2)}}:=\frac{\exp(\boldsymbol{\lambda}\cdot\mathbf{Y}_m )}{\Q^{(2)}(\exp(\boldsymbol{\lambda}\cdot\mathbf{Y}_m ))}.\]
Let
$\widehat{\mathbf{N}}_m :=(\widehat{N}_{m,1},\widehat{N}_{m,2})$ be a two-dimensional Gaussian vector with the same mean and covariance matrix as $\mathbf{Y}_m $ under $\widetilde{\Q}^{(2)}$. By \eqref{eq:Nm1} and since exponential tilts preserves the covariance for Gaussians, we have \[\E [e^{i\mathbf{x}\cdot\widehat{\mathbf{N}}_m }] = \exp\left(-\frac{1}{2}\mathbf{x}^\top\Sigma_m\mathbf{x}+i\mathbf{u}\cdot\mathbf{x}\right),\ \mathbf{x}\in\R^2.\]
The corresponding characteristic function of $\mathbf{Y}_m $ under the tilted measure $\widetilde{\Q}^{(2)}$ is given by
\begin{equation}
    \E^{\widetilde{\Q}^{(2)}}[e^{i\mathbf{x}\cdot\mathbf{Y}_m }] = \exp\left(-\frac{1}{2}\mathbf{x}^\top\Sigma_m\mathbf{x}+i\mathbf{u}\cdot\mathbf{x}+S_m(\mathbf{x}-i\boldsymbol{\lambda})-S_m(-i\boldsymbol{\lambda})\right),\ \mathbf{x}\in\R^2,
    \label{Qym0-char}
\end{equation}where we recall $S_m$ from \eqref{eq:xm0char}. 
Following \citep{sadikova1966two}, we have 
\begin{align}
\begin{split}
&\hspace{0.5cm}\Big|\widetilde{\Q}^{(2)}(u_m\leq Y_{m,0}\leq u_m+m^{-3},\ v_m\leq Y_{m,\theta}\leq v_m+m^{-3})\\
&\hspace{3cm}-\P(u_m\leq \widehat{N}_{m,1}\leq u_m+m^{-3},\ v_m\leq \widehat{N}_{m,2}\leq v_m+m^{-3})\Big|\\
&\ll \int_{-e^{m/9}}^{e^{m/9}}\int_{-e^{m/9}}^{e^{m/9}}\left|\frac{\Delta(s,t)}{st}\right|\,\d s\,\d t+\int_{-e^{m/9}}^{e^{m/9}}\bigg|\frac{\E^{\widetilde{\Q}^{(2)}}[\exp(isY_{m,0})]-\E[\exp({is\widehat{N}_{m,1}})]}{s}\bigg|\,\d s\\
&\hspace{2cm}+\int_{-e^{m/9}}^{e^{m/9}}\bigg|\frac{\E^{\widetilde{\Q}^{(2)}}[\exp(itY_{m,\theta})]-\E[\exp({it\widehat{N}_{m,2}})]}{t}\bigg|\,\d t+e^{-m/9},
\end{split}\label{eq:BerryEsseen}
\end{align}
where \begin{align*}
\Delta(s,t)&:=\E^{\widetilde{\Q}^{(2)}}\left[\exp\left(isY_{m,0}+itY_{m,\theta}\right)\right]-\E[\exp(is\widehat{N}_{m,1}+it\widehat{N}_{m,2})]\\
&\hspace{1cm}-\E^{\widetilde{\Q}^{(2)}}[\exp(isY_{m,0})]\,\E^{\widetilde{\Q}^{(2)}}[\exp(itY_{m,\theta})]+\E[\exp({is\widehat{N}_{m,1}})]\,\E[\exp({it\widehat{N}_{m,2}})].
\end{align*}
For $\mathbf{x}_1=(s,0)$, we estimate
\begin{align}\label{S}
\begin{split}    
    |S_m(\mathbf{x}_1-i\boldsymbol{\lambda})-S_m(-i\boldsymbol{\lambda})| &= \left|\int_0^{s}\frac{\d S_m(t-i\lambda_1,-i\lambda_2)}{\d t}\d t\right|\\
    &\ll e^{-m/2}\int_0^{|s|}(1+|t|^2+\n{\boldsymbol{\lambda}}^2)\d t\ll e^{-m/2}(1+\n{\mathbf{u}}^2)(|s|+|s|^3),
\end{split}
\end{align}
where we used $\n{\boldsymbol{\lambda}}\le \n{\Sigma_m^{-1}}_{\mathrm{op}}\n{\mathbf{u}}=(\sigma_m^2-|\rho_m\sigma_m^2 |)^{-1}\n{\mathbf{u}}\ll\n{\mathbf{u}}$, since $\sigma_m^2\gg 1 $ and $ \rho_m\sigma_m^2 =o(1)$ by \eqref{eq:sm2} and \eqref{eq:rm2}.
Inserting back into \eqref{Qym0-char}, the second integral of \eqref{eq:BerryEsseen} can be bounded by
\begin{align*}
    \int_{-e^{m/9}}^{e^{m/9}}\bigg|\frac{\E^{\widetilde{\Q}^{(2)}}[\exp(isY_{m,0})]-\E[\exp({is\widehat{N}_{m,1}})]}{s}\bigg|\,\d s&\leq \int_{-e^{m/9}}^{e^{m/9}}\left|\frac{e^{S_m(\mathbf{x}_1-i\boldsymbol{\lambda})-S_m(-i\boldsymbol{\lambda})}-1}{s}\right|e^{-{\sigma_m^2}s^2/2}\d s\\
    &\ll \int_{-e^{m/9}}^{e^{m/9}}\left|\frac{S_m(\mathbf{x}_1-i\boldsymbol{\lambda})-S_m(-i\boldsymbol{\lambda})}{s}\right|e^{-{\sigma_m^2}s^2/2}\d s\\
    &\leq e^{-m/2}(1+u_m^2+v_m^2)\int_{-e^{m/9}}^{e^{m/9}}(1+s^2)e^{-{\sigma_m^2}s^2/2}\d s\\
    &\ll e^{-m/2}m^2\ll e^{-m/9}.
\end{align*} 
A similar estimate holds for the third term. We now bound the first integral of \eqref{eq:BerryEsseen}, where we consider $\mathbf{x}_2=(0,t)$ and $\mathbf{x}_3=(s,t)$ in \eqref{Qym0-char}:
\begin{align*}
|\Delta(s,t)|&= \Big|\exp\Big(-\frac{\sigma_m^2}{2}\n{\mathbf{x}_3}^2+i\mathbf{u}\cdot\mathbf{x}_3\Big)\Big[\exp(-\rho_m\sigma_m^2  st+S_m(\mathbf{x}_3-i\boldsymbol{\lambda})-S_m(-i\boldsymbol{\lambda}))+1 \\
&\qquad -\exp\Big(-\rho_m\sigma_m^2  st\Big)-\exp(S_m(\mathbf{x}_1-i\boldsymbol{\lambda})+S_m(\mathbf{x}_2-i\boldsymbol{\lambda})-2S_m(-i\boldsymbol{\lambda}))\Big]\Big|\\
&\le \exp\Big(-\frac{\sigma_m^2}{2}\n{\mathbf{x}_3}^2\Big)\Bigg[\Big|e^{-\rho_m\sigma_m^2  st}-1\Big|\Big|e^{S_m(\mathbf{x}_3-i\boldsymbol{\lambda})-S_m(-i\boldsymbol{\lambda})}-1\Big|\\
&\qquad +\Big|e^{S_m(\mathbf{x}_3-i\boldsymbol{\lambda})-S_m(-i\boldsymbol{\lambda})}\Big|\Big|e^{S_m(\mathbf{x}_1-i\boldsymbol{\lambda})+S_m(\mathbf{x}_2-i\boldsymbol{\lambda})-S_m(\mathbf{x}_3-i\boldsymbol{\lambda})-S_m(-i\boldsymbol{\lambda})}-1\Big|\Bigg]\\
&\ll \exp\Big(-\frac{\sigma_m^2}{2}\n{\mathbf{x}_3}^2\Big)\Big[|st||S_m(\mathbf{x}_3-i\boldsymbol{\lambda})-S_m(-i\boldsymbol{\lambda})|\\
&\hspace{2cm}+|S_m(\mathbf{x}_1-i\boldsymbol{\lambda})+S_m(\mathbf{x}_2-i\boldsymbol{\lambda})-S_m(\mathbf{x}_3-i\boldsymbol{\lambda})-S_m(-i\boldsymbol{\lambda})|\Big]
\end{align*}
where we used \eqref{eq:rm2} in the last step. Similarly to \eqref{S}, we have 
\begin{align*}
    |S_m(\mathbf{x}_3-i\boldsymbol{\lambda})-S_m(-i\boldsymbol{\lambda})| &\le |S_m(\mathbf{x}_3-i\boldsymbol{\lambda})-S_m(\mathbf{x}_1-i\boldsymbol{\lambda})|+|S_m(\mathbf{x}_1-i\boldsymbol{\lambda})-S_m(-i\boldsymbol{\lambda})|\\
    &\ll e^{-m/2}(1+s^2+\n{\mathbf{u}}^2)(|t|+|t|^3)+e^{-m/2}(1+\n{\mathbf{u}}^2)(|s|+|s|^3)
\end{align*}
and also
\begin{align*}
    |S_m(\mathbf{x}_1-i\boldsymbol{\lambda})+S_m(\mathbf{x}_2-i\boldsymbol{\lambda})-S_m(\mathbf{x}_3-i\boldsymbol{\lambda})-S_m(-i\boldsymbol{\lambda})| &= \left|\int_0^s\int_0^t\frac{\partial^2 S_m(u-i\lambda_1,v-i\lambda_2)}{\partial u \partial v}\d u\d v\right|\\
    &\ll e^{-m/2}(1+\n{\boldsymbol{\lambda}}_1)(1+|s|+|t|)|st|\\
    &\ll m e^{-m/2}(1+|s|+|t|)|st|.
\end{align*}
Therefore,
\begin{align*}
    \int_{-e^{m/9}}^{e^{m/9}}\int_{-e^{m/9}}^{e^{m/9}}\left|\frac{\Delta(s,t)}{st}\right|\,\d s\d t &\ll e^{-m/2} \int_{-e^{m/9}}^{e^{m/9}}\int_{-e^{m/9}}^{e^{m/9}}\exp\Big(-\frac{\sigma_m^2}{2}\n{\mathbf{x}_3}^2\Big)\Big[(1+s^2+\n{\mathbf{u}}^2)(|t|+|t|^3)\\
    &\qquad +(1+\n{\mathbf{u}}^2)(|s|+|s|^3)+m(1+|s|+|t|)|st|\Big]\d s\d t\\
    &\ll m^2 e^{-m/2}\ll e^{-m/9}.
\end{align*}
Plugging back to \eqref{eq:BerryEsseen}, we obtain
\begin{align*}
    &\hspace{0.5cm}\Big|\widetilde{\Q}^{(2)}(u_m\leq Y_{m,0}\leq u_m+m^{-3},\ v_m\leq Y_{m,\theta}\leq v_m+m^{-3})\\
&\hspace{3cm}-\P(u_m\leq \widehat{N}_{m,1}\leq u_m+m^{-3},\ v_m\leq \widehat{N}_{m,2}\leq v_m+m^{-3})\Big|\ll e^{-m/9}
\end{align*}
uniformly for all $|u_m|,|v_m|\ll m,\, M\le m\le\log K_r$. We may replace the absolute error by a multiplier of $(1+O(e^{-m/10}))$, since $(\widehat{N}_{m,1},\widehat{N}_{m,2})$ is centered at $(u_m,v_m)$ and has constant order variances with vanishing correlation. This gives that for $|u_m|,|v_m|\ll m$ and $M\le m\le\log K_r$,
\begin{align*}
&\hspace{0.5cm}\widetilde{\Q}^{(2)}(u_m\leq Y_{m,0}\leq u_m+m^{-3},\ v_m\leq Y_{m,\theta}\leq v_m+m^{-3})\\
&\hspace{3cm}=(1+O(e^{-m/10}))\P(u_m\leq \widehat{N}_{m,1}\leq u_m+m^{-3},\ v_m\leq \widehat{N}_{m,2}\leq v_m+m^{-3}).
\end{align*}
Now note that if $Y_{m,0}\in[u_m,u_m+m^{-3}] $ and $Y_{m,\theta}\in[v_m,v_m+m^{-3}]$, it holds that $|\boldsymbol{\lambda}\cdot\boldsymbol{Y}_m-\boldsymbol{\lambda}\cdot\boldsymbol{u}|\ll {\n{\boldsymbol{u}}}/{m^3}\ll m^{-2}.$ Therefore, we have
\begin{align*}
    &\Q^{(2)}(u_m\leq Y_{m,0}\leq u_m+m^{-3},\ v_m\leq Y_{m,\theta}\leq v_m+m^{-3})\\
    = &\E^{\Q^{(2)}}[e^{\boldsymbol{\lambda}\cdot\mathbf{Y}_m }]\,\widetilde{\Q}^{(2)}\Big(e^{-\boldsymbol{\lambda}\cdot\mathbf{Y}_m }\bone_{\{u_m\leq Y_{m,0}\leq u_m+m^{-3},\ v_m\leq Y_{m,\theta}\leq v_m+m^{-3}\}}\Big)\\
    = &\exp\Big(-\frac{1}{2}\mathbf{u}^\top\Sigma_m^{-1}\mathbf{u}+S_m(-i\boldsymbol{\lambda})\Big)(1+O(m^{-2}))\,\widetilde{\Q}^{(2)}(u_m\leq Y_{m,0}\leq u_m+m^{-3},\ v_m\leq Y_{m,\theta}\leq v_m+m^{-3})\\
    = &(1+O(m^{-2}))\exp\Big(-\frac{1}{2}\mathbf{u}^\top\Sigma_m^{-1}\mathbf{u}\Big)(1+O(e^{-m/10}))\,\P(u_m\leq \widehat{N}_{m,1}\leq u_m+m^{-3},\ v_m\leq \widehat{N}_{m,2}\leq v_m+m^{-3}).
\end{align*}
Thus,
\begin{align*}
    &\Q^{(2)}(u_m\leq Y_{m,0}\leq u_m+m^{-3},\ v_m\leq Y_{m,\theta}\leq v_m+m^{-3})\\
    &\hspace{1cm}=(1+O(m^{-2}))\exp\Big(-\frac{1}{2}\mathbf{u}^\top\Sigma_m^{-1}\mathbf{u}\Big)\,\P(u_m\leq \widehat{N}_{m,1}\leq u_m+m^{-3},\ v_m\leq \widehat{N}_{m,2}\leq v_m+m^{-3}).
\end{align*}
By a standard Gaussian computation,
\begin{align*}
    &\P(u_m\leq N_{m,1}\leq u_m+m^{-3},\ v_m\leq N_{m,2}\leq v_m+m^{-3})\\
    &\hspace{1cm}=(1+O(m^{-2}))\exp\Big(-\frac{1}{2}\mathbf{u}^\top\Sigma_m^{-1}\mathbf{u}\Big)\,\P(u_m\leq \widehat{N}_{m,1}\leq u_m+m^{-3},\ v_m\leq \widehat{N}_{m,2}\leq v_m+m^{-3}).
\end{align*} The proof is then complete.
\end{proof}

\section{The stretched exponential phase}
\label{sec:subcritical}
\sloppy Recall in the stretched exponential case (SE), $\P(|R_k|\ge u)=\exp(-(u/c_p)^p)$ with $0<p<1$ and $c_p=(2\Gamma(2/p)/p)^{-1/2}$. 
Denoting by  $\lambda^*=(1,\cdots,1)$ the all-one partition of $N$, Theorem \ref{thm:main} in the (SE) case follows from the following proposition, which shows that the contribution from $\lambda^*$ alone is the main term in low moments of $A_N$.
\begin{proposition}
    For all $N$ large enough and any $q>0$, \begin{equation}
        \E\left[|A_N|^{2q}\right]=(1+o(1)) \E\left[|a(\lambda^*)|^{2q}\right]=(1+o(1)) (2\pi)^{{1}/{2}-q}\sqrt{\frac{2q}{p}}C_{p,q}^{{2qN}/{p}}N^{2q({1}/{p}-1)N+{1}/{2}-q},\label{eq:sub-critical}
    \end{equation} where $C_{p,q}=2qc_p^p/(pe^{1-p})$.
\end{proposition}\begin{proof}
Recall that $a(\lambda^*)=X_1^N/N!$.     The asymptotic formula of $\E\left[|a(\lambda^*)|^{2q}\right]$ then follows from \eqref{u}. For the first asymptotic relation, by concavity, we have for $0<q<1/2$, \begin{equation*}
        \E\left[|a(\lambda^*)|^{2q}\right]-\sum_{\substack{|\lambda|=N\\ m_1<N}}\E\left[|a(\lambda)|^{2q}\right]\le \E\left[|A_N|^{2q}\right]\le \E\left[|a(\lambda^*)|^{2q}\right]+\sum_{\substack{|\lambda|=N\\ m_1<N}}\E\left[|a(\lambda)|^{2q}\right],
    \end{equation*} and by Minkowski's inequality, for $q\geq 1/2$,  \begin{align*}
        \E\left[|a(\lambda^*)|^{2q}\right]^{1/(2q)}-\sum_{\substack{|\lambda|=N\\ m_1<N}}\E\left[|a(\lambda)|^{2q}\right]^{1/(2q)}&\le \E\left[|A_N|^{2q}\right]^{1/(2q)}\le \E\left[|a(\lambda^*)|^{2q}\right]^{1/(2q)}+\sum_{\substack{|\lambda|=N\\ m_1<N}}\E\left[|a(\lambda)|^{2q}\right]^{1/(2q)}.
    \end{align*} Therefore it suffices to show for $0<q<1/2$, \begin{equation}
        \sum_{\substack{|\lambda|=N\\ m_1<N}}\E\left[|a(\lambda)|^{2q}\right]=o\Big(\E\left[|a(\lambda^*)|^{2q}\right]\Big)\label{eqn-HT-q-small-o}
    \end{equation} and for $q\geq 1/2$, \begin{equation}
        \sum_{\substack{|\lambda|=N\\ m_1<N}}\E\left[|a(\lambda)|^{2q}\right]^{1/(2q)}=o\Big(\E\left[|a(\lambda^*)|^{2q}\right]^{1/(2q)}\Big)\label{eqn-HT-q-large-o}
    \end{equation} as $N\to\infty$. Fix $q\in(0,1/2)$, by \eqref{eq:gamma} we have \begin{align}
        \sum_{\substack{|\lambda|=N\\ m_1<N}}\E\left[|a(\lambda)|^{2q}\right] &\leq \sum_{k=1}^N\sum_{\substack{|\lambda'|=k\\m_1(\lambda')=0}}\E\left[\left|\frac{X_1^{N-k}}{(N-k)!}a(\lambda')\right|^{2q}\right]\nonumber\\
        &\ll \sum_{k=1}^NC_{p,q}^{{2q}(N-k)/p}(N-k)^{2q({1}/{p}-1)(N-k)+1/2-q}\sum_{\substack{|\lambda'|=k\\m_1(\lambda')=0}}\E\left[|a(\lambda')|^{2q}\right].\label{ineqn:HT-q-small-subterm}
    \end{align} For each $\lambda'$ as a partition of $k\in\{1,\dots,N\}$ with $m_1(\lambda')=0$, we have with some constant $c=c_{p,q}>0$ (which may vary from line to line) that \begin{align*}
        \E\left[|a(\lambda')|^{2q}\right] &\ll \prod_{j=2}^k cC_{p,q}^{{2q}m_j/p}m_j^{2q({1}/{p}-1)m_j+1}\ll c^{{k}}\Big(\frac{k}{2}\Big)^{q({1}/{p}-1)k+\sqrt{2k}},
    \end{align*} where in the last inequality we use the fact that $|\lambda'|=k=\sum_{j=2}^k jm_j\ge 2\sum_{j=2}^k m_j$, and that $\sum_{j=2}^k\bone_{\{m_j>0\}}\le \sqrt{2k}$. Recall the partition number $p_k\ll e^{\pi\sqrt{2/3}\sqrt{k}}$, we can bound \eqref{ineqn:HT-q-small-subterm} from above by \begin{align*}
        \sum_{\substack{|\lambda|=N\\ m_1<N}}\E\left[|a(\lambda)|^{2q}\right] &\ll \sum_{k=1}^N C_{p,q}^{{2q}(N-k)/p}(N-k)^{2q({1}/{p}-1)(N-k)+1/2-q}e^{\pi\sqrt{2/3}\sqrt{k}}c^{{k}}\Big(\frac{k}{2}\Big)^{q({1}/{p}-1)k+\sqrt{2k}}\\
        &\ll \E\left[|a(\lambda^*)|^{2q}\right]\sum_{k=1}^{\infty}C_{p,q}^{-{2q}k/p}e^{c{k}}\frac{(N-k)^{2q({1}/{p}-1)(N-k)+1/2-q}(k/2)^{q({1}/{p}-1)k+\sqrt{2k}}}{N^{2q({1}/{p}-1)N+{1}/{2}-q}}\\
        &\ll \E\left[|a(\lambda^*)|^{2q}\right]\sum_{k=1}^{\infty}\exp\left(ck-q\Big(\frac{1}{p}-1\Big)(2N\log N-2(N-k)\log(N-k)-k\log k)\right)\\
        &\ll \E\left[|a(\lambda^*)|^{2q}\right]\sum_{k=1}^{\infty}\exp\left(ck-q\Big(\frac{1}{p}-1\Big)k\log N\right)\\
        &\ll \E\left[|a(\lambda^*)|^{2q}\right]\sum_{k=1}^{\infty}N^{-ck}=o\left(\E\left[|a(\lambda^*)|^{2q}\right]\right),
    \end{align*} establishing \eqref{eqn-HT-q-small-o} and hence \eqref{eq:sub-critical}. A similar computation proves \eqref{eqn-HT-q-large-o} and hence the result follows.
\end{proof}

\section{The exponential phase}\label{sec:critical}
Let us recall that in the exponential phase (EXP), $\P(|R_k|\geq u)=\exp(-\gamma(u-c_\gamma))\wedge 1$ for $u\geq 0$, where $\gamma\in(0,2q]$ and $c_\gamma=\log({\gamma^2}/{2})/\gamma$. Our goal is to provide asymptotics for $\E[|A_N|^{2q}]$ as $N\to\infty$ where $q\in(0,1]$. To this end, we split into two cases: $\gamma<2q$ and $\gamma=2q$.

\subsection{The case \texorpdfstring{$\gamma<2q$}{}}
In this section, we prove Theorem \ref{thm:main2} for $\gamma<2q$, that
\begin{align}
     \E[|A_N|^{2q}]\asymp \phi(N)=\phi_{q,\gamma}(N):=
    N^{1/2-q}\left(\frac{2q}{\gamma}\right)^{2qN},\quad q\in(0,1],~\gamma<2q.\label{eq:to show gamma<2q}
\end{align}

\subsubsection{Proof of the upper bound}\label{sec:critical-gamma<2q-UB}
A direct application of Minkowski's inequality together with \eqref{eq:moments of SE} yields for $q\geq 1/2$,
\begin{align*}
    \E[|A_N|^{2q}]^{1/(2q)}&\leq \sum_{\lambda\in\cP_N}\E\Bigg[\bigg|\prod_{k\geq 1}\left(\frac{X_k}{\sqrt{k}}\right)^{m_k}\frac{1}{m_k!}\bigg|^{2q}\Bigg]^{1/(2q)} \\
    &\ll \sum_{\lambda\in\cP_N}\prod_{\substack{k\geq 1\\ m_k\neq 0}}\frac{C\gamma^{-m_k}\Gamma(2m_kq+1)^{1/(2q)}}{k^{m_k/2}m_k!}\\
    &\ll \frac{\Gamma(2Nq+1)^{1/(2q)}}{\gamma^{N}N!}+ \sum_{j=2}^N \sum_{\substack{\lambda\in \cP_N\\ m_1(\lambda)=N-j}}\frac{\Gamma(2(N-j)q+1)^{1/(2q)}}{\gamma^{N-j}(N-j)!}\prod_{\substack{k\geq 2\\ m_k\neq 0}}\frac{C\Gamma(2m_kq+1)^{1/(2q)}}{\gamma^{m_k}k^{m_k/2}m_k!}.
\end{align*}
By \eqref{eq:gamma},
$$\frac{\Gamma(2m_kq+1)^{1/(2q)}}{m_k!}\leq C(2q)^{m_k}m_k^{1/(4q)-1/2}.$$
On the other hand, 
$$\prod_{\substack{k\geq 2\\ m_k\neq 0}}\frac{Cm_k^{1/(4q)-1/2}}{k^{m_k/2}}\ll1, $$
\sloppy since the product on $k\geq C^2$ can be bounded by one and the product on $2\leq k\leq C^2$ is bounded by a constant, where $C$ may depend on $q$. Therefore, for $\lambda\in \cP_N$ with $m_1(\lambda)=N-j$, we have
$$\prod_{\substack{k\geq 2\\ m_k\neq 0}}\frac{C\gamma^{-m_k}\Gamma(2m_kq+1)^{1/(2q)}}{k^{m_k/2}m_k!}\ll \prod_{\substack{k\geq 2\\ m_k\neq 0}}\left(\frac{2q}{\gamma}\right)^{m_k}\leq \left(\frac{2q}{\gamma}\right)^{j/2}.$$
Combining the above leads to 
\begin{align*}
    &\hspace{0.5cm}\E[|A_N|^{2q}]^{1/(2q)}\\
    &\ll \frac{\gamma^{-N}\Gamma(2Nq+1)^{1/(2q)}}{N!}+\sum_{j=2}^N \sum_{\substack{\lambda\in \cP_N\\ m_1(\lambda)=N-j}}\frac{\gamma^{-(N-j)}\Gamma(2(N-j)q+1)^{1/(2q)}(2q/\gamma)^{j/2}}{(N-j)!}\\
    &\ll {\gamma^{-N}(2q)^NN^{1/(4q)-1/2}}+\sum_{j=2}^N p_j{\gamma^{-(N-j/2)}(2q)^{N-j/2}(N-j)^{1/(4q)-1/2}}\\
    &\ll {\gamma^{-N}(2q)^NN^{1/(4q)-1/2}}\left(1+\sum_{j=2}^N e^{C\sqrt{j}}{\Big(\frac{\gamma}{2q}\Big)^{j/2}\Big(\frac{N-j}{N}\Big)^{1/(4q)-1/2}}\right)\\
    &\ll {\left(\frac{2q}{\gamma}\right)^NN^{1/(4q)-1/2}}=\phi(N)^{1/(2q)},
\end{align*}where we have used \eqref{eq:gamma} again. This leads to the upper bound in \eqref{eq:to show gamma<2q}. 
The case $q\in(0,1/2)$ is similar by using concavity instead of Minkowski's inequality.

\subsubsection{Proof of the lower bound}\label{sec:critical-gamma<2q-LB}
Recall from \eqref{eq:to show gamma<2q} that our goal is to show $\E[|A_N|^{2q}]\gg\phi(N)$. We may always assume that $N$ is large enough. Observe the decomposition
\begin{align}
    A_N=\sum_{\substack{\lambda\in \cP_N\\ m_1(\lambda)\geq  N-C}}\prod_{k\geq 1}\left(\frac{X_k}{\sqrt{k}}\right)^{m_k}\frac{1}{m_k!}+\sum_{\substack{\lambda\in \cP_N\\ m_1(\lambda)< N-C}}\prod_{k\geq 1}\left(\frac{X_k}{\sqrt{k}}\right)^{m_k}\frac{1}{m_k!}\label{eq:AN decomp}
\end{align}for $C\geq 0$. 
A careful examination of the proof of the upper bound yields the following lemma, giving an upper bound for the second term in \eqref{eq:AN decomp}.
\begin{lemma}\label{lemma:tail small}
    For any $\ee>0$, there exists $C=C(\ee)>0$ such that
    $$\E\Bigg[\Big|\sum_{\substack{\lambda\in \cP_N\\ m_1(\lambda)< N-C}}\prod_{k\geq 1}\left(\frac{X_k}{\sqrt{k}}\right)^{m_k}\frac{1}{m_k!}\Big|^{2q}\Bigg]\leq \ee \phi(N).$$
\end{lemma}

We will focus on the sum over partitions that have a large number of ones, and use Lemma \ref{lemma:tail small} to show that the remaining terms are negligible. To this end, let us fix a large constant $C_*>0$ to be determined. For a given $\lambda\in\cP_N$ with $m_1(\lambda)\geq N-C_*$, we let  $\lambda_*$ denote the partition after removing $N-C_*$ ones from $\lambda$. In particular, $\lambda_*\in\cP_{C_*}$ and there is a bijective correspondence between such $\lambda$ and $\lambda_*$.

Consider $M=M(\gamma,q)\in\N$ such that \begin{align}
    \Big(\frac{2q}{\gamma}\Big)^{2M}>\Big(\frac{2}{\gamma}\Big)^2.\label{M}
\end{align} We will restrict to the event that $|X_1|\in[(2qN-\sqrt{N})/\gamma, 2qN/\gamma]$ and $|X_2|,\dots,|X_M|\leq 1$. More precisely, we have
\begin{align}
\begin{split}
    \E[|A_N|^{2q}]&\geq \E\Big[|A_N|^{2q}\bone_{\{|R_1|\in [(2qN-\sqrt{N})/\gamma, 2qN/\gamma],\, |R_j|\leq 1,\, 2\leq j\leq M\}}\Big]\\
    &\gg \E\left[|A_N|^{2q}\bone_{\{|R_1|\in [(2qN-\sqrt{N})/\gamma, 2qN/\gamma]\}}\mid {\{|R_j|\leq 1,\, 2\leq j\leq M\}}\right].
\end{split}    \label{eq:conditioning}
\end{align}
Therefore, in the following, we may without loss of generality condition on the event $\{|R_j|\leq 1,\, 2\leq j\leq M\}$. The resulting probability measure and expectation operator will be denoted by $\widetilde{\P}$ and $\widetilde{\E}$ respectively.

We first make a few simplifications to the first sum of \eqref{eq:AN decomp}. It holds that
\begin{align}
    \Bigg|\sum_{\substack{\lambda\in \cP_N\\ m_1(\lambda)\geq  N-C_*}}\prod_{k\geq 1}\left(\frac{X_k}{\sqrt{k}}\right)^{m_k(\lambda)}\frac{1}{m_k(\lambda)!}\Bigg|=|X_1|^{N-C_*}\Big|\sum_{\lambda_*\in \cP_{C_*}}\frac{m_1(\lambda_*)!}{(m_1(\lambda_*)+N-C_*)!}\prod_{k\geq 1}\left(\frac{X_k}{\sqrt{k}}\right)^{m_k(\lambda_*)}\frac{1}{m_k(\lambda_*)!}\Big|.\label{eq:extract X_1}
\end{align}Here, with the bijective correspondence between $\lambda$ and $\lambda_*$ described above, we have $m_k(\lambda)=m_k(\lambda_*)$ except when $k=1$. When the situation is clear, we omit writing the dependence on $\lambda$ or $\lambda_*$. 
With $|X_1|\approx 2qN/\gamma$, we expect that 
$$\sum_{\lambda_*\in \cP_{C_*}}\frac{m_1!}{(m_1+N-C_*)!}\prod_{k\geq 1}\left(\frac{X_k}{\sqrt{k}}\right)^{m_k}\frac{1}{m_k!}\approx \sum_{\lambda_*\in \cP_{C_*}}\left(\frac{2qNe^{i\tau_1}}{\gamma}\right)^{m_1}\frac{1}{(m_1+N-C_*)!}\prod_{k\geq 2}\left(\frac{X_k}{\sqrt{k}}\right)^{m_k}\frac{1}{m_k!}.$$
The goal of the next lemma is to make this precise.
\begin{lemma}\label{lemma:X1 approx}
On the event $|R_1|=|X_1|\in [(2qN-\sqrt{N})/\gamma, 2qN/\gamma]$,    it holds that
  \begin{align*}
    &\hspace{0.5cm}\Big|\sum_{\lambda_*\in \cP_{C_*}}\frac{m_1!}{(m_1+N-C_*)!}\prod_{k\geq 1}\left(\frac{X_k}{\sqrt{k}}\right)^{m_k}\frac{1}{m_k!}- \sum_{\lambda_*\in \cP_{C_*}}\left(\frac{2qNe^{i\tau_1}}{\gamma}\right)^{m_1}\frac{1}{(N-C_*)!N^{m_1}}\prod_{k\geq 2}\left(\frac{X_k}{\sqrt{k}}\right)^{m_k}\frac{1}{m_k!}\Big|\\
    &\ll \sum_{\lambda_*\in \cP_{C_*}}\frac{(2q/\gamma)^{m_1}m_1C_*}{\sqrt{N}(N-C_*)!}\,\prod_{k\geq 2}\left(\frac{X_k}{\sqrt{k}}\right)^{m_k}\frac{1}{m_k!}.
\end{align*}
\end{lemma}
\begin{proof}
Consider $\lambda_*\in\cP_{C_*}$. By triangle inequality,
\begin{align*}
    &\hspace{0.5cm}\left|\frac{(N-C_*)!}{(m_1+N-C_*)!}\,X_1^{m_1}-N^{-m_1}e^{i\tau_1m_1}\left(\frac{2qN}{\gamma}\right)^{m_1}\right|\\
    &\leq \left|\frac{(N-C_*)!}{(m_1+N-C_*)!}-N^{-m_1}\right|\cdot|X_1|^{m_1}+N^{-m_1}\Big||R_1|^{m_1}-\left(\frac{2qN}{\gamma}\right)^{m_1}\Big|.
\end{align*}
Using  Taylor's expansion and that $m_1\leq C_*$, we have for $N$ large that
$$\bigg|1-\prod_{j=1}^{m_1}\frac{N}{N-C_*+j}\bigg|\leq \exp\left(\sum_{j=1}^{m_1}\frac{C_*-j}{N-C_* +j}\right)-1\ll \frac{m_1C_*}{N-C_*}. $$
In addition, by the mean-value theorem
$$\Bigg||R_1|^{m_1}-\left(\frac{2qN}{\gamma}\right)^{m_1}\Bigg|\ll m_1\sqrt{N}\left(\frac{2qN}{\gamma}\right)^{m_1-1}.$$
Altogether, we have
\begin{align}
    \left|\frac{(N-C_*)!}{(m_1+N-C_*)!}\,X_1^{m_1}-N^{-m_1}e^{i\tau_1m_1}\left(\frac{2qN}{\gamma}\right)^{m_1}\right|\ll \frac{(2q/\gamma)^{m_1}m_1C_*}{\sqrt{N}}.\label{eq:triangle-bound}
\end{align}
Next, we use \eqref{eq:triangle-bound} and the triangle inequality to obtain
\begin{align*}
    &\hspace{0.5cm}\Big|\sum_{\lambda_*\in \cP_{C_*}}\frac{m_1!}{(m_1+N-C_*)!}\prod_{k\geq 1}\left(\frac{X_k}{\sqrt{k}}\right)^{m_k}\frac{1}{m_k!}- \sum_{\lambda_*\in \cP_{C_*}}\left(\frac{2qNe^{i\tau_1}}{\gamma}\right)^{m_1}\frac{1}{(N-C_*)!N^{m_1}}\prod_{k\geq 2}\left(\frac{X_k}{\sqrt{k}}\right)^{m_k}\frac{1}{m_k!}\Big|\\
    &\leq \sum_{\lambda_*\in \cP_{C_*}}\Big|\frac{X_1^{m_1}}{(m_1+N-C_*)!}- \left(\frac{2qNe^{i\tau_1}}{\gamma}\right)^{m_1}\frac{1}{(N-C_*)!N^{m_1}}\Big|\,\prod_{k\geq 2}\left(\frac{X_k}{\sqrt{k}}\right)^{m_k}\frac{1}{m_k!}\\
    &\ll \sum_{\lambda_*\in \cP_{C_*}}\frac{(2q/\gamma)^{m_1}m_1C_*}{\sqrt{N}(N-C_*)!}\,\prod_{k\geq 2}\left(\frac{X_k}{\sqrt{k}}\right)^{m_k}\frac{1}{m_k!},
\end{align*}
    as desired.
\end{proof}

\begin{lemma}\label{lemma:uniformly large}
  There exists $\delta>0$ such that for $C$ large enough,
    \begin{align}
        \widetilde{\P}\left(\left|\sum_{\lambda_*\in\cP_C}\prod_{k\geq 2}\left(\frac{X_k}{\sqrt{k}(2qe^{i\tau_1}/\gamma)^k}\right)^{m_k}\frac{1}{m_k!}\right|>\delta\right)>\delta.\label{eq:unilarge}
    \end{align}
\end{lemma}

\begin{proof}
    Denote by
    $$\xi_C:=\sum_{\lambda_*\in\cP_C}\prod_{k\geq 2}\left(\frac{X_k}{\sqrt{k}(2qe^{i\tau_1}/\gamma)^k}\right)^{m_k}\frac{1}{m_k!}.$$
    We show that $\xi_C\overset{L^2}{\to}\xi$ for some random variable $\xi\neq 0$ as $C\to\infty$. Since a partition $\lambda_*\in\cP_C$ is uniquely determined by the vector $\bm=\bm(\lambda_*):=(m_2(\lambda_*),\dots,m_C(\lambda_*))$, we can rewrite $\xi_C$ as
$$\xi_C=\sum_{\substack{\bm\in \N_0^{\{2,\dots,C\}}\\ \sum_{j=2}^Cjm_j\leq C}}\prod_{k\geq 2}\left(\frac{X_k}{\sqrt{k}(2qe^{i\tau_1}/\gamma)^k}\right)^{m_k}\frac{1}{m_k!},$$where $\N_0=\{0,1,2,\dots\}$. 
This motivates us to compute the tail $L^2$ norm using \eqref{eq:moments of SE}, that
\begin{align}\widetilde{\E}\Bigg[\bigg|\sum_{\substack{\bm\in \N_0^{\{2,3,\dots\}}\\ \sum_{j=2}^\infty jm_j> C}}\prod_{k\geq 2}\left(\frac{X_k}{\sqrt{k}(2qe^{i\tau_1}/\gamma)^k}\right)^{m_k}\frac{1}{m_k!}\bigg|^2\Bigg]    &=\sum_{\substack{\bm\in \N_0^{\{2,3,\dots\}}\\ \sum_{j=2}^\infty jm_j> C}}\prod_{k\geq 2}\left(\frac{1}{\sqrt{k}(2q/\gamma)^k}\right)^{2m_k}\frac{1}{(m_k!)^2}\widetilde{\E}[|X_k|^{2m_k}]\label{f}\\
&\leq \sum_{\substack{\bm\in \N_0^{\{2,3,\dots\}}\\ \sum_{j=2}^\infty jm_j> C}}\prod_{\substack{k\geq M\\ m_k\neq 0}}\frac{L'\gamma^{-2m_k}}{k^{m_k}(2q/\gamma)^{2km_k}} \binom{2m_k}{m_k}\prod_{2\leq k<M}\frac{1}{(m_k!)^2},\nonumber
\end{align}
where we recall that we conditioned on the event $\{|R_j|\leq 1,\, 2\leq j\leq M\}$ and that $\E[e^{im\tau_k}\overline{e^{in\tau_k}}]=0$ for $m\neq n$. 
Suppose we have proved that
\begin{align}
    \sum_{\substack{\bm\in \N_0^{\{2,3,\dots\}}\\ \sum_{j=2}^\infty jm_j> C}}\prod_{\substack{k\geq M\\ m_k\neq 0}}\frac{L'\gamma^{-2m_k}}{k^{m_k}(2q/\gamma)^{2km_k}} \binom{2m_k}{m_k}\prod_{2\leq k<M}\frac{1}{(m_k!)^2}\to 0.\label{eq:L2 limit}
\end{align}
Then it follows that $\xi_C\overset{L^2}{\to}\xi$ where 
$$\widetilde{\E}[|\xi|^2]=\widetilde{\E}\Bigg[\bigg|\sum_{\substack{\bm\in \N_0^{\{2,3,\dots\}}\\ }}\prod_{\substack{k\geq M\\ m_k\neq 0}}\left(\frac{X_k}{\sqrt{k}(2qe^{i\tau_1}/\gamma)^k}\right)^{m_k}\frac{1}{m_k!}\bigg|^2\Bigg]\geq \frac{1}{C(q)}>0. $$Since $L^2$ convergence implies convergence in probability, \eqref{eq:unilarge} follows. 
Therefore, it suffices to establish \eqref{eq:L2 limit}. Using \eqref{M} and \eqref{eq:gamma}, we have for some $\ee=\ee(\gamma,q)>0$,
\begin{align}\sum_{\substack{\bm\in \N_0^{\{2,3,\dots\}}\\ \sum_{j=2}^\infty jm_j> C}}\prod_{\substack{k\geq M\\ m_k\neq 0}}\frac{L'\gamma^{-2m_k}}{k^{m_k}(2q/\gamma)^{2km_k}} \binom{2m_k}{m_k}\prod_{2\leq k<M}\frac{1}{(m_k!)^2}
\ll
\sum_{\substack{\bm\in \N_0^{\{2,3,\dots\}}\\ \sum_{j=2}^\infty jm_j> C}}\prod_{\substack{k\geq M\\ m_k\neq 0}}\frac{1}{e^{\ee km_k}}\prod_{2\leq k<M}\frac{1}{e^{\ee m_k}},\label{d}
\end{align}
where we may without loss of generality  replace $M$ by $\max(M,L')$. Note that for each $\bm\in \N_0^{\{2,3,\dots\}}$ with $\sum_{j=2}^\infty jm_j> C$, either $m_k\geq C^{1/3}$ for some $k\in[2,C^{1/3})$, or $m_k\geq 1$ for some $k\geq C^{1/3}$. Therefore, for $C>M^3$,
\begin{align*}
   \sum_{\substack{\bm\in \N_0^{\{2,3,\dots\}}\\ \sum_{j=2}^\infty jm_j> C}}\prod_{\substack{k\geq M\\ m_k\neq 0}}\frac{1}{e^{\ee km_k}}\prod_{2\leq k<M}\frac{1}{e^{\ee m_k}} &\leq  e^{-\ee C^{1/3}}\sum_{\substack{\bm\in \N_0^{\{2,3,\dots\}}}}\prod_{\substack{k\geq M}}\frac{1}{e^{\ee km_k}}\prod_{2\leq k<M}\frac{1}{e^{\ee m_k}}\\
   &\leq e^{-\ee C^{1/3}}\prod_{2\leq k<M}\left(\sum_{m_k\geq 0}\frac{1}{e^{\ee m_k}}\right)\prod_{\substack{k\geq M}}\left(\sum_{m_k\geq 0}\frac{1}{e^{\ee km_k}}\right)\ll e^{-\ee C^{1/3}}.
\end{align*}
This proves \eqref{eq:L2 limit} and hence \eqref{eq:unilarge}.
\end{proof}

In the rest of this subsection, we prove the lower bound of Theorem \ref{thm:main2} in the case $\gamma<2q$. First, combining \eqref{eq:extract X_1} and Lemma \ref{lemma:X1 approx} yields
\begin{align}
    &\hspace{0.5cm}\widetilde{\E}\Bigg[\Bigg|\sum_{\substack{\lambda\in \cP_N\nonumber\\ m_1(\lambda)\geq  N-C_*}}\prod_{k\geq 1}\left(\frac{X_k}{\sqrt{k}}\right)^{m_k(\lambda)}\frac{1}{m_k(\lambda)!}\Bigg|^{2q}\bone_{\{|R_1|\in [(2qN-\sqrt{N})/\gamma, 2qN/\gamma]\}}\Bigg]\nonumber\\
    &=\widetilde{\E}\Bigg[|X_1|^{2q(N-C_*)}\Big|\sum_{\lambda_*\in \cP_{C_*}}\frac{m_1(\lambda_*)!}{(m_1(\lambda_*)+N-C_*)!}\prod_{k\geq 1}\left(\frac{X_k}{\sqrt{k}}\right)^{m_k(\lambda_*)}\frac{1}{m_k(\lambda_*)!}\Big|^{2q}\bone_{\{|R_1|\in [(2qN-\sqrt{N})/\gamma, 2qN/\gamma]\}}\Bigg]\nonumber\\
    &\geq \widetilde{\E}\Bigg[|X_1|^{2q(N-C_*)}\bigg|\sum_{\lambda_*\in \cP_{C_*}}\left(\frac{2qNe^{i\tau_1}}{\gamma}\right)^{m_1}\frac{1}{(N-C_*)!N^{m_1}}\prod_{k\geq 2}\left(\frac{X_k}{\sqrt{k}}\right)^{m_k}\frac{1}{m_k!}\bigg|^{2q}\bone_{\{|R_1|\in [(2qN-\sqrt{N})/\gamma, 2qN/\gamma]\}}\Bigg]\nonumber\\
    &\hspace{1.8cm}-C\widetilde{\E}\Bigg[|X_1|^{2q(N-C_*)}\bigg|\sum_{\lambda_*\in \cP_{C_*}}\frac{(2q/\gamma)^{m_1}m_1C_*}{\sqrt{N}(N-C_*)!}\,\prod_{k\geq 2}\left(\frac{X_k}{\sqrt{k}}\right)^{m_k}\frac{1}{m_k!}\bigg|^{2q}\bone_{\{|R_1|\in [(2qN-\sqrt{N})/\gamma, 2qN/\gamma]\}}\Bigg],\label{eq:two terms}
\end{align}where $C$ is the implied constant in Lemma \ref{lemma:X1 approx} and in the last line we assumed $q\leq 1/2$, and the case $q\in(1/2,1]$ follows similarly by Minkowski's inequality. 
In the first expectation of \eqref{eq:two terms}, the term
$$\bigg|\sum_{\lambda_*\in \cP_{C_*}}\left(\frac{2qNe^{i\tau_1}}{\gamma}\right)^{m_1}\frac{1}{(N-C_*)!N^{m_1}}\prod_{k\geq 2}\left(\frac{X_k}{\sqrt{k}}\right)^{m_k}\frac{1}{m_k!}\bigg|^{2q}$$
is independent of the rest. This yields
\begin{align}
    &\hspace{0.5cm}\widetilde{\E}\Bigg[\bigg|\sum_{\lambda_*\in \cP_{C_*}}\left(\frac{2qNe^{i\tau_1}}{\gamma}\right)^{m_1}\frac{1}{(N-C_*)!N^{m_1}}\prod_{k\geq 2}\left(\frac{X_k}{\sqrt{k}}\right)^{m_k}\frac{1}{m_k!}\bigg|^{2q}|X_1|^{2q(N-C_*)}\bone_{\{|R_1|\in [(2qN-\sqrt{N})/\gamma, 2qN/\gamma]\}}\Bigg]\nonumber\\
    &=\widetilde{\E}\Bigg[\bigg|\sum_{\lambda_*\in \cP_{C_*}}\left(\frac{2qNe^{i\tau_1}}{\gamma}\right)^{m_1}\frac{1}{(N-C_*)!N^{m_1}}\prod_{k\geq 2}\left(\frac{X_k}{\sqrt{k}}\right)^{m_k}\frac{1}{m_k!}\bigg|^{2q}\Bigg]\E\Bigg[|X_1|^{2q(N-C_*)}\bone_{\{|R_1|\in [(2qN-\sqrt{N})/\gamma, 2qN/\gamma]\}}\Bigg]\nonumber\\
    &=\frac{1}{(N-C_*)!^{2q}}\widetilde{\E}\Bigg[\bigg|\sum_{\lambda_*\in \cP_{C_*}}\left(\frac{2qe^{i\tau_1}}{\gamma}\right)^{m_1}\prod_{k\geq 2}\left(\frac{X_k}{\sqrt{k}}\right)^{m_k}\frac{1}{m_k!}\bigg|^{2q}\Bigg]\E\Bigg[|X_1|^{2q(N-C_*)}\bone_{\{|R_1|\in [(2qN-\sqrt{N})/\gamma, 2qN/\gamma]\}}\Bigg].\label{eq:split expectation}
\end{align}
Here, we supply a simple lower bound for the latter expectation in \eqref{eq:split expectation}:
\begin{align*}
    \E\Bigg[|X_1|^{2q(N-C_*)}\bone_{\{|R_1|\in [(2qN-\sqrt{N})/\gamma, 2qN/\gamma]\}}\Bigg]&\gg \int_{(2qN-\sqrt{N})/\gamma}^{2qN/\gamma} x^{2q(N-C_*)}e^{-\gamma x}\d x\\
    &= \gamma^{-2q(N-C_*)}\int_{2qN-\sqrt{N}}^{2qN} y^{2q(N-C_*)}e^{-y}\d y\\
    &\gg \sqrt{N}\left(\frac{2qN}{\gamma}\right)^{2q(N-C_*)}e^{-2qN},
\end{align*}
where in the last step we use the inequality $y^{2q(N-C_*)}e^{-y}\gg (2qN)^{2q(N-C_*)}e^{-2qN}$ which is a consequence of  Taylor's expansion  when $N$ is large enough compared to $C_*$.

We next bound the first expectation in \eqref{eq:split expectation}. Since for $\lambda_*\in \cP_{C_*}$, $m_1=C_*-\sum km_k$, we have
$$\widetilde{\E}\Bigg[\bigg|\sum_{\lambda_*\in \cP_{C_*}}\left(\frac{2qe^{i\tau_1}}{\gamma}\right)^{m_1}\prod_{k\geq 2}\left(\frac{X_k}{\sqrt{k}}\right)^{m_k}\frac{1}{m_k!}\bigg|^{2q}\Bigg]
=
\left(\frac{2q}{\gamma}\right)^{2qC_*}\widetilde{\E}\Bigg[\bigg|\sum_{\lambda_*\in \cP_{C_*}}\prod_{k\geq 2}\left(\frac{X_k}{\sqrt{k}(2qe^{i\tau_1}/\gamma)^{km_k}}\right)^{m_k}\frac{1}{m_k!}\bigg|^{2q}\Bigg].$$
In view of Lemma \ref{lemma:uniformly large}, there exists $\delta_0>0$ (independent of $C_*$) such that $$\widetilde{\E}\Bigg[\bigg|\sum_{\lambda_*\in \cP_{C_*}}\prod_{k\geq 2}\left(\frac{X_k}{\sqrt{k}(2qe^{i\tau_1}/\gamma)^{km_k}}\right)^{m_k}\frac{1}{m_k!}\bigg|^{2q}\Bigg]>\delta_0$$ for any $C_*$ large enough. In this case,
\begin{align*}
    &\hspace{0.5cm}\widetilde{\E}\Bigg[\bigg|\sum_{\lambda_*\in \cP_{C_*}}\left(\frac{2qNe^{i\tau_1}}{\gamma}\right)^{m_1}\frac{1}{(N-C_*)!N^{m_1}}\prod_{k\geq 2}\left(\frac{X_k}{\sqrt{k}}\right)^{m_k}\frac{1}{m_k!}\bigg|^{2q}|X_1|^{2q(N-C_*)}\bone_{\{|R_1|\in [(2qN-\sqrt{N})/\gamma, 2qN/\gamma]\}}\Bigg]\\
    &\gg\frac{1}{(N-C_*)!^{2q}} \left(\frac{2q}{\gamma}\right)^{2qC_*}\E\Bigg[|X_1|^{2q(N-C_*)}\bone_{\{|R_1|\in [(2qN-\sqrt{N})/\gamma, 2qN/\gamma]\}}\Bigg]\\
    &\gg\frac{1}{(N-C_*)!^{2q}} \left(\frac{2q}{\gamma}\right)^{2qC_*}\sqrt{N}\left(\frac{2qN}{\gamma}\right)^{2q(N-C_*)}e^{-2qN}\\&\gg \phi(N).
\end{align*}
On the other hand, for a fixed $C_*$, the second term of \eqref{eq:two terms} is bounded by 
\begin{align*}
    &\hspace{0.5cm}\widetilde{\E}\Bigg[|X_1|^{2q(N-C_*)}\bigg|\sum_{\lambda_*\in \cP_{C_*}}\frac{(2q/\gamma)^{m_1}m_1C_*}{\sqrt{N}(N-C_*)!}\,\prod_{k\geq 2}\left(\frac{X_k}{\sqrt{k}}\right)^{m_k}\frac{1}{m_k!}\bigg|^{2q}\bone_{\{|R_1|\in [(2qN-\sqrt{N})/\gamma, 2qN/\gamma]\}}\Bigg]\\
    &\ll \left(\frac{2qN}{\gamma}\right)^{2q(N-C_*)}\widetilde{\E}\left[\bigg|\sum_{\lambda_*\in \cP_{C_*}}\frac{(2q/\gamma)^{m_1}m_1C_*}{\sqrt{N}(N-C_*)!}\,\prod_{k\geq 2}\left(\frac{X_k}{\sqrt{k}}\right)^{m_k}\frac{1}{m_k!}\bigg|^{2q}\bone_{\{|R_1|\in [(2qN-\sqrt{N})/\gamma, 2qN/\gamma]\}}\right]\\
    &\ll_{C_*}\left(\frac{2qN}{\gamma}\right)^{2q(N-C_*)}\left(\frac{(2q/\gamma)^{C_*}C_*^2}{\sqrt{N}(N-C_*)!}\right)^{2q}\\
    &=o_{C_*}(\phi(N))
\end{align*}
for $N$ large enough.\footnote{The term $o_{C^*}(\varphi(N))$ means that it can be bounded by the product of a constant depending on $C^*$ and a term that is $o(\varphi(N))$.} 

 Altogether, we conclude from \eqref{eq:two terms}  that for some constant $M'=M'(\gamma,q)$ independent of $C_*$ and $N$,
 $$\widetilde{\E}\Bigg[\Bigg|\sum_{\substack{\lambda\in \cP_N\nonumber\\ m_1(\lambda)\geq  N-C_*}}\prod_{k\geq 1}\left(\frac{X_k}{\sqrt{k}}\right)^{m_k}\frac{1}{m_k!}\Bigg|^{2q}\bone_{\{|R_1|\in [(2qN-\sqrt{N})/\gamma, 2qN/\gamma]\}}\Bigg]\geq \frac{\phi(N)}{M'}-o_{C_*}(\phi(N)).$$
Pick $\ee=1/(2M'e^{LM})$ in Lemma \ref{lemma:tail small} and let $C_*$ be the implied constant $C$ therein. Recalling \eqref{eq:conditioning}, we have
\begin{align*}
    \E[|A_N|^{2q}]&\geq e^{-LM}\left(\frac{\phi(N)}{M'}-o_{C_*}(\phi(N))\right)-\ee\phi(N)\gg \phi(N)
\end{align*}
as $N\to\infty$. This completes the proof of the lower bound.

\begin{remark}
    \label{rem:p}
    The same arguments in this section carry through if we replace the rotationally invariant complex inputs $\{X_k\}_{k\geq 1}$ with real ones. 
    Suppose that $\{X_k\}_{k\geq 1}$ forms a sequence of i.i.d.~two-sided symmetric shifted {exponential} random variables with unit variance. That is, $\P(|X_k|\geq u)=\exp(-\gamma(u-c_\gamma))\wedge 1$ for $u\geq 0$, where $\gamma\in(0,2q)$ and $c_\gamma=\log({\gamma^2}/{2})/\gamma$. Note that the same moment asymptotics \eqref{eq:moments of SE} still applies. The only argument in the current section that depends on the rotation invariance is \eqref{f}, whereas in the real case, cross-terms may exist. Nevertheless, the cross-terms can be estimated similarly: the left-hand side of \eqref{f} can be bounded by
\begin{align*}
    \widetilde{\E}\Bigg[\bigg|\sum_{\substack{\bm\in \N_0^{\{2,3,\dots\}}\\ \sum_{j=2}^\infty jm_j> C}}\prod_{k\geq 2}\left(\frac{X_k}{\sqrt{k}(2q/\gamma)^k}\right)^{m_k}\frac{1}{m_k!}\bigg|^2\Bigg]  &\ll \sum_{\substack{\bm,\bn\in \N_0^{\{2,3,\dots\}}\\ \min\{\sum_{j=2}^\infty jm_j,\sum_{j=2}^\infty jn_j\}> C\\ \forall j,\,m_j+n_j\in 2\N_0}}\prod_{\substack{k\geq M\\ m_k\neq 0}}\frac{1}{e^{\ee k(m_k+n_k)}}\prod_{2\leq k<M}\frac{1}{e^{\ee (m_k+n_k)}}\\
    &\ll e^{-2\ee C^{1/3}},
\end{align*} for $C>M^3$ (possibly with a smaller $\ee$ and a larger $M$ than \eqref{d}), 
    which replaces \eqref{d} and the arguments that follow.
\end{remark}

\subsection{The case \texorpdfstring{$\gamma=2q$}{}}\label{sec:gamma=2q}

Fix $q\in(0,1]$. Our goal in this section is to prove \eqref{eq:gamma=2q result} that under the case (EXP) with $\gamma=2q$, 
$$     \E[|A_N|^{2q}]\asymp \frac{N^{1-q+q^2/2}}{(1+(1-q)\sqrt{\log N})^q}\asymp\begin{cases}
        N^{1-q+q^2/2}(\log N)^{-q/2}&\text{ if }q\in(0,1);\\
        \sqrt{N}&\text{ if }q=1.
    \end{cases}.$$
Let us recall the philosophy for the case $\gamma=2q$ that the main contribution to $A_N$ from the sum over partitions \eqref{eq:AN expression} is random and depends on the value of $|X_1|$. More precisely, on the event that $|X_1|$ is close to $ m$, we expect that the partitions $\lambda\in\cP_N$ with $m_1(\lambda)=m+O(\sqrt{N})$ dominate $A_N$. After applying the decomposition
$$\E[|A_N|^{2q}]=\sum_{m=0}^\infty \E[|A_N|^{2q}\bone_{|X_1|\in[m,m+1)}],$$
we will show that for $m\leq N-2$,
\begin{align}
    \E\big[|A_N|^{2q}\bone_{|X_1|\in[m,m+1)}\big]\asymp \frac{m^{-q+q^2/2}}{(1+(1-q)\sqrt{\log(N-m)})^{q}}.\label{s}
\end{align} The desired asymptotics then follows from Karamata's theorem (Theorem 1.5.11 of \citep{bingham1989regular}):
\begin{align}
    \sum_{m=0}^{N-2}\frac{m^{-q+q^2/2}}{(1+(1-q)\sqrt{\log(N-m)})^{q}}\asymp \frac{N^{1-q+q^2/2}}{(1+(1-q)\sqrt{\log N})^q}.\label{eq:sum formula}
\end{align}

As illustrated in Section \ref{sec:main ideas}, a key step towards \eqref{s} is reproducing the same multiplicative chaos approach from ($q$-UNIV) case to estimate \eqref{g} after conditioning on $R_1$. For this reason, we first prepare ourselves with a stronger version of Proposition \ref{prop:superC-asymp-truncated-moments}, which is the focus of Section \ref{sec:str}. When performing the multiplicative chaos analysis, an essential ingredient is establishing uniform estimates of a weighted mass of truncated multiplicative chaos (cf.~\eqref{b}). This will be the goal of Sections \ref{sec:partial mass} and \ref{sec: weighted partial mass}.

In this section, we will use $C_1,C_2,\dots$ to denote large positive constants (greater than one) that only depend on $q$ and each other, \emph{which may be different from constants defined in other sections}. Throughout, we fix $q\in(0,1]$, and the notation $\ll$ will always denote asymptotic constants that depend only on $q$.

\subsubsection{Strengthening Proposition \texorpdfstring{\ref{prop:superC-asymp-truncated-moments}}{}}\label{sec:str}
In this section, we show that Proposition \ref{prop:superC-asymp-truncated-moments} also extends to the (EXP) case with $\gamma=2q$, if we restrict to partitions without ones (i.e., $m_1(\lambda)=0$). Recall \eqref{eq:ANM}.

\begin{proposition}\label{prop:M*2}
Fix an integer $M_*\geq 2$ larger than some constant depending only on the distribution of $X_k$. For any large $N$ and any $q\in(0,1]$, under (EXP) with $\gamma=2q$ we have \begin{equation*}\E[|A_{N,M_*}|^{2q}]\asymp \left(\frac{1}{1+(1-q)\sqrt{\log N}}\right)^q\asymp\begin{cases}
        (\log N)^{-q/2}&\text{ if }q\in(0,1);\\
        1&\text{ if }q=1.
    \end{cases}
    \end{equation*}    
\end{proposition}
\begin{proof}
The proof is very similar to the proof of Proposition \ref{prop:superC-asymp-truncated-moments}, and the only modification needed lies in the upper bound of Proposition \ref{SZ-prop3.1}, where we denote by $J=\ceil{\log(C_1\sqrt{N})/\log 2}$. Indeed, \eqref{eq:a lambda moments} becomes
$$\E[|a(\lambda)|^{2q}]\ll \prod_{k:m_k>0} \frac{C_2}{k^{qm_k}}.$$
Since $k\geq M_*\geq 2$, \eqref{eq:holder p hat} becomes 
$$\E\Bigg[\bigg|\sum_{\substack{|\lambda|=N\\ \lambda_1\leq N/2^J}}a(\lambda)\bigg|^{2q}\Bigg]^{1/(2q)}\le \sum_{\substack{|\lambda|=N\\ \lambda_1\leq \sqrt{N}/C_1}}\prod_{k:m_k>0} \frac{C_2}{k^{qm_k}} \le |\cP_N| C_2^{\sqrt{N}/C_1}2^{-q\sum_{k=M_*}^{\sqrt{N}/C_1}m_k}.$$
Using \eqref{eq:o}, we obtain
$$\E\Bigg[\bigg|\sum_{\substack{|\lambda|=N\\ \lambda_1\leq N/2^J}}a(\lambda)\bigg|^{2q}\Bigg]^{1/(2q)}\ll \frac{1}{N}$$
for $C_1$ picked large enough.
    The rest follows line by line as in the proof of Proposition \ref{prop:superC-asymp-truncated-moments}.
\end{proof}

\begin{proposition}\label{prop:strengthened}
Suppose that (EXP) holds with $\gamma=2q$. It holds that
    $$\E\Bigg[\bigg|\sum_{\substack{\lambda\in\cP_{N}\\ m_1(\lambda)=0}}\prod_{k\geq 2}\left(\frac{X_k}{\sqrt{k}}\right)^{m_k}\frac{1}{m_k!}\bigg|^{2q}\Bigg]\asymp \frac{1}{(1+(1-q)\sqrt{\log N})^q}\asymp\begin{cases}
        (\log N)^{-q/2}&\text{ if }q\in(0,1);\\
        1&\text{ if }q=1.
    \end{cases}$$
\end{proposition}

\begin{proof}Recall \eqref{eq:P def}. 
    We follow a similar argument that deduced Theorem \ref{thm:main} equation \eqref{eq:asymp LT} from Proposition \ref{prop:superC-asymp-truncated-moments} in Section \ref{sec:LT reduction}, using now Proposition \ref{prop:M*2} instead. For the upper bound when $q\leq 1/2$ (and similarly for $q>1/2$ using instead Minkowski's inequality), by Proposition \ref{prop:M*2},
\begin{align*}
    &\hspace{0.5cm}\E\Bigg[\bigg|\sum_{\substack{\lambda\in\cP_{N}\\ m_1(\lambda)=0}}\prod_{k\geq 2}\left(\frac{X_k}{\sqrt{k}}\right)^{m_k}\frac{1}{m_k!}\bigg|^{2q}\Bigg]\\
    &\ll \sum_{\bm\in\cP_{N,M_*}^{<}}\E\left[\big|A_{N-\sum_{j=1}^{M_*-1}jm_j,M_*}\big|^{2q}\right]\prod_{k=1}^{M_*-1}\E\left[\left|\left(\frac{X_k}{\sqrt{k}}\right)^{m_k}\frac{1}{m_k!}\right|^{2q}\right]+ \sum_{\lambda:\lambda_1<M_*}\prod_{k=1}^{M_*-1}\E\left[\left|\left(\frac{X_k}{\sqrt{k}}\right)^{m_k}\frac{1}{m_k!}\right|^{2q}\right].\\
    &\ll \sum_{\bm\in\cP_{N,M_*}^{<}}\frac{1}{(1+(1-q)\sqrt{\log(N-\sum_{j=1}^{M_*-1}jm_j)})^{q}}\prod_{k=2}^{ M_*-1}C_3m_k^{1/2-q}k^{-qm_k}+\sum_{\lambda:\lambda_1<M_*}\prod_{k=2}^{M_*-1}C_3m_k^{1/2-q}k^{-qm_k}.
\end{align*}
Since the product ranges in $k\geq 2$, the second sum and the first sum restricted to $\lambda$ such that $\sum_{j=1}^{M_*-1}jm_j> N/2$ can be controlled in the same way as before by $1/N$. The rest is then 
\begin{align*}&\hspace{0.5cm}\sum_{\substack{\bm\in\cP_{N,M_*}^{<}\\ \sum_{j=1}^{M_*-1}jm_j\leq N/2}}\frac{1}{(1+(1-q)\sqrt{\log(N-\sum_{j=1}^{M_*-1}jm_j)})^{q}}\prod_{k=2}^{M_*-1}C_3m_k^{1/2-q}k^{-qm_k}\\
&\ll \frac{1}{(1+(1-q)\sqrt{\log N})^q}\sum_{\substack{\bm\in\cP_{N,M_*}^{<}\\ \sum_{j=1}^{M_*-1}jm_j\leq N/2}}\prod_{k=2}^{M_*-1}k^{-qm_k/2}\\
    &\leq \frac{1}{(1+(1-q)\sqrt{\log N})^q} \prod_{k=2}^{M_*-1}\sum_{m_k=0}^\infty k^{-qm_k/2}\ll \frac{1}{(1+(1-q)\sqrt{\log N})^q}.
\end{align*}
The lower bound follows from precisely the same proof in Section \ref{sec:LT reduction}.
\end{proof}

\subsubsection{Low moments of partial mass of truncated chaos}\label{sec:partial mass}
The goal of this section is to establish the following two propositions. Fix large constants $C_4,C_5$ to be determined. The reason why we introduce these constants will become apparent later in Section \ref{sec: weighted partial mass}.

\begin{proposition}\label{prop:uniform}
    It holds that uniformly for $r\in[e^{-C_5/K},e^{C_5/K}]$, 
    $$ \E\bigg[\bigg(\int_{|\theta|\leq \frac{1}{C_4\sqrt{N}}}|{F}_{K,M_*}(re^{i\theta})|^2\d\theta \bigg)^q\bigg]\gg \frac{K_r^qN^{-q+q^2/2}}{(1+(1-q)\sqrt{\log N})^{q}}.$$
\end{proposition}

\begin{proposition}\label{prop:uniform2}
    Uniformly in $1\leq j\leq \log(\pi C_4\sqrt{N})$, $K\gg \sqrt{N}$, and $r\in[e^{-C_5/K},e^{C_5/K}]$,
    \begin{align}
        \E\bigg[\bigg(\int_{ \frac{e^{j-1}}{C_4\sqrt{N}}\leq|\theta|\leq \frac{e^{j}}{C_4\sqrt{N}}}|{F}_{K,M_*}(re^{i\theta})|^2\d\theta \bigg)^q\bigg]\ll    \Big(\frac{e^{j-1}}{\sqrt{N}}\Big)^{2q-q^2} \frac{C_4^{q^2}K^q}{(1+(1-q)\sqrt{\log N})^{q}}.\label{eq:jub1}
    \end{align}
    Moreover,
\begin{align}
    \E\bigg[\bigg(\int_{|\theta|\leq \frac{1}{C_4\sqrt{N}}}|{F}_{K,M_*}(re^{i\theta})|^2\d\theta \bigg)^q\bigg]\ll \frac{C_4^{q^2}N^{-q+q^2/2}K^q}{(1+(1-q)\sqrt{\log N})^{q}}.\label{eq:jub2}
\end{align}
\end{proposition}

A key strategy of proving Propositions \ref{prop:uniform} and \ref{prop:uniform2} is to decompose $F_{K,M_*}$ into two terms depending on the region of integration of $\theta$; cf.~\eqref{eq:decomp ex}. Consider a large constant $C_6$ to be determined and our goal is to show that with high probability,
$$\bigg|\exp\Big(\sum_{k=M_*}^{K_*/C_6}\frac{X_k}{\sqrt{k}}(re^{i\theta})^k\Big)\bigg|\approx \bigg|\exp\Big(\sum_{k=M_*}^{K_*/C_6}\frac{X_k}{\sqrt{k}}r^k\Big)\bigg|$$
uniformly for $|\theta|\leq 1/K_*$ and $K_*\leq\sqrt{N}$.
To this end, we take the logarithm and compute the difference
$$\Re \Big(\sum_{k=M_*}^{K_*/C_6}\frac{X_k}{\sqrt{k}}(re^{i\theta})^k\Big)-\Re \Big(\sum_{k=M_*}^{K_*/C_6}\frac{X_k}{\sqrt{k}}r^k\Big)=\sum_{k=M_*}^{K_*/C_6}\frac{r^k}{\sqrt{k}}\Re(X_ke^{ik\theta}-X_k).$$
It then suffices to understand the magnitudes of
$$\sup_{|\theta|\leq 1/K_*}\sum_{k=M_*}^{K_*/C_6}\frac{r^k}{\sqrt{k}}\Re(X_ke^{ik\theta}-X_k)\quad\text{ and }\quad \inf_{|\theta|\leq 1/K_*}\sum_{k=M_*}^{K_*/C_6}\frac{r^k}{\sqrt{k}}\Re(X_ke^{ik\theta}-X_k).$$
However, bounds under the original probability $\P$ turn out not sufficient for our purpose. 
To this end, we define a new probability measure $\widehat{\Q}_{K_*/C_6,M_*}$, where
\begin{align}
    \frac{\d\widehat{\Q}_{K_*/C_6,M_*}}{\d \P}=\frac{\exp(2q\sum_{k=M_*}^{K_*/C_6}\frac{r^k}{\sqrt{k}}\Re(X_k))}{\E[\exp(2q\sum_{k=M_*}^{K_*/C_6}\frac{r^k}{\sqrt{k}}\Re(X_k))]}.\label{eq:hat Q}
\end{align}
If the situation is clear we write instead $\widehat{\Q}$. The denominator in \eqref{eq:hat Q} can be estimated using Lemma \ref{SZ-lemma2.2} (in a similar way that derives Lemma \ref{SZ-lemma2.2!}), which gives 
\begin{align}
    \E\bigg[\exp\Big(2q\sum_{k=M_*}^{K_*/C_6}\frac{r^k}{\sqrt{k}}\Re(X_k)\Big)\bigg]\asymp \Big(\frac{K_*}{C_6M_*}\Big)^{q^2}\label{l}
\end{align}
uniformly in $K_*$ large enough. 

\begin{proposition}\label{prop:chaining}
  There exists a constant $C_7>0$ depending only on $q$, independent of $K_*\leq C_4\sqrt{N}$, such that uniformly in $r\in[e^{-C_5/N},e^{C_5/N}]$ and $C_6,N$ large enough,
    \begin{align}
        {\widehat{\Q}}\bigg(\sup_{|\theta|\leq 1/K_*}\Big|\sum_{k=M_*}^{K_*/C_6}\frac{r^k}{\sqrt{k}}\Re(X_ke^{ik\theta}-X_k)\Big|\geq u\bigg)\leq C_7 e^{-10u},\quad u\geq 0.\label{eq:chaining1}
    \end{align}
  In particular, uniformly in $r\in[e^{-C_5/N},e^{C_5/N}]$ and $C_6,N$ large enough,
    \begin{align}
        \E^{\widehat{\Q}}\bigg[\exp\Big(2q\sup_{|\theta|\leq 1/K_*}\sum_{k=M_*}^{K_*/C_6}\frac{r^k}{\sqrt{k}}\Re(X_ke^{ik\theta}-X_k)\Big)\bigg]\ll 1\label{eq:chaining2}
    \end{align}and
    \begin{align}
        \E^{\widehat{\Q}}\bigg[\exp\Big(2q\inf_{|\theta|\leq 1/K_*}\sum_{k=M_*}^{K_*/C_6}\frac{r^k}{\sqrt{k}}\Re(X_ke^{ik\theta}-X_k)\Big)\bigg]\gg 1.\label{eq:chaining3}
    \end{align}
\end{proposition}
To prove Proposition \ref{prop:chaining}, we need the following preliminary result on generic chaining that offers tail bounds arising from \eqref{eq:chaining1}. For a thorough treatment on chaining, we refer to \citep{talagrand2014upper}.
\begin{lemma}[Theorem 3.2 of \citep{dirksen2015tail}]\label{lemma:chaining}
    Consider a continuous stochastic process $\{Z_t\}_{t\in T}$ indexed by a bounded metric space $(T,d)$ satisfying
    \begin{align}
        \forall s,t\in T,~\forall u\geq 0,~\P(|Z_s-Z_t|\geq u\,d(s,t))\leq 2e^{-u}.\label{eq:tail cond}
    \end{align}
    Then for any $t_0\in T$,
    \begin{align}
        \P\Bigg(\sup_{t\in T}|Z_t-Z_{t_0}|\geq L\Big(\inf_\T\sup_{t\in T}\sum_{n\geq 1}2^{n}d(t,T_n)+u\sup_{s,t\in T}d(s,t)\Big)\Bigg)\leq e^{-u},\label{eq:tail bound chaining}
    \end{align}
    where $\T$ is the set of all admissible sequence of subsets $\{T_n\}_{n\geq 0}$ of $T$ satisfying $|T_0|=1$ and $|T_n|\leq 2^{2^n}$ for $n\geq 1$. 
\end{lemma}

\begin{proof}[Proof of Proposition \ref{prop:chaining}]
Consider large constants $C_8,C_9$ to be determined. 
To apply Lemma \ref{lemma:chaining}, we let $T=\{\theta:|\theta|\leq 1/K_*\}$ with $d(\theta,\theta')=K_*|\theta-\theta'|/C_8$, $t_0=0$, and 
$$Z_\theta=\sum_{k=M_*}^{K_*/C_6}\frac{r^k}{\sqrt{k}}\Re (X_ke^{ik\theta}).$$
Note that the chaining is performed under the tilted probability measure $\widehat{\Q}$. 
We first verify \eqref{eq:tail cond}, with $\P=\widehat{\Q}$ therein. Consider $\theta,\theta'\in T$ and define $\beta=10C_9(K_*|\theta-\theta'|)^{-1}$. We have
\begin{align}\begin{split}
    &\hspace{0.5cm}\widehat{\Q}\left(\Big|\sum_{k=M_*}^{K_*/C_6}\frac{r^k}{\sqrt{k}}\Re (X_k(e^{ik\theta}-e^{ik\theta'}))\Big|\geq \frac{uK_*|\theta-\theta'|}{C_8}\right)\\
    &=\widehat{\Q}\left(\sum_{k=M_*}^{K_*/C_6}\frac{r^k}{\sqrt{k}}\Re (X_k(e^{ik\theta}-e^{ik\theta'}))\geq \frac{uK_*|\theta-\theta'|}{C_8}\right)+\widehat{\Q}\left(\sum_{k=M_*}^{K_*/C_6}\frac{r^k}{\sqrt{k}}\Re (X_k(e^{ik\theta'}-e^{ik\theta}))\geq \frac{uK_*|\theta-\theta'|}{C_8}\right)\\
    &\leq \frac{\E^{\widehat{\Q}}[\exp(2\beta \sum_{k=M_*}^{K_*/C_6}\frac{r^k}{\sqrt{k}}\Re (X_k(e^{ik\theta}-e^{ik\theta'})))]}{\exp(2\beta uK_*|\theta-\theta'|/C_8)}+\frac{\E^{\widehat{\Q}}[\exp(2\beta \sum_{k=M_*}^{K_*/C_6}\frac{r^k}{\sqrt{k}}\Re (X_k(e^{ik\theta'}-e^{ik\theta})))]}{\exp(2\beta uK_*|\theta-\theta'|/C_8)}.
\end{split}\label{eq:2exp}
\end{align}
To offer an upper bound for the numerators, we apply the same asymptotic computation in Lemma \ref{lemma-Q-computations}, which gives\begin{align*}
    \E^{\widehat{\Q}}\bigg[\exp\Big(2\beta \frac{r^k}{\sqrt{k}}\Re (X_k(e^{ik\theta}-e^{ik\theta'}))\Big)\bigg]&=\frac{\E\bigg[\exp\Big(2\beta \frac{r^k}{\sqrt{k}}\Re (X_k(e^{ik\theta}-e^{ik\theta'}))+2q\frac{r^k}{\sqrt{k}}\Re X_k\Big)\bigg]}{\E[\exp(2q\frac{r^k}{\sqrt{k}}\Re X_k)]}.
\end{align*}
Recalling $X_k=e^{i\tau_k}R_k$, we have by Taylor's expansion that
\begin{align}
    &\hspace{0.5cm}\E\bigg[\exp\Big(2\beta \frac{r^k}{\sqrt{k}}\Re (X_k(e^{ik\theta}-e^{ik\theta'}))+2q\frac{r^k}{\sqrt{k}}\Re X_k\Big)\bigg]\nonumber\\
    \begin{split}
        &=1+2\frac{r^{2k}}{k}\E[|R_k(q\cos(\tau_k)+\beta(\cos(\tau_k+k\theta)-\cos(\tau_k+k\theta')))|^2]\\
    &\hspace{0.7cm}+\sum_{j=3}^\infty \frac{(2r^k/\sqrt{k})^j}{j!}\E[|R_k(q\cos(\tau_k)+\beta(\cos(\tau_k+k\theta)-\cos(\tau_k+k\theta')))|^j].\label{eq:jth power}
    \end{split}
\end{align} For the second moment in \eqref{eq:jth power}, write $\Delta_k:=k(\theta-\theta')/2,\alpha_k:=k(\theta+\theta')/2$. Using independence of $R_k$ and $\tau_k$, a direct computation gives 
\begin{align*}
    &2\frac{r^{2k}}{k}\E[|R_k(q\cos(\tau_k)+\beta(\cos(\tau_k+k\theta)-\cos(\tau_k+k\theta')))|^2]\\
    =&2\frac{r^{2k}}{k}\E[|q\cos(\tau_k)+\beta(\cos(\tau_k+k\theta)-\cos(\tau_k+k\theta'))|^2]\\
    =&(q^2+4\beta^2\sin^2(\Delta_k)-4q\beta\sin(\Delta_k)\sin(\alpha_k))\frac{r^{2k}}{k}.
\end{align*}
For the remainder of \eqref{eq:jth power}, the inequality $|a+b|^j\leq 2^j(|a|^j+|b|^j)$ implies \begin{align*}
    |q\cos(\tau_k)+\beta(\cos(\tau_k+k\theta)-\cos(\tau_k+k\theta'))|^j &\leq (2q)^j+(4\beta\Delta_k)^j.
\end{align*} Similarly as in the proof of Lemma \ref{SZ-lemma2.2} (i), we have 
\begin{align*}
    &\Big|\sum_{j=3}^\infty \frac{(2r^k/\sqrt{k})^j}{j!}\E[|R_k(q\cos(\tau_k)+\beta(\cos(\tau_k+k\theta)-\cos(\tau_k+k\theta')))|^j]\Big|\\
    &\hspace{1cm}\leq \sum_{j=3}^\infty \frac{(2r^k/\sqrt{k})^j}{j!}\E[|R_k|^j]\big((2q)^j+(4\beta\Delta_k)^j\big) \ll k^{-3/2}\left((2q)^3+(4\beta\Delta_k)^3+\E[e^{\frac{6q}{k^{1/8}}|R_1|}]+\E[e^{\frac{12\beta\Delta_k}{k^{1/8}}|R_1|}]\right),
\end{align*}where in the last step the asymptotic constant is absolute given that $N$ is large enough (in terms of constants $C_4,C_5$ that depend only on $q$), where we recall that $k\leq K_*\leq C_4\sqrt{N}$ while $|\log r|\leq C_5/N$. Since $\beta\Delta_k=10C_9k/K_*\leq 10C_9$, $k\geq M_*\geq (120C_9/q)^8$ will suffice to ensure $k^{-1/8}\max\{6q,120C_9\}<\gamma=2q$. Therefore,
$$\Big|\sum_{j=3}^\infty \frac{(2r^k/\sqrt{k})^j}{j!}\E[|R_k(q\cos(\tau_k)+\beta(\cos(\tau_k+k\theta)-\cos(\tau_k+k\theta')))|^j]\Big|\ll C_9^3k^{-3/2}.$$
Inserting these estimates back into \eqref{eq:jth power}, we have
$$\E\bigg[\exp\Big(2\beta \frac{r^k}{\sqrt{k}}\Re (X_k(e^{ik\theta}-e^{ik\theta'}))+2q\frac{r^k}{\sqrt{k}}\Re X_k\Big)\bigg]=1+(q^2+4\beta^2\sin^2(\Delta_k)-4q\beta\sin(\Delta_k)\sin(\alpha_k))\frac{r^{2k}}{k}+C_9^3O(k^{-3/2}).$$
Taking products over $k$, and using \eqref{eq:hat Q} and \eqref{l}, we have \begin{align*}
    \E^{\widehat{\Q}}\bigg[\exp\Big(2\beta \sum_{k=M_*}^{K_*/C_6}\frac{r^k}{\sqrt{k}}\Re (X_k(e^{ik\theta}-e^{ik\theta'}))\Big)\bigg]& \ll \exp\left(\sum_{k=M_*}^{K_*/C_6}\frac{4\beta^2\sin^2(\Delta_k)r^{2k}}{k}+\sum_{k=M_*}^{K_*/C_6}\frac{4q\beta|\sin(\Delta_k)|r^{2k}}{k}\right)\\
    &\ll\exp\left(\beta^2|\theta-\theta'|^2\sum_{k=M_*}^{K_*/C_6}kr^{2k}+2q\beta|\theta-\theta'|\sum_{k=M_*}^{K_*/C_6}r^{2k}\right)\\
    &\ll\exp\left(\frac{50C_9^2}{C_6^{2}}+\frac{20C_9}{C_6}\right).
\end{align*} Note that the asymptotic constants do not depend on $C_9$ since we assumed $k\geq M_*\geq (120C_9/q)^8$. Therefore, there exists a constant $C_{10}$ depending only on $q$, we may pick $C_6$ large enough depending on $C_9$ so that 
$$\E^{\widehat{\Q}}\bigg[\exp\Big(2\beta \sum_{k=M_*}^{K_*/C_6}\frac{r^k}{\sqrt{k}}\Re (X_k(e^{ik\theta}-e^{ik\theta'}))\Big)\bigg]\leq C_{10}.$$
Inserting into \eqref{eq:2exp} while interchanging $\theta,\theta'$ yields
\begin{align*}
    \widehat{\Q}\left(\Big|\sum_{k=M_*}^{K_*/C_6}\frac{r^k}{\sqrt{k}}\Re (X_k(e^{ik\theta}-e^{ik\theta'}))\Big|\geq \frac{uK_*|\theta-\theta'|}{C_8}\right)\leq 2C_{10}e^{-5C_9u/C_8}.
\end{align*}
Picking $C_9$ large depending on $C_8,C_{10}$ gives that $\min\{1,2C_{10}e^{-5C_9u/C_8}\}\leq 2e^{-u}$, as desired.

Applying \eqref{eq:tail bound chaining} with $t_0=0$ leads to 
$$\widehat{\Q}\Bigg(\sup_{\theta\in T}\Big|\sum_{k=M_*}^{K_*/C_6}\frac{r^k}{\sqrt{k}}\Re(X_ke^{ik\theta}-X_k)\Big|\geq L\Big(\inf_\T\sup_{t\in T}\sum_{n\geq 1}2^{n}d(t,T_n)+u\sup_{s,t\in T}d(s,t)\Big)\Bigg)\leq e^{-u}.$$
Recalling that $d(\theta,\theta')=K_*|\theta-\theta'|/C_8$ and $K_*|\theta-\theta'|\leq 2$, we have $\sup_{s,t\in T}d(s,t)\leq 2/C_8$. Next, we claim that 
$$\inf_\T\sup_{t\in T}\sum_{n\geq 1}2^{n}d(t,T_n)\ll \frac{1}{C_8}.$$
Indeed, this follows by taking uniform nets of $\{T_n\}$, so that $\sup_{t\in T}d(t,T_n)\ll 2^{-2^n}/C_8$ for $n\geq 1$. Therefore, by picking $C_8$ large enough (depending only on the universal constant $L$), we have
\begin{align*}
    \widehat{\Q}\Bigg(\sup_{\theta\in T}\Big|\sum_{k=M_*}^{K_*/C_6}\frac{r^k}{\sqrt{k}}\Re(X_ke^{ik\theta}-X_k)\Big|\geq 1+u\Bigg)\leq e^{-20 u}.
\end{align*}This establishes \eqref{eq:chaining1} by picking e.g.~$C_7=e^{20}$. 
Note that by our construction, $M_*$ depends only on $C_8,C_9$, and hence only on $q$.
Directly integrating \eqref{eq:chaining1} yields \eqref{eq:chaining2}, since $q\leq 1$.
Moreover, by taking $u=\log C_7$ in \eqref{eq:chaining1}, we have
$${\widehat{\Q}}\bigg(\inf_{|\theta|\leq 1/K_*}\sum_{k=M_*}^{K_*/C_6}\frac{r^k}{\sqrt{k}}\Re(X_ke^{ik\theta}-X_k)\geq -\log C_7\bigg)\geq {\widehat{\Q}}\bigg(\sup_{|\theta|\leq 1/K_*}\Big|\sum_{k=M_*}^{K_*/C_6}\frac{r^k}{\sqrt{k}}\Re(X_ke^{ik\theta}-X_k)\Big|\leq \log C_7\bigg)\geq \frac{1}{2}.$$
Since $C_7$ is a constant, \eqref{eq:chaining3} follows.
\end{proof}

 \begin{proof}[Proof of Proposition  {\ref{prop:uniform}}]
 We introduce the decomposition
$$F_{K,M_*}(z)=\overline{F}_{K,M_*}(z)\times \underline{F}_{K,M_*}(z):=\exp\Big(\sum_{k=M_*}^{C_4\sqrt{N}/C_6}\frac{X_k}{\sqrt{k}}z^k\Big)\times \exp\Big(\sum^K_{k=C_4\sqrt{N}/C_6}\frac{X_k}{\sqrt{k}}z^k\Big).$$
By independence,
\begin{align*}
    \E\left[\left(\int_{|\theta|\leq \frac{1}{C_4\sqrt{N}}}|{F}_{K,M_*}(re^{i\theta})|^2\d\theta \right)^q\right]    &\geq \E\left[\left(\int_{|\theta|\leq \frac{1}{C_4\sqrt{N}}}|\underline{F}_{K,M_*}(re^{i\theta})|^2\d\theta \right)^q\right]\times \E\left[\inf_{|\theta|\leq \frac{1}{C_4\sqrt{N}}}|\overline{F}_{K,M_*}(re^{i\theta})|^{2q}\right].
\end{align*}
We first analyze the second term. Using \eqref{eq:hat Q} and \eqref{l}, we write
\begin{align*}
   \E\bigg[\inf_{|\theta|\leq \frac{1}{C_4\sqrt{N}}}|\overline{F}_{K,M_*}(re^{i\theta})|^{2q}\bigg]&=\E\bigg[|\overline{F}_{K,M_*}(r)|^{2q}\inf_{|\theta|\leq \frac{1}{C_4\sqrt{N}}}\Big|\frac{\overline{F}_{K,M_*}(re^{i\theta})}{\overline{F}_{K,M_*}(r)}\Big|^{2q}\bigg]\\
   &=\Big(\frac{C_4\sqrt{N}}{C_6M_*}\Big)^{q^2}\E^{\widehat{\Q}}\bigg[\inf_{|\theta|\leq \frac{1}{C_4\sqrt{N}}}\Big|\frac{\overline{F}_{K,M_*}(re^{i\theta})}{\overline{F}_{K,M_*}(r)}\Big|^{2q}\bigg]\\
   &=\Big(\frac{C_4\sqrt{N}}{C_6M_*}\Big)^{q^2}\E^{\widehat{\Q}}\bigg[\exp\Big(2q\inf_{|\theta|\leq \frac{1}{C_4\sqrt{N}}}\sum_{k=M_*}^{C_4\sqrt{N}/C_6}\frac{r^k}{\sqrt{k}}\Re(X_ke^{ik\theta}-X_k)\Big)\bigg].
\end{align*}
As a consequence of \eqref{eq:chaining3} and since $C_6$ and $M_*$ depend only on $q$, we have
$$\E\bigg[\inf_{|\theta|\leq \frac{1}{C_4\sqrt{N}}}|\overline{F}_{K,M_*}(re^{i\theta})|^{2q}\bigg]\gg \Big(\frac{C_4\sqrt{N}}{C_6M_*}\Big)^{q^2}\gg N^{q^2/2}.$$
It remains to show 
\begin{align}
    \E\Bigg[\bigg(\int_{|\theta|\leq \frac{1}{C_4\sqrt{N}}}|\underline{F}_{K,M_*}(re^{i\theta})|^2\d\theta \bigg)^q\Bigg]\gg \frac{K_r^qN^{-q}}{(1+(1-q)\sqrt{\log N})^{q}}.\label{eq:Flow1}
\end{align}
The case of $q=1$ follows by a direction computation in the form of Lemma \ref{SZ-lemma2.2!}.
Let us assume $q<1$. The approach is to perform a change of variable and then adapt the proof of Proposition \ref{SZ-Prop-8.2}. Using the change of variable $\theta\mapsto \pi C\theta \sqrt{N}$, we have 
\begin{align}
    \E\left[\left(\int_{|\theta|\leq \frac{1}{C_4\sqrt{N}}}|\underline{F}_{K,M_*}(re^{i\theta})|^2\d\theta \right)^q\right]\gg N^{-q/2}\E\left[\left(\int_{-\pi}^\pi|\underline{F}_{K,M_*}(re^{i\theta/(\pi C_4\sqrt{N})})|^2\d\theta \right)^q\right].\label{eq:change of variable}
\end{align}
Define the tilted measures
 \begin{align*}
        \frac{\d\Q^{(1)}}{\d\P} := \frac{\exp(2\sum_{k=C_4\sqrt{N}/C_6}^{K-1}\frac{r^k}{\sqrt{k}}R_k\cos(\tau_k))}{\E\left[\exp(2\sum_{k=C_4\sqrt{N}/C_6}^{K-1}\frac{r^k}{\sqrt{k}}R_k\cos(\tau_k))\right]}
    \end{align*}
    and
    \begin{align*}
        \frac{\d\Q^{(2)}_{r,M,K,\theta}}{\d\P} := \frac{\exp(2\sum_{m=\log(C_4\sqrt{N}/C_6)+1}^{\log K_r}(Z_0(m)+Z_\theta(m)))}{\E[\exp(2\sum_{m=\log(C_4\sqrt{N}/C_6)+1}^{\log K_r}(Z_0(m)+Z_\theta(m)))]},
    \end{align*}
    where
    \begin{align*}
    Z_\theta(m) &:=\Re \sum_{e^{m-1}\le k<e^m}\frac{X_kr^ke^{ik\theta/(\pi C_4\sqrt{N})}}{\sqrt{k}} =\sum_{e^{m-1}\le k<e^m}\frac{r^k}{\sqrt{k}}R_k\cos\Big(\tau_k+\frac{k\theta}{\pi C_4\sqrt{N}}\Big).
    \end{align*}
    The definition \eqref{eq:muk def} still stands, and \eqref{eq:nuk def} is replaced by
    $$\nu_k=\nu_k(\theta):=\E^{\Q^{(2)}}\left[\frac{r^k}{\sqrt{k}}R_k\cos(\tau_k)\right] = \frac{r^{2k}}{k}+\frac{\cos(k\theta/(\pi C_4\sqrt{N}))r^{2k}}{k}+O(k^{-3/2}).$$
Next, we need to adapt Definition \ref{def:event L} by the following.

\begin{definition}\label{def:event L2}
    Fix $L_1>20$. Let $A$ be a real number with $1\le A\le\sqrt{\log K_r}$. Define $\cL(\theta)=\cL(A,\theta;K)$ as the event that for each $\log (C_4\sqrt{N}/C_6)\le n\le \log K_r$, one has  \[-A-L_1n\le \sum_{k=C_4\sqrt{N}/C_6}^{e^n-1}\left(\Re \frac{X_kr^ke^{ik\theta/(\pi C_4\sqrt{N})}}{\sqrt{k}}-\mu_k\right)\le A-5\log \Big(n-\log\Big(\frac{\sqrt{N}}{C_6}\Big)\Big).\]
     Also, let $\cL=\cL(A;K)$ be the random subset of $\theta\in [-\pi,\pi)$ such that $\cL(\theta)$ holds.
\end{definition}

We apply H\"{o}lder's inequality to obtain 
\begin{align}
    \E\left[\left(\int_{-\pi}^{\pi}|\underline{F}_{K,M_*}(re^{i\theta/(\pi C_4\sqrt{N})})|^2\d\theta \right)^q\right]\gg \frac{\left(\E\left[\int_{\mathcal{L}}|\underline{F}_{K,M_*}(re^{i\theta/(\pi C_4\sqrt{N})})|^2\d\theta \right]\right)^{2-q}}{\left(\E\left[\left(\int_{\mathcal{L}}|\underline{F}_{K,M_*}(re^{i\theta/(\pi C_4\sqrt{N})})|^2\d\theta \right)^2\right]\right)^{1-q}}.\label{eq:holder 2}
\end{align}
The lower bound of the numerator follows in exactly the same way as Lemma \ref{lemma:lbnumerator} (the normal approximation argument remains the same, since the angle $\theta$ has not come into play yet). The factor arising from the change of measure is now $K_r/(C_4\sqrt{N}/C_6)\asymp K_r/\sqrt{N}$ (instead of $K_r$), and thus
\begin{align}
    \E\left[\int_{\mathcal{L}}|\underline{F}_{K,M_*}(re^{i\theta/(\pi C_4\sqrt{N})})|^2\d\theta \right]\gg \frac{AK_r}{\sqrt{N\log (K_r/\sqrt{N})}}.\label{eq:lb 2}
\end{align}

\sloppy To give an upper bound of the denominator, we follow the proof of Proposition \ref{SZ-prop-10.2} and indicate the necessary changes. First, we redefine the threshold $M$ as the smallest integer such that $e^M\ge\max\{\min\{10^3C_4\sqrt{N}/(C_6|\theta|),K_r/e\}, M_*\}$. In the definition of the event $\widetilde{\cL}$, we replace \eqref{eq:tildeL1} by 
$$ -A-L_1M\le A_0(M),A_\theta(M)\le A-5\log \Big(M-\log\Big(\frac{C_4\sqrt{N}}{C_6}\Big)\Big).$$
As the constant arising from the change of measure alters, \eqref{eq:?} is replaced by 
$$\E\left[\bone_{\cL(0)\cap\cL(\theta)}|\underline{F}_{K,M_*}(r)|^2|\underline{F}_{K,M_*}(re^{i\theta/(\pi C_4\sqrt{N})})|^2\right]\ll \frac{e^{4M}}{N^2}\E\left[\bone_{\widetilde{\cL}}\,e^{2A_0(M)+2A_\theta(M)}\prod_{m=M+1}^{\log K_r}e^{2Z_0(m)+2Z_\theta(m)}\right].$$
Next, the form of Proposition \ref{SZ-prop-11.1} remains the same, whose proof can be adapted with $\theta$ replaced by $\theta/(\pi C_4\sqrt{N})$, as long as we show the equivalent of decorrelation step \eqref{eq:decor}. For this, let us first compute using Lemma \ref{lemma-Q-computations} and \eqref{eq:A51} of Lemma \ref{lemma:basic 2} that for $m\in[M,\log K_r]$ (recalling that $e^{-C_5/K}\leq r\leq e^{C_5/K}$ and $K_r\leq K$), \[\sigma_m^2=\sum_{e^{m-1}\leq k<e^m}\left(\frac{r^{2k}}{2k}+O(k^{-3/2})\right)=\frac{1}{2}+O(e^{-m/2})\] and \[\rho_m\sigma_m^2=\sum_{e^{m-1}\leq k<e^m}\left(\frac{r^{2k}\cos(k\theta/(\pi C_4\sqrt{N}))}{2k}+O(k^{-3/2})\right)=O(\sqrt{N}e^{-m}+e^{-m/2}).\]
In particular, 
$$\rho_m\ll \sqrt{N}e^{-m}+e^{-m/2},$$
and hence
\begin{align*}
    \prod_{\log(C_4\sqrt{N}/C_6)<m\leq \log K_r}\sqrt{\frac{1+|\rho_m|}{1-|\rho_m|}}\ll 1,
\end{align*}
which serves as the equivalent of \eqref{eq:decor} and suffices for our purpose.

In the proof that deduces Proposition \ref{SZ-prop-10.2} from Proposition \ref{SZ-prop-11.1}, the only change is in \eqref{eq:A0M}, which is now replaced by
$$\E\left[\bone_{\{A_0(M)\leq A-5\log M\}}\,e^{2A_0(M)}\right]\leq \E[e^{2A_0(M)}]=\frac{\E[\exp({2\Re\sum_{k=C_4\sqrt{N}/C_6}^{e^M}X_kr^k/\sqrt{k}})]}{\exp({2\sum_{k=C_4\sqrt{N}/C_6}^{e^M}\mu_k})}\ll \sqrt{N}e^{-M}.$$
These altogether result in the following equivalent of Proposition \ref{SZ-prop-10.2}:
\begin{align*}
&\hspace{0.5cm}\E\left[\bone_{\cL(\theta_1)}|\underline{F}_{K,M_*}(re^{i\theta_1/(\pi C_4\sqrt{N})})|^2\bone_{\cL(\theta_2)}|\underline{F}_{K,M_*}(re^{i\theta_2/(\pi C_4\sqrt{N})})|^2\right]\\
&\ll A^2e^{2A}\frac{K_r^2/N^{3/2}}{\log K_r}\frac{\min\{K_r,2\pi\sqrt{N}/|\theta_1-\theta_2|\}}{(\log(\min\{K_r,2\pi\sqrt{N}/|\theta_1-\theta_2|\})-\log(C_4\sqrt{N}/C_6))^7}\\
&\ll A^2e^{2A}\frac{K_r^2/{N}}{\log N}\frac{\min\{\sqrt{N},2\pi/|\theta_1-\theta_2|\}}{(\log(\min\{K_r,2\pi\sqrt{N}/|\theta_1-\theta_2|\})-\log(C_4\sqrt{N}/C_6))^7}.
\end{align*} 
Integrating with respect to $\theta_1,\theta_2$ yields
\begin{align}
    \E\left[\left(\int_{\mathcal{L}}|\underline{F}_{K,M_*}(re^{i\theta/(\pi C_4\sqrt{N})})|^2\d\theta \right)^2\right]\ll A^2e^{2A}\frac{K_r^2/N}{\log N}.\label{eq:ub 2}
\end{align}

\sloppy Finally, inserting \eqref{eq:lb 2} and \eqref{eq:ub 2} into \eqref{eq:holder 2} leads to (note that we apply with $A\asymp 1$ for $q<1$)
$$\E\left[\left(\int_{-\pi}^{\pi}|\underline{F}_{K,M_*}(re^{i\theta/(\pi C_4\sqrt{N})})|^2\d\theta \right)^q\right]\gg \frac{(K_r/\sqrt{N\log K_r})^{2-q}}{(K_r^2/(N\log N)^{1-q}}\gg \frac{K_r^qN^{-q/2}}{(1+(1-q)\sqrt{\log N})^{q}}.$$
Combined with \eqref{eq:change of variable} proves \eqref{eq:Flow1}, and hence completing the proof of Proposition \ref{prop:uniform}.
\end{proof}

\begin{proof}[Proof of Proposition {\ref{prop:uniform2}}]
    For $1\leq j\leq \log(\pi C_4\sqrt{N})$, we introduce the decomposition
$$F_{K,M_*}(z)=\overline{F}_{K,M_*;j}(z)\times \underline{F}_{K,M_*;j}(z):=\exp\Big(\sum_{k=M_*}^{C_4\sqrt{N}/(C_6e^j)}\frac{X_k}{\sqrt{k}}z^k\Big)\times \exp\Big(\sum^K_{k=C_4\sqrt{N}/(C_6e^j)}\frac{X_k}{\sqrt{k}}z^k\Big).$$
Applying independence and rotational symmetry, we have
\begin{align*}
    &\hspace{0.5cm}\E\left[\left(\int_{ \frac{e^{j-1}}{C_4\sqrt{N}}\leq\theta\leq \frac{e^{j}}{C_4\sqrt{N}}}|{F}_{K,M_*}(re^{i\theta})|^2\d\theta \right)^q\right]\\
   &\leq \E\left[\left(\int_{ \frac{e^{j-1}}{C_4\sqrt{N}}\leq\theta\leq \frac{e^{j}}{C_4\sqrt{N}}}|\underline{F}_{K,M_*;j}(re^{i\theta})|^2\d\theta \right)^q\right]\times \E\bigg[\sup_{ \frac{e^{j-1}}{C_4\sqrt{N}}\leq\theta\leq \frac{e^{j}}{C_4\sqrt{N}}}|\overline{F}_{K,M_*;j}(re^{i\theta})|^{2q}\bigg]\\
   &=\E\left[\left(\int_{ \frac{e^{j-1}}{C_4\sqrt{N}}\leq\theta\leq \frac{e^{j}}{C_4\sqrt{N}}}|\underline{F}_{K,M_*;j}(re^{i\theta})|^2\d\theta \right)^q\right]\times \E\bigg[\sup_{ |\theta|\leq \frac{e^{j}+e^{j-1}}{2C_4\sqrt{N}}}|\overline{F}_{K,M_*;j}(re^{i\theta})|^{2q}\bigg].
\end{align*}
We decompose 
$$\sup_{ |\theta|\leq \frac{e^{j}+e^{j-1}}{2C_4\sqrt{N}}}|\overline{F}_{K,M_*;j}(re^{i\theta})|^{2q}=|\overline{F}_{K,M_*;j}(r)|^{2q}\times \exp\bigg(2q\sup_{ |\theta|\leq \frac{e^{j}+e^{j-1}}{2C_4\sqrt{N}}}\sum_{k=M_*}^{C_4\sqrt{N}/(C_6e^j)}\frac{r^k}{\sqrt{k}}\Re(X_ke^{ik\theta}-X_k)\bigg).$$
The change of measure formula \eqref{eq:hat Q} together with \eqref{l} yield
\begin{align*}
    \E\bigg[\sup_{ |\theta|\leq \frac{e^{j}+e^{j-1}}{2C_4\sqrt{N}}}|\overline{F}_{K,M_*;j}(re^{i\theta})|^{2q}\bigg]&=\Big(\frac{C_4\sqrt{N}}{C_6M_*e^j}\Big)^{q^2}\E^{\widehat{\Q}}\bigg[\exp\bigg(2q\sup_{ |\theta|\leq \frac{e^{j}+e^{j-1}}{2C_4\sqrt{N}}}\sum_{k=M_*}^{C_4\sqrt{N}/(C_6e^j)}\frac{r^k}{\sqrt{k}}\Re(X_ke^{ik\theta}-X_k)\bigg)\bigg].
\end{align*}
Moreover, by \eqref{eq:chaining2},
$$ \E\bigg[\sup_{ |\theta|\leq \frac{e^{j}+e^{j-1}}{2C_4\sqrt{N}}}|\overline{F}_{K,M_*;j}(re^{i\theta})|^{2q}\bigg]\ll \Big(\frac{C_4\sqrt{N}}{C_6M_*e^j}\Big)^{q^2}\ll \Big(\frac{C_4\sqrt{N}}{e^j}\Big)^{q^2}.$$
Therefore,
$$\E\left[\left(\int_{ \frac{e^{j-1}}{C_4\sqrt{N}}\leq\theta\leq \frac{e^{j}}{C_4\sqrt{N}}}|{F}_{K,M_*}(re^{i\theta})|^2\d\theta \right)^q\right]\leq \E\left[\left(\int_{ \frac{e^{j-1}}{C_4\sqrt{N}}\leq\theta\leq \frac{e^{j}}{C_4\sqrt{N}}}|\underline{F}_{K,M_*;j}(re^{i\theta})|^2\d\theta \right)^q\right]\times \Big(\frac{C_4\sqrt{N}}{e^j}\Big)^{q^2}.$$
To show \eqref{eq:jub1}, it remains to prove 
\begin{align}
    \E\left[\left(\int_{ \frac{e^{j-1}}{C_4\sqrt{N}}\leq\theta\leq \frac{e^{j}}{C_4\sqrt{N}}}|\underline{F}_{K,M_*;j}(re^{i\theta})|^2\d\theta \right)^q\right]\ll \frac{e^{2qj}(K/N)^q}{(1+(1-q)\sqrt{\log N})^{q}},\label{eq:Flow2}
\end{align}
where the asymptotic constant does not depend on $j,K,N$. 
The case of $q=1$ being a consequence of Lemma \ref{SZ-lemma2.2!}, we assume that $q<1$.
To this end, we adapt the proof of Proposition \ref{SZ-prop3.2}. Recall Definition \ref{def:Gr}. We first claim that
\begin{align}
    \E\left[\bone_{\mathcal{G}_r(A;K)}\int_{ \frac{e^{j-1}}{C_4\sqrt{N}}\leq\theta\leq \frac{e^{j}}{C_4\sqrt{N}}}|\underline{F}_{K,M_*;j}(re^{i\theta})|^2\d\theta \right]\ll\frac{AKe^{2j}}{N\sqrt{\log N}}.\label{eq:theta ub}
\end{align}
\sloppy The only change compared to Proposition \ref{SZ-prop5.3} is that we replace the definition of $M$ therein by $\max\{M,\log(C_4\sqrt{N}/(C_6e^j))\}$. This leads to 
$$\E\left[\bone_{\mathcal{G}_r(A,\theta;K)}|\underline{F}_{K,M_*;j}(re^{i\theta})|^2 \right]\ll\frac{AKe^j}{\sqrt{N\log K}}
.$$
Integration with respect to $\theta$ then gives \eqref{eq:theta ub}. Applying the same argument that deduces Proposition \ref{SZ-prop3.2} from Propositions \ref{SZ-prop5.2} and  \ref{SZ-prop5.3} then leads to \eqref{eq:Flow2} and hence \eqref{eq:jub1}. A similar argument establishes \eqref{eq:jub2} and can be omitted. 
\end{proof}

\subsubsection{Low moments of weighted mass of truncated chaos}\label{sec: weighted partial mass}

Recall \eqref{eq:ujuj}. 
Define 
$$\widetilde{F}_{K,M_*,m}(z):=\exp\Big(\sum_{M_*\leq k\leq K}\frac{X_k}{\sqrt{k}}z^k\Big)\times\Big(\sum_{j:|j-m|\leq N^{9/10}}u(j)z^j\Big),$$
our goal is to show the following.

\begin{proposition}
    \label{prop:christmas}Let $K_r$ be such that  $\log K_r$ is the largest integer with $K_r\le\min\{\frac{-1}{4\log r},K\}$. Uniformly for  $r\in[e^{-C_5/K},e^{C_5/K}]$ and $K\gg\sqrt{N}$,
\begin{align}
   N^{q^2/2}\left(\frac{r^{2m}K_r}{1+(1-q)\sqrt{\log K_r}}\right)^q\ll  \E\left[\left(\int_{-\pi}^{\pi}|\widetilde{F}_{K,M_*,m}(re^{i\theta})|^2\d\theta \right)^q\right]\ll N^{q^2/2}\left(\frac{r^{2m}K}{1+(1-q)\sqrt{\log K}}\right)^q,\label{eq:real82replication}
\end{align}
where the asymptotic constants depend on $q$ only. In particular, with the choice $K=N/2$,
$$N^{q^2/2}\left(\frac{r^{2m}K_r}{1+(1-q)\sqrt{\log K_r}}\right)^q\ll  \E\left[\left(\int_{-\pi}^{\pi}|\widetilde{F}_{K,M_*,m}(re^{i\theta})|^2\d\theta \right)^q\right]\ll N^{q^2/2}\left(\frac{r^{2m}N}{1+(1-q)\sqrt{\log N}}\right)^q.$$
\end{proposition}

Recall the {truncated} chaos $F_{K,M_*}$ from \eqref{eq:FK}. 
We have
\begin{align}
    \frac{|\widetilde{F}_{K,M_*,m}(re^{i\theta})|}{|F_{K,M_*}(re^{i\theta})|}= \bigg|\sum_{j:|j-m|\leq N^{9/10}}e^{ij(\tau_1+\theta)}\,\frac{|X_1|^{j-m}}{j!/m!}r^j\bigg|.\label{eq:tilde F decomp}
\end{align}
The following two lemmas deal with the right-hand side of \eqref{eq:tilde F decomp}, whose proofs are elementary and deferred to Appendix \ref{sec:deferred proofs}. Recall the constants $C_4,C_5$.

\begin{lemma}\label{lemma:sum by parts}Uniformly in $N$ large enough,  $m\in[N/6,N/3]$, $x_1\in[m,m+1)$,  $\tau_1\in[-\pi,\pi)$, and $r\in[e^{-C_5/N},e^{C_5/N}]$, we have 
$$\bigg|\sum_{j:|j-m|\leq N^{9/10}}e^{ij\tau_1}\,\frac{|x_1|^{j-m}}{j!/m!}r^j\bigg|
    \ll \min\left\{\frac{r^m}{|\tau_1|},\sqrt{N}r^m\right\}.$$
\end{lemma}

\begin{lemma}\label{lemma:sum by parts2}There exists a large universal constant $C_4>0$ such that the following holds. Uniformly in $N$ large enough,  $m\in[N/6,N/3]$, $x_1\in[m,m+1)$,  $|\tau|\leq 1/(C_4\sqrt{N})$, and $r\in[e^{-C_5/N},e^{C_5/N}]$, we have 
$$\bigg|\sum_{j:|j-m|\leq N^{9/10}}e^{ij\tau}\,\frac{|x_1|^{j-m}}{j!/m!}r^j\bigg|
    \gg \sqrt{N}r^m.$$
\end{lemma}
We define $C_4$ so that the conditions in Lemma \ref{lemma:sum by parts2} are satisfied. In particular, we may remove the $C_4$ that appears in \eqref{eq:jub1} and \eqref{eq:jub2}. 

\begin{proof}[Proof of Proposition \ref{prop:christmas}]
 For the first part of \eqref{eq:real82replication}, we apply Lemma \ref{lemma:sum by parts2} and Proposition \ref{prop:uniform} to obtain 
\begin{align*}
    \E\left[\left(\int_{-\pi}^{\pi}|\widetilde{F}_{K,M_*,m}(re^{i\theta})|^2\d\theta \right)^q\right]&\geq \E\left[\left(\int_{|\theta+\tau_1|\leq \frac{1}{C_4\sqrt{N}}}|\widetilde{F}_{K,M_*,m}(re^{i\theta})|^2\d\theta \right)^q\right]\\
    &\gg N^qr^{2qm}\E\left[\left(\int_{|\theta+\tau_1|\leq \frac{1}{C_4\sqrt{N}}}|{F}_{K,M_*}(re^{i\theta})|^2\d\theta \right)^q\right]\\
    &\gg N^{q^2/2}\left(\frac{r^{2m}K_r}{1+(1-q)\sqrt{\log N}}\right)^q,
\end{align*}where in the last step we use rotational symmetry conditionally on $\tau_1$. 
  For the second part of \eqref{eq:real82replication}, we apply Lemma \ref{lemma:sum by parts}, Proposition \ref{prop:uniform2} (conditionally on $\tau_1$ and using rotational symmetry), and $|a+b|^q\leq |a|^q+|b|^q$ for $q\in(0,1]$ to get
  \begin{align*}
      &\hspace{0.5cm}\E\bigg[\bigg(\int_{-\pi}^\pi|\widetilde{F}_{K,M_*,m}(re^{i\theta})|^2\d\theta \bigg)^q\bigg]\\
      &\leq \sum_{j=1}^{\log(\pi C_4\sqrt{N})}\E\bigg[\bigg(\int_{ \frac{e^{j-1}}{C_4\sqrt{N}}\leq|\theta+\tau_1|\leq \frac{e^{j}}{C_4\sqrt{N}}}|\widetilde{F}_{K,M_*,m}(re^{i\theta})|^2\d\theta \bigg)^q\bigg]+\E\bigg[\bigg(\int_{|\theta+\tau_1|\leq \frac{1}{C_4\sqrt{N}}}|\widetilde{F}_{K,M_*,m}(re^{i\theta})|^2\d\theta \bigg)^q\bigg]\\
      &\ll \sum_{j=1}^{\log(\pi C_4\sqrt{N})}r^{2qm}\Big(\frac{e^{j}}{C_4\sqrt{N}}\Big)^{-2q}\E\bigg[\bigg(\int_{ \frac{e^{j-1}}{C_4\sqrt{N}}\leq|\theta+\tau_1|\leq \frac{e^{j}}{C_4\sqrt{N}}}|{F}_{K,M_*}(re^{i\theta})|^2\d\theta \bigg)^q\bigg]\\
      &\hspace{5cm}+N^qr^{2qm}\E\bigg[\bigg(\int_{|\theta+\tau_1|\leq \frac{1}{C_4\sqrt{N}}}|{F}_{K,M_*}(re^{i\theta})|^2\d\theta \bigg)^q\bigg]\\
      &\ll \sum_{j=1}^{\log(\pi C_4\sqrt{N})}r^{2qm} e^{-2qj}N^q \times \Big(\frac{e^{j-1}}{\sqrt{N}}\Big)^{2q-q^2} \frac{K^q}{(1+(1-q)\sqrt{\log N})^{q}}+N^qr^{2qm}\times\frac{N^{-q+q^2/2}K^q}{(1+(1-q)\sqrt{\log N})^{q}} \\
      &\ll N^{q^2/2}\left(\frac{r^{2m}K}{1+(1-q)\sqrt{\log N}}\right)^q.
  \end{align*}
This completes the proof.
\end{proof}

\subsubsection{Proof of the lower bound}\label{sec:critical-gamma=2q-LB}

Observe from \eqref{eq:sum formula} that
$$\sum_{m=N/6}^{N/3} \frac{m^{-q+q^2/2}}{(1+(1-q)\sqrt{\log(N-m)})^{q}}\gg \frac{N^{1-q+q^2/2}}{(1+(1-q)\sqrt{\log N})^q}\asymp \sum_{m=0}^{N-2}\frac{m^{-q+q^2/2}}{(1+(1-q)\sqrt{\log(N-m)})^{q}}.$$
This hints at the strategy of restricting to the event $m\approx |X_1|\in[N/6,N/3]$ in order to achieve an asymptotic lower bound for $\E[|A_N|^{2q}]$, which has considerably reduced technicality. To this end, let us write
\begin{align}
    \E[|A_N|^{2q}]\geq \sum_{m=N/6}^{N/3}\E\Bigg[\bigg|\sum_{\substack{\lambda\in\cP_N}}\prod_{k\geq 1}\left(\frac{X_k}{\sqrt{k}}\right)^{m_k}\frac{1}{m_k!}\bigg|^{2q}\bone_{\{|X_1|\in[m,m+1)\}}\Bigg].\label{eq:lbfirst step}
\end{align}
Fix $m\in[N/6,N/3]$. In view of the discussions in Section \ref{sec:main ideas}, we expect that 
\begin{align}
    \E\Bigg[\bigg|\sum_{\substack{\lambda\in\cP_N}}\prod_{k\geq 1}\left(\frac{X_k}{\sqrt{k}}\right)^{m_k}\frac{1}{m_k!}\bigg|^{2q}\bone_{\{|X_1|\in[m,m+1)\}}\Bigg]\asymp \E\Bigg[\bigg|\sum_{\substack{\lambda\in\cP_N\\ |m_1(\lambda)-m|\leq N^{9/10}}}\prod_{k\geq 1}\left(\frac{X_k}{\sqrt{k}}\right)^{m_k}\frac{1}{m_k!}\bigg|^{2q}\bone_{\{|X_1|\in[m,m+1)\}}\Bigg].\label{w}
\end{align}
This is confirmed by the following lemma.
\begin{lemma}\label{lemma:tailsmall}
For $m\in[N/6,N/3]$,
$$\E\Bigg[\bigg|\sum_{\substack{\lambda\in\cP_N\\ |m_1(\lambda)-m|> N^{9/10}}}\prod_{k\geq 1}\left(\frac{X_k}{\sqrt{k}}\right)^{m_k}\frac{1}{m_k!}\bigg|^{2q}\bone_{\{|X_1|\in[m,m+1)\}}\Bigg]=o(N^{-2}).$$
\end{lemma}

\begin{proof}
We focus first on the sum over $m_1(\lambda)\geq m$. Using in turn concavity (and similarly Minkowski's inequality for $q>1/2$), independence, \eqref{eq:gamma}, and Proposition \ref{prop:strengthened}, we have
\begin{align*}
    &\hspace{0.5cm}\E\Bigg[\bigg|\sum_{\substack{\lambda\in\cP_N\\ m_1(\lambda)-m>  N^{9/10}}}\prod_{k\geq 1}\left(\frac{X_k}{\sqrt{k}}\right)^{m_k}\frac{1}{m_k!}\bigg|^{2q}\bone_{\{|X_1|\in[m,m+1)\}}\Bigg]\\
    &\leq \sum_{j=m+N^{9/10}}^{N}\E\Bigg[\bigg|\sum_{\substack{\lambda\in\cP_N\\ m_1(\lambda)=j}}\prod_{k\geq 2}\left(\frac{X_k}{\sqrt{k}}\right)^{m_k}\frac{1}{m_k!}\bigg|^{2q}\frac{|X_1|^{2qj}}{(j!)^{2q}}\bone_{\{|X_1|\in[m,m+1)\}}\Bigg]\\
    &\ll \sum_{j=m+N^{9/10}}^{N}\frac{m^{2qj}e^{-\gamma m}}{(j!)^{2q}}\E\Bigg[\bigg|\sum_{\substack{\lambda\in\cP_{N-j}\\ m_1(\lambda)=0}}\prod_{k\geq 2}\left(\frac{X_k}{\sqrt{k}}\right)^{m_k}\frac{1}{m_k!}\bigg|^{2q}\Bigg]\\
    &\ll \sum_{j=m+N^{9/10}}^{N}\Big(\frac{m}{j}\Big)^{\gamma j}e^{\gamma(j-m)}j^{-q}.
\end{align*}
Using Taylor's expansion (that $-\log(1-x)\leq -x-x^2/2$ for $0<x<1$), we have the bound that for $j'\geq 0$,
\begin{align}
   e^{j'}\Big(\frac{m}{m+j'}\Big)^{m+j'}\leq \exp\Big(-\frac{(j')^2}{2(m+j')}\Big).\label{eq:taylor}
\end{align}
Inserting $j'=j-m$ and using $m\geq N/6$ lead to 
$$\sum_{j=m+N^{9/10}}^{N}\Big(\frac{m}{j}\Big)^{\gamma j}e^{\gamma(j-m)}\leq \sum_{j'=N^{9/10}}^{N-m}\exp\Big(-\frac{\gamma(j')^2}{14m}\Big).$$
Since $m\in[N/6,N/3]$, we conclude that the above is $O(\exp(-\gamma N^{1/3}))=o(N^{-2})$ for $N$ large. The sum over $m_1(\lambda)<m$ is similar and omitted.   
\end{proof}

We will apply the same second moment approach from Section \ref{sec:lb} to estimate the right-hand side of \eqref{w}. We need a few preliminary computations, the proofs of which are elementary and deferred to the appendix.


\begin{lemma}\label{lemma:replicate thm 2}
    For $m\in[N/6,N/3]$ and a (random) function $u$ of the form \begin{align*}
    u(j)=e^{ij\tau_1}\,\frac{|x_1|^{j-m}}{j!/m!},
\end{align*}
where $|x_1|\in[m,m+1)$ and we recall that $\tau_1$ is uniformly distributed on $[-\pi,\pi]$ independent of anything else, 
 we have   \begin{align*}
     \E\Bigg[\bigg|\sum_{\substack{\lambda\in\cP_N\\ |m_1(\lambda)-m|\leq N^{9/10}}}u(m_1)\prod_{k\geq 2}\left(\frac{X_k}{\sqrt{k}}\right)^{m_k}\frac{1}{m_k!}\bigg|^{2q}\Bigg]\gg \frac{N^{q^2/2}}{(1+(1-q)\sqrt{\log N})^q}.
 \end{align*}
\end{lemma}

\begin{proof}
First, taking advantage of $|u(j)|\leq 1$, we may apply a similar symmetrization argument that deduced Theorem \ref{thm:main} equation \eqref{eq:asymp LT} from Proposition \ref{prop:superC-asymp-truncated-moments} in Section \ref{sec:LT reduction}, so that it suffices to prove
\begin{align}
     \E\Bigg[\bigg|\sum_{\substack{\lambda\in\cP_N\\ |m_1(\lambda)-m|\leq N^{9/10}\\\forall 2\leq k<M_*,\,m_k(\lambda)=0}}u(m_1)\prod_{k\geq M_*}\left(\frac{X_k}{\sqrt{k}}\right)^{m_k}\frac{1}{m_k!}\bigg|^{2q}\Bigg]\gg \frac{N^{q^2/2}}{(1+(1-q)\sqrt{\log N})^q}.\label{eq:N^q more 2}
 \end{align}
The proof of \eqref{eq:N^q more 2} follows from a straightforward adaptation of the lower bound part for the universality phase. First, in the proof of Proposition \ref{SZ-Prop-8.1}, we have used the relation \eqref{eq:FA relation}. 
If we instead define 
\begin{align}
    \widetilde{F}_{K,M_*,m}(z):=\exp\Big(\sum_{M_*\leq k\leq K}\frac{X_k}{\sqrt{k}}z^k\Big)\times\Big(\sum_{j:|j-m|\leq N^{9/10}}u(j)z^j\Big),\label{eq:FKMm}
\end{align}
then \eqref{eq:FA relation} has the analogue
$$\widetilde{F}_{N/2,M_*,m}(z)=\sum_{n=0}^\infty\bigg(\sum_{\substack{\lambda\in\cP_n\\\lambda_1\leq N/2\\ |m_1(\lambda)-m|\leq N^{9/10}\\\forall 2\leq k<M_*,\,m_k(\lambda)=0}}u(m_1)\prod_{k=M_*}^{N/2}\left(\frac{X_k}{\sqrt{k}}\right)^{m_k}\frac{1}{m_k!}\bigg)z^n.$$
Therefore, the following equivalent of Proposition \ref{SZ-Prop-8.1} holds:
\begin{align}\begin{split}
    &\hspace{0.5cm}\E\Bigg[\bigg|\sum_{\substack{\lambda\in\cP_N\\ |m_1(\lambda)-m|\leq N^{9/10}\\ \forall 2\leq k<M_*,\,m_k(\lambda)=0}}u(m_1)\prod_{k\geq 2}\left(\frac{X_k}{\sqrt{k}}\right)^{m_k}\frac{1}{m_k!}\bigg|^{2q}\Bigg]\\
    &\gg \frac{1}{N^q}\left(\E\left[\left(\int_{-\pi}^{\pi}|\widetilde{F}_{N/2,M_*,m}(re^{i\theta})|^2\d\theta \right)^q\right]-\E\left[r^{Nq}\left(\int_{-\pi}^{\pi}|\widetilde{F}_{N/2,M_*,m}(e^{i\theta})|^2\d\theta \right)^q\right]\right),
\end{split}\label{eq:3.10 replic}
\end{align}
and hence it suffices to prove the counterparts of Proposition \ref{SZ-Prop-8.2} and Proposition \ref{SZ-prop3.2} with $F_{K,M_*}$ replaced by $\widetilde{F}_{K,M_*,m}$. 
Indeed, by Proposition \ref{prop:christmas},
\begin{align}
    \E\left[\left(\int_{-\pi}^{\pi}|\widetilde{F}_{K,M_*,m}(re^{i\theta})|^2\d\theta \right)^q\right]\gg N^{q^2/2}\left(\frac{r^{2m}K_r}{1+(1-q)\sqrt{\log K_r}}\right)^q,\label{eq:82replication}
\end{align}
where $r=e^{-C_5/N}$ for some fixed large constant $C_5$, and
\begin{align}
    \E\left[\left(\int_{-\pi}^{\pi}|\widetilde{F}_{K,M_*,m}(e^{i\theta})|^2\d\theta \right)^q\right]\ll N^{q^2/2}\left(\frac{K}{1+(1-q)\sqrt{\log K}}\right)^q.\label{eq:32 replication}
\end{align}
Note that the constants in \eqref{eq:82replication} and \eqref{eq:32 replication} do not depend on $C_5$. 
Inserting \eqref{eq:82replication} and \eqref{eq:32 replication} into \eqref{eq:3.10 replic} yields that for some constant $C_{11}>0$,
\begin{align*}
    &\hspace{0.5cm}\E\Bigg[\bigg|\sum_{\substack{\lambda\in\cP_N\\ |m_1(\lambda)-m|\leq N^{9/10}\\ \forall 2\leq k<M_*,\,m_k(\lambda)=0}}u(m_1)\prod_{k\geq 2}\Big(\frac{X_k}{\sqrt{k}}\Big)^{m_k}\frac{1}{m_k!}\bigg|^{2q}\Bigg]\\
    &\geq \frac{1}{N^q}\bigg(\frac{N^{q^2/2}}{C_{11}}\Big(\frac{e^{-2C_5m/N}K_r}{1+(1-q)\sqrt{\log K_r}}\Big)^q-C_{11}e^{-C_5q}N^{q^2/2}\Big(\frac{K}{1+(1-q)\sqrt{\log K}}\Big)^q\bigg)
\end{align*}
Since $m/N\in[1/6,1/3]$, the right-hand side is $\gg N^{q^2/2}/(1+(1-q)\sqrt{\log N})^q$ for $C_5$ picked large enough. This completes the proof. 
\end{proof}

\begin{proof}[Proof of lower bound of \eqref{eq:gamma=2q result}]
     By \eqref{eq:lbfirst step} and Lemmas \ref{lemma:tailsmall} and \ref{lemma:replicate thm 2}, we have
\begin{align*}
    \E[|A_N|^{2q}]&\geq\sum_{m=N/6}^{N/3}\E\Bigg[\bigg|\sum_{\substack{\lambda\in\cP_N\\ |m_1(\lambda)-m|\leq N^{9/10}}}\prod_{k\geq 1}\left(\frac{X_k}{\sqrt{k}}\right)^{m_k}\frac{1}{m_k!}\bigg|^{2q}\bone_{\{|X_1|\in[m,m+1)\}}\Bigg]-o(N^{-1})\\
    &\geq\sum_{m=N/6}^{N/3} \E\Bigg[\bigg|\sum_{\substack{\lambda\in\cP_N\\ |m_1(\lambda)-m|\leq N^{9/10}}}u(m_1)\prod_{k\geq 2}\left(\frac{X_k}{\sqrt{k}}\right)^{m_k}\frac{1}{m_k!}\bigg|^{2q}\Bigg]\E\Big[\Big(\frac{|X_1|^m}{m!}\Big)^{2q}\bone_{\{|X_1|\in[m,m+1)\}}\Big]-o(N^{-1})\\
    &\gg \sum_{m=N/6}^{N/3}\frac{N^{q^2/2}}{(1+(1-q)\sqrt{\log N})^q}\,N^{-q}-o(N^{-1})\\
    &\gg \frac{N^{1-q+q^2/2}}{(1+(1-q)\sqrt{\log N})^q},
\end{align*}
as desired.
\end{proof}

\subsubsection{Proof of the upper bound}\label{sec:critical-gamma=2q-UB}
 Fix a large constant $C_{12}>0$ to be determined. We expect that for each $m\in[C_{12},N-C_{12}]$, 
\begin{align}
    \E\Bigg[\bigg|\sum_{\substack{\lambda\in\cP_N}}\prod_{k\geq 1}\left(\frac{X_k}{\sqrt{k}}\right)^{m_k}\frac{1}{m_k!}\bigg|^{2q}\bone_{\{|X_1|\in[m,m+1)\}}\Bigg]\approx \E\Bigg[\bigg|\sum_{\substack{\lambda\in\cP_N\\ |m_1-m|\leq N^{9/10}}}\prod_{k\geq 1}\left(\frac{X_k}{\sqrt{k}}\right)^{m_k}\frac{1}{m_k!}\bigg|^{2q}\bone_{\{|X_1|\in[m,m+1)\}}\Bigg],\label{i}
\end{align}
and that the contribution to $\E[|A_N|^{2q}]$ from the event $\{|X_1|\not\in[C_{12},N-C_{12}]\}$ does not exceed the order of the lower bound established in Section \ref{sec:critical-gamma=2q-LB}.
This is the purpose of the following result. 

\begin{lemma}\label{lemma:ubfirst step}
    It holds that
    \begin{align*}\E[|A_N|^{2q}]\ll \sum_{m=C_{12}}^{N-C_{12}}\E\Bigg[\bigg|\sum_{\substack{\lambda\in\cP_N\\ |m_1-m|\leq N^{9/10}}}\prod_{k\geq 1}\left(\frac{X_k}{\sqrt{k}}\right)^{m_k}\frac{1}{m_k!}\bigg|^{2q}\bone_{\{|X_1|\in[m,m+1)\}}\Bigg]+\frac{N^{1-q+q^2/2}}{(1+(1-q)\sqrt{\log N})^q}.\end{align*}
\end{lemma}
\begin{proof}
The first step is to control the sum
$$\sum_{m\not\in[C_{12},N-C_{12}]}\E\Bigg[\bigg|\sum_{\substack{\lambda\in\cP_N}}\prod_{k\geq 1}\left(\frac{X_k}{\sqrt{k}}\right)^{m_k}\frac{1}{m_k!}\bigg|^{2q}\bone_{\{|X_1|\in[m,m+1)\}}\Bigg].$$
We use the same approach as in Lemma \ref{lemma:tailsmall}. We illustrate the case $0<q<1/2$ first using concavity. Using  Proposition \ref{prop:strengthened} and \eqref{eq:taylor}, the contribution from the sum over $m\geq N-C_{12}$ can be bounded by
\begin{align*}
    &\hspace{0.5cm}\sum_{m\geq N-C_{12}}\E\Bigg[\bigg|\sum_{\substack{\lambda\in\cP_N}}\prod_{k\geq 1}\left(\frac{X_k}{\sqrt{k}}\right)^{m_k}\frac{1}{m_k!}\bigg|^{2q}\bone_{\{|X_1|\in[m,m+1)\}}\Bigg]\\
    &\leq \sum_{m\geq N-C_{12}}\sum_{j=0}^N\E\Bigg[\bigg|\sum_{\substack{\lambda\in\cP_N\\ m_1(\lambda)=j}}\prod_{k\geq 2}\left(\frac{X_k}{\sqrt{k}}\right)^{m_k}\frac{1}{m_k!}\bigg|^{2q}\Bigg]\times\E\Bigg[\frac{|X_1|^{2qj}}{(j!)^{2q}}\bone_{\{|X_1|\in[m,m+1)\}}\Bigg]\\
    &\leq \sum_{m\geq N-C_{12}}\sum_{j=0}^N \Big(\frac{m}{j}\Big)^{\gamma j}e^{\gamma(j-m)}j^{-q}(1+(1-q)\sqrt{\log(1+N-j)})^{-q}\\
    &\ll \sum_{m\geq N-C_{12}}\sum_{j=0}^Nj^{-q}e^{-(j-m)^2/(2j)}(1+(1-q)\sqrt{\log(1+N-j)})^{-q}\\
    &\ll N^{1-q}(1+(1-q)\sqrt{\log N})^{-q}.
\end{align*}
Likewise,
$$\sum_{m\leq C_{12}}\E\Bigg[\bigg|\sum_{\substack{\lambda\in\cP_N}}\prod_{k\geq 1}\left(\frac{X_k}{\sqrt{k}}\right)^{m_k}\frac{1}{m_k!}\bigg|^{2q}\bone_{\{|X_1|\in[m,m+1)\}}\Bigg]\ll \sum_{m\leq C_{12}}\sum_{j=0}^Nj^{-q}e^{-(j-m)^2/(2j)}\ll 1.$$

Suppose now that $q\geq 1/2$. By Minkowski's inequality, Proposition \ref{prop:strengthened}, and \eqref{eq:taylor}, there are constants $C_{13},C_{14}>0$ such that for each $m\geq N-C_{12}$,
\begin{align*}
  &\hspace{0.5cm} \E\Bigg[\bigg|\sum_{\substack{\lambda\in\cP_N}}\prod_{k\geq 1}\left(\frac{X_k}{\sqrt{k}}\right)^{m_k}\frac{1}{m_k!}\bigg|^{2q}\bone_{\{|X_1|\in[m,m+1)\}}\Bigg]^{1/(2q)}\\
  &\leq \sum_{j=0}^N\Bigg(\E\Bigg[\bigg|\sum_{\substack{\lambda\in\cP_N\\ m_1(\lambda)=j}}\prod_{k\geq 2}\left(\frac{X_k}{\sqrt{k}}\right)^{m_k}\frac{1}{m_k!}\bigg|^{2q}\Bigg]\times \E\Bigg[\frac{|X_1|^{2qj}}{(j!)^{2q}}\bone_{\{|X_1|\in[m,m+1)\}}\Bigg]\Bigg)^{1/(2q)}\\
   &\leq \sum_{j=0}^N \Bigg(\Big(\frac{m}{j}\Big)^{\gamma j}e^{\gamma(j-m)}j^{-q}(1+(1-q)\sqrt{\log(1+N-j)})^{-q}\Bigg)^{1/(2q)}\\
    &\ll \sum_{j=0}^Nj^{-1/2}e^{-(j-m)^2/(4qj)}(1+(1-q)\sqrt{\log(1+N-j)})^{-1/2}.
\end{align*}
Note that 
$$\frac{j^{-1/2}e^{-(j-m)^2/(4qj)}}{(1+(1-q)\sqrt{\log(1+N-j)})^{1/2}}
    \ll \begin{cases}
       j^{-1/2} (1+(1-q)\sqrt{\log(1+N-j)})^{-1/2}&\text{ if }j\geq m-C_{14}\sqrt{N};\\
        O(e^{-(m-j-C_{14}\sqrt{N})/C_{13}})&\text{ if }j< m-C_{14}\sqrt{N}.
    \end{cases}$$
Therefore,
\begin{align*}
     &\hspace{0.5cm}\E\Bigg[\sum_{m\geq N-C_{12}}\bigg|\sum_{\substack{\lambda\in\cP_N}}\prod_{k\geq 1}\left(\frac{X_k}{\sqrt{k}}\right)^{m_k}\frac{1}{m_k!}\bigg|^{2q}\bone_{\{|X_1|\in[m,m+1)\}}\Bigg]\\
     &\ll 1+\sum_{N-C_{12}\leq m\leq N+C_{14}\sqrt{N}}\Big((1+(1-q)\sqrt{\log N})^{-q}+1\Big)\ll \sqrt{N}
     \ll \frac{N^{1-q+q^2/2}}{(1+(1-q)\sqrt{\log N})^q}.
\end{align*}
The sum over $m\leq C_{12}$ can be similarly bounded by $\ll 1$. 

Finally, for $m\in[C_{12},N-C_{12}]$, we apply the same technique in Lemma \ref{lemma:tailsmall} that leads to
$$\E\Bigg[\bigg|\sum_{\substack{\lambda\in\cP_N\\ |m_1-m|> N^{9/10}}}\prod_{k\geq 1}\left(\frac{X_k}{\sqrt{k}}\right)^{m_k}\frac{1}{m_k!}\bigg|^{2q}\bone_{\{|X_1|\in[m,m+1)\}}\Bigg]\ll N^{-2}.$$
This completes the proof by concavity if $q<1/2$ and Minkowski's inequality if $q\geq 1/2$.
\end{proof}

\begin{proof}[Proof of upper bound of \eqref{eq:gamma=2q result}] 
    Conditioning on $R_1$, the right-hand side of \eqref{i} can be rewritten using independence as
\begin{align}
    \E\Bigg[\frac{|X_1|^{2qm}}{(m!)^{2q}}\bone_{\{|X_1|\in[m,m+1)\}}\Bigg]\times \E\Bigg[\bigg|\sum_{\substack{\lambda\in\cP_N\\ |m_1-m|\leq N^{9/10}}}\frac{e^{i\tau_1 (m_1-m)}|x_1|^{m_1-m}}{m_1!/m!}\prod_{k\geq 2}\left(\frac{X_k}{\sqrt{k}}\right)^{m_k}\frac{1}{m_k!}\bigg|^{2q} \Bigg],\label{e}
\end{align}where $x_1\in[m,m+1)$. 
The first expectation contributes $\asymp m^{-q}$, which directly follows from \eqref{eq:gamma}. For the second expectation, we use the same multiplicative chaos approach as in the universality phase. Similarly as done in Section \ref{sec:critical-gamma=2q-LB}, we first apply the same argument that deduced Theorem \ref{thm:main} equation \eqref{eq:asymp LT} from Proposition \ref{prop:superC-asymp-truncated-moments}  in Section \ref{sec:LT reduction}, to start the product in \eqref{e} from $k=M_*$ instead of from $k=2$. Define
$$A_{N,M_*,m}:=\sum_{\substack{\lambda\in\cP_N\\ |m_1-m|\leq N^{9/10}\\m_1=\dots=m_{M_*-1}=0}}u(m_1)\prod_{k\geq M_*}\left(\frac{X_k}{\sqrt{k}}\right)^{m_k}\frac{1}{m_k!}.$$
It remains to bound 
$$\E\Bigg[\bigg|\sum_{\substack{\lambda\in\cP_N\\ |m_1-m|\leq N^{9/10}\\ m_1=\dots=m_{M_*-1}=0}}\frac{e^{i\tau_1 (m_1-m)}|x_1|^{m_1-m}}{m_1!/m!}\prod_{k\geq M_*}\left(\frac{X_k}{\sqrt{k}}\right)^{m_k}\frac{1}{m_k!}\bigg|^{2q} \Bigg]=\E\left[\left|A_{N,M_*,m}\right|^{2q} \right]$$
from above. 
First, the same argument of Proposition \ref{SZ-prop3.1} applies (possibly with a different constant $C(q)$) with $F_{K,M_*}$ replaced by $\widetilde{F}_{K,M_*,m}$ defined in \eqref{eq:FKMm}: for $1/2\leq q\leq 1$,
\begin{align}
    \E\left[\left|A_{N,M_*,m}\right|^{2q} \right]^{1/(2q)}\ll \frac{1}{\sqrt{N}}\sum_{j=1}^J\E\left[\left(\frac{1}{2\pi}\int_{-\pi}^{\pi}|\widetilde{F}_{N/2^j,M_*,m}(\exp(j/N+i\theta))|^2 \d\theta \right)^q \right]^{1/(2q)}+\frac{1}{N},\label{eq:q1}
\end{align}
and for $0<q<1/2$,
\begin{align}
    \E\left[\left|A_{N,M_*,m}\right|^{2q} \right]\ll \frac{1}{N^q}\sum_{j=1}^J\E\left[\left(\frac{1}{2\pi}\int_{-\pi}^{\pi}|\widetilde{F}_{N/2^j,M_*,m}(\exp(j/N+i\theta))|^2 \d\theta \right)^q \right]+\frac{1}{N^{2q}}.\label{eq:q2}
\end{align}
These follow by arguing similarly as in the proof of Proposition \ref{prop:M*2}, and adapting Proposition \ref{SZ-prop3.1} to the case $0<q<1/2$ by applying concavity instead of Minkowski's inequality in \eqref{eq:SZ3.1Minkowski}. Inserting the upper bound of Proposition \ref{prop:christmas} (second inequality of \eqref{eq:real82replication}) into \eqref{eq:q1} and \eqref{eq:q2}
yields that for $q\in(0,1]$ and $m\in[C_{12},N-C_{12}]$,
\begin{align}
    \E\Bigg[\bigg|\sum_{\substack{\lambda\in\cP_N\\ |m_1-m|\leq N^{9/10}\\m_1=\dots=m_{M_*-1}=0}}u(m_1)\prod_{k\geq 2}\left(\frac{X_k}{\sqrt{k}}\right)^{m_k}\frac{1}{m_k!}\bigg|^{2q} \Bigg]\ll \frac{m^{q^2/2}}{(1+(1-q)\sqrt{\log(N-m)})^{q}}.\label{eq:ub key step}
\end{align}
 Combining Lemma \ref{lemma:ubfirst step} with \eqref{eq:ub key step}, we have
\begin{align*}
    \E[|A_N|^{2q}]&\ll \sum_{m=C_{12}}^{N-C_{12}}\E\Bigg[\bigg|\sum_{\substack{\lambda\in\cP_N\\ |m_1-m|\leq N^{9/10}}}\prod_{k\geq 1}\left(\frac{X_k}{\sqrt{k}}\right)^{m_k}\frac{1}{m_k!}\bigg|^{2q}\bone_{\{|X_1|\in[m,m+1)\}}\Bigg]+O\Big(\frac{N^{1-q+q^2/2}}{(1+(1-q)\sqrt{\log N})^q}\Big)\\
    &\ll \sum_{m=C_{12}}^{N-C_{12}}m^{-q}\,\frac{m^{q^2/2}}{(1+(1-q)\sqrt{\log(N-m)})^{q}}+O\Big(\frac{N^{1-q+q^2/2}}{(1+(1-q)\sqrt{\log N})^q}\Big)\\
    &\ll \frac{N^{1-q+q^2/2}}{(1+(1-q)\sqrt{\log N})^q},
\end{align*}as desired. 
\end{proof}

\section{Remaining proofs}

\subsection{Sharp tightness of scaled secular coefficients
}\label{subsection-coro}

In this subsection, we prove Corollary \ref{coro} using the universality result of Theorem \ref{thm:main} (i). 


\begin{proof}[Proof of Corollary \ref{coro}]Suppose  that $\E[e^{\ee|X_k|}]<\infty$, where we may without loss of generality assume that $\ee<1$.
For (i), we apply Markov's inequality to $|A_N|^{\ee/3}$, which yields
$$\P\big(|A_N|(\log N)^{1/4}>C\big)=\P\big(|A_N|^{\ee/3}>C^{\ee/3}(\log N)^{-\ee/12}\big)\leq \frac{\E[|A_N|^{\ee/3}]}{C^{\ee/3}(\log N)^{-\ee/12}}\leq \frac{C(\ee)}{C^{\ee/3}},$$
where $C(\ee)$ is the asymptotic constant in \eqref{eq:asymp LT}. Picking $C$ large enough gives the tightness of $|A_N|(\log N)^{1/4}$.

    For (ii), we apply the Paley-Zygmund inequality to $|A_N|^{\ee/3}$, which gives that for some large constant $C(\ee)>0$, uniformly in $N$,
$$\P\bigg(|A_N|>\frac{(\log N)^{-1/4}}{C(\ee)}\bigg)\geq \P\bigg(|A_N|^{\ee/3}>\frac{\E[|A_N|^{\ee/3}]}{2}\bigg)\geq \frac{(\E[|A_N|^{\ee/3}])^2}{4\E[|A_N|^{2\ee/3}]}\geq\frac{1}{C(\ee)},$$where the last step follows from Theorem \ref{thm:main} (i).
\end{proof}

\subsection{Regularity of critical non-Gaussian holomorphic chaos}\label{sec:reg}

In this subsection, we prove Corollary \ref{coro-reg} using the universality result Theorem \ref{thm:main} (i).

\begin{proof}[Proof of Corollary \ref{coro-reg}]
    By definition of the Sobolev space $H^s$, it suffices to show that the series \[C_s:=\sum_{n=0}^\infty(1+n^2)^s|A_n|^2\] converges almost surely for desired $s$. If there is some $\gamma>2$ such that $\E[e^{\gamma|X_k|}]<\infty$, then by Theorem \ref{thm:main} (i) we have \[\E[C_s]\ll \sum_{n=0}^\infty (1+n)^{2s}<\infty\] for $s<-1/2$. Therefore, the corresponding $\hmc_1$ is in $H^s$ almost surely for $s<-1/2$. 
    Similarly, if $X_k$ satisfies (EXP) with $\gamma\in(0,2]$, then for any $0<q<\gamma/2$, Theorem \ref{thm:main} (i) implies \[\E[|C_s|^q]\leq \sum_{n=0}^{\infty}(1+n^2)^{qs}\E[|A_n|^{2q}]\ll \sum_{n=0}^\infty\frac{(1+n)^{2qs}}{(1+(1-q)\sqrt{\log(1+n)})^{q}}<\infty\] for $s<-1/(2q)$. Since $q<\gamma/2$ is arbitrary the result follows.
\end{proof}

\begin{remark}
To show the other direction of irregularity, i.e.~non-Gaussian $\hmc_1$ is almost surely not $H^s$ for $s>-1/2$ in the first case of Corollary \ref{coro-reg} (and respectively $s>-1/\gamma$ in the second case), one may need a convergence in probability result of the total mass of critical non-Gaussian multiplicative chaos $$\frac{1}{2\pi}\int_{-\pi}^{\pi}\exp\bigg(2\Re\sum_{k=1}^{\infty}\frac{X_k}{\sqrt{k}}(re^{i\theta})^k\bigg)\d\theta$$ as $r\to 1$; see the proof of irregularity of Gaussian $\hmc$ in Section 6 of \citep{najnudel2023secular}. This problem, to the best of our knowledge, remains open. We refer to \citep{junnila2020multiplicative} and references therein for studies on sub-critical non-Gaussian multiplicative chaos.
\end{remark}

\section*{Statements and declarations}
\noindent The authors declare no competing interests.

\section*{Acknowledgment}
\noindent We thank Nicholas Cook, Adam Harper, and Max Wenqiang Xu for reading and providing thoughtful comments on our earlier drafts.
We are grateful to two anonymous referees for giving valuable advice on the previous version of our article.
We also thank  Kannan Soundararajan for enlightening discussions. H.G.~is partially supported by NSF grant DMS-2154029.

\bibliography{bibfile}
\bibliographystyle{plain}

\appendix

\section{Some technical computations regarding exponential moments}\label{sec:ubdef}
We collect some technical computations of (exponential) moments under distinct probability measures in this section. Recall the definitions in Section \ref{sec:prop2.1setup}.
\begin{lemma}\label{lemma:basic}
Fix $M_*\in\N$. For any $K$ large enough and $e^{-1/K}\leq r\leq e^{1/K}$,
    $$\exp\Big(\sum_{k=M_*}^K\frac{r^{2k}}{k}\Big)\asymp_{M_*} K.$$
\end{lemma}
\begin{remark} In practice, $M_*$ is a constant depending only on $q$ and the distribution of $X_k$, while $K$ grows to infinity with $N\to\infty$.\end{remark}
\begin{proof}
    Suppose that $1\leq r\leq e^{1/K}$. Then the lower bound $\exp\left(\sum_{k=M_*}^K\frac{r^{2k}}{k}\right)\gg K$ is trivial and there exists some constant $L>0$ such that
    \[\sum_{k=M_*}^K\frac{r^{2k}}{k}\leq \sum_{n=\floor{\log M_*}}^{\ceil{\log K}}\sum_{e^{n-1}\leq k<e^n}\frac{r^{2k}}{k}\leq \sum_{n=\floor{\log M_*}}^{\ceil{\log K}}e^{2e^n/K}\leq \log K+1+L\sum_{n=\floor{\log M_*}}^{\ceil{\log K}}\frac{e^n}{K}\leq \log K+L.\]
    The other case $e^{-1/K}\leq r\leq 1$ is similarly established.
\end{proof}

\begin{lemma}\label{SZ-lemma2.2}
   Suppose that $\E[e^{c_0|R_1|}]<\infty$ for some $c_0>0$. Fix a positive $\beta$. For any $K\geq 1$ and $e^{-1/K}\leq r\leq e^{1/K}$, we have the following asymptotics for $k_0\le k\le K$, where $k_0$ is some constant depending only on $c_0$ and $\beta$. 
    \begin{enumerate}[(i)]
        \item $\E\left[\exp\left(2\beta\Re \frac{X_k r^k}{\sqrt{k}}\right)\right] = \E\left[\exp\left(2\beta\frac{r^k}{\sqrt{k}}R_k\cos(\tau_k)\right)\right]=1+\beta^2\frac{r^{2k}}{k}+O(k^{-3/2});$
        \item $\E\left[\exp\left(2\beta\frac{r^k }{\sqrt{k}}R_k\cos(\tau_k)\right)R_k\cos(\tau_k)\right]=\beta\frac{r^k}{\sqrt{k}}+O(k^{-1});$
        \item $\E\left[\exp\left(2\beta\frac{r^k }{\sqrt{k}}R_k\cos(\tau_k)\right)R_k^2\right] = 1+O(k^{-{1}/{2}});$
        \item For any 
$\alpha\in\mathbb{R}$, $\E\left[\exp\left(2\beta\frac{r^k }{\sqrt{k}}R_k\cos(\tau_k)\right)R_k^2\cos^2(\tau_k+\alpha)\right] = \frac{1}{2}+O(k^{-{1}/{2}});$
        \item $\E\left[\exp\left(2\beta\frac{r^k }{\sqrt{k}}R_k\cos(\tau_k)\right)|R_k\cos(\tau_k)|^3\right]\ll 1.$
    \end{enumerate}
  Moreover, if $X_1$ is compactly supported, the above statements hold for all $1\leq k\leq K$ (i.e., $k_0=1$ works).
\end{lemma}

\begin{proof}
Recall that $\E[\cos(\tau_k)^2]=\frac{1}{2}$. For (i), by  Taylor's expansion  and Fubini's theorem,
\begin{align*}
    \E\left[\exp\left(2\beta\Re \frac{X_k r^k}{\sqrt{k}}\right)\right] &= \E\left[\exp\left(2\beta \frac{r^k}{\sqrt{k}}R_k\cos(\tau_k)\right)\right]=: 1+\beta^2\frac{r^{2k}}{k}+E_r(k,\beta),
\end{align*}
where, after noticing $k^{({m-3})/({2m})}\ge k^{1/8}$ for $m\ge 4$, we have
\begin{align*}
    |E_r(k,\beta)| &= \left|\sum_{m=3}^\infty\frac{2^m\beta^m r^{mk}}{m!k^{m/2}}\E\left[(R_k\cos(\tau_k))^m\right]\right|\\
    &\ll k^{-3/2}\left(\E[|R_k|^3]+\sum_{m=4}^\infty\frac{(2e\beta/k^{1/8})^m}{m!}\E[|R_k|^m]\right)\\
    &\le k^{-3/2}\left(\E[|R_k|^3]+\E\left[e^{c_0|R_k|}\right]\right)=O(k^{-3/2})
\end{align*}for $k$ larger than some universal constant. Likewise, we have for (ii),
\begin{align*}
    \E\left[\exp\left(2\beta\Re \frac{X_k r^k}{\sqrt{k}}\right)\Re X_k\right] &= \E\left[\exp\left(2\beta\frac{r^k }{\sqrt{k}}R_k\cos(\tau_k)\right)R_k\cos(\tau_k)\right]=: \beta\frac{r^k}{\sqrt{k}}+\widetilde{E}_r(k,\beta),
\end{align*}
where for $k$ larger than some universal constant (and additionally the fact $m\le (\frac{3}{2})^m$ for $ m\ge 1$)
\begin{align*}
    |\widetilde{E}_r(k,\beta)| = \left|\sum_{m=2}^\infty\frac{2^m\beta^mr^{mk}}{m!k^{m/2}}\E\left[(R_k\cos(\tau_k))^{m+1}\right]\right| \le k^{-1}\left(\E[|R_k|^3]+\E\left[e^{c_0|R_k|}\right]\right)=O(k^{-1}).
\end{align*}
For (iv), we use only the first term in the expansion and get
\begin{align*}
   &\hspace{0.5cm} \E\left[\exp\left(2\beta\frac{r^k }{\sqrt{k}}R_k\cos(\tau_k)\right)R_k^2\cos^2(\tau_k+\alpha)\right]\\
   &= \E[R_k^2\cos^2(\tau_k+\alpha)]+\E\left[\sum_{m=1}^\infty \frac{(2\beta r^k)^m}{m!k^{m/2}}R_k^{m+2}\cos^{m}(\tau_k)\cos^2(\tau_k+\alpha)\right]= \frac{1}{2}+O(k^{-\frac 12}). 
\end{align*}
For (iii), similarly as in (iv):
\begin{align*}
    \E\left[\exp\left(2\beta\frac{r^k }{\sqrt{k}}R_k\cos(\tau_k)\right)R_k^2\right] = \E[R_k^2]+\E\left[\sum_{m=1}^\infty\frac{(2\beta r^k)^m}{m!k^{m/2}}R_k^{m+2}\cos^m(\tau_k)\right] = 1+O(k^{-{1}/{2}}).
\end{align*}
For (v), similarly we have:
\begin{align*}
    \E\left[\exp\left(2\beta\frac{r^k }{\sqrt{k}}R_k\cos(\tau_k)\right)|R_k|^3\right] \le \E[|R_k|^3]+\E\left[\sum_{m=1}^\infty\frac{(2\beta r^k)^m}{m!k^{m/2}}|R_k|^{m+3}\cos^m(\tau_k)\right] \ll 1.
\end{align*}
This completes the proof.
\end{proof}

\begin{lemma}\label{SZ-lemma2.2!}
      Assume the same settings of Lemma \ref{SZ-lemma2.2}. For $\theta\in[-\pi,\pi)$, $e^{-1/K}\leq r\leq e^{1/K}$, and any $M_*$ large enough depending on $c_0$,
    \[\E\big[|F_{K,M_*}(re^{i\theta})|^{2}\big] \asymp_{M_*} K.\]
\end{lemma}
\begin{proof}
Using independence of $\{(R_k,\tau_k)\}_{k\ge 1}$ and rotational invariance of $X_k$, as well as Lemma \ref{SZ-lemma2.2} (i), we have
\begin{align*}
    \E\big[|F_{K,M_*}(re^{i\theta})|^2\big] = \E\big[|F_{K,M_*}(r)|^2\big] &= \E\bigg[\exp\Big(2\Re \sum_{k=M_*}^K\frac{X_k r^k}{\sqrt{k}}\Big)\bigg]\\
    &= \prod_{k=M_*}^K\E\bigg[\exp\Big(2\frac{r^k}{\sqrt{k}}R_k\cos(\tau_k)\Big)\bigg]\\
    &=\exp\bigg(\sum_{k=M_*}^K\frac{r^{2k}}{k}+O(1)\bigg)\asymp_{M_*} \exp\bigg(\sum_{k=M_*}^K\frac{r^{2k}}{k}\bigg)\asymp_{M_*} K,
\end{align*}where the last step uses Lemma \ref{lemma:basic}.    
\end{proof}

Let us recall definitions of tilted probability measures $\Q^{(1)}_{r,M,K}$ and $\Q^{(2)}_{r,M,K,\theta}$ in \eqref{eq:Q1 def} and \eqref{eq:dq2}.

\begin{lemma}\label{lemma-Q-computations}
Assume the same settings of Lemma \ref{SZ-lemma2.2}.  We have for $\Q^{(1)}=\Q^{(1)}_{r,M,K}$, and any  $e^M\vee k_0\leq k< K$ and $e^{-1/K}\le r\leq e^{1/K}$, 
\begin{enumerate}[(i)]
    \item $\E^{\Q^{(1)}}\left[R_k\cos(\tau_k)\right]= \frac{r^k}{\sqrt{k}}+O(k^{-1});$
    \item $\E^{\Q^{(1)}}\left[R_k^2\cos^2(\tau_k)\right]=\frac{1}{2}+O(k^{-{1}/{2}});$
    \item $\E^{\Q^{(1)}}\left[|R_k\cos(\tau_k)|^3\right]\ll 1$.
\end{enumerate}
We also have for $\Q^{(2)}=\Q^{(2)}_{r,M,K,\theta}$, and any  $e^M\vee k_0\leq k<K_r$, $\theta\in[-\pi,\pi)$, and $e^{-1/K}\le r\le e^{1/K}$,
\begin{enumerate}[(i)]
    \item $\E^{\Q^{(2)}}\left[R_k\cos(\tau_k)\right]=(1+\cos(k\theta))\frac{r^k}{\sqrt{k}}+O(k^{-1})$;
    \item $\E^{\Q^{(2)}}\left[R_k^2\cos^2(\tau_k)\right]=\frac{1}{2}+O(k^{-{1}/{2}})\text{ and } \E^{\Q^{(2)}}\left[R_k^2\cos(\tau_k)\cos(\tau_k+k\theta)\right]=\frac{\cos(k\theta)}{2}+O(k^{-{1}/{2}});$
    \item $\E^{\Q^{(2)}}\left[|R_k\cos(\tau_k)|^3\right]\ll 1$.
\end{enumerate}
 Moreover, if $X_1$ is compactly supported, the above statements hold with the choice $k_0=1$.
\end{lemma}

\begin{proof}
    By definition,  Taylor's expansion, Fubini's theorem, and Lemma \ref{SZ-lemma2.2}, we have
    \begin{align*}
        \E^{\Q^{(1)}}\left[R_k\cos(\tau_k)\right] &= \frac{\E\left[\exp(\frac{2r^k}{\sqrt{k}}R_k\cos(\tau_k))R_k\cos(\tau_k)\right]}{\E\left[\exp(\frac{2r^k}{\sqrt{k}}R_k\cos(\tau_k))\right]}= \frac{\frac{r^{k}}{\sqrt{k}}+O(k^{-1})}{1+\frac{r^{2k}}{k}+O(k^{-3/2})}= \frac{r^k}{\sqrt{k}}+O(k^{-1}),\\
        \E^{\Q^{(1)}}\left[R_k^2\cos^2(\tau_k)\right] &= \frac{\E\left[\exp(\frac{2r^k}{\sqrt{k}}R_k\cos(\tau_k))R_k^2\cos^2(\tau_k)\right]}{\E\left[\exp(\frac{2r^k}{\sqrt{k}}R_k\cos(\tau_k))\right]}= \frac{\frac{1}{2}+O(k^{-{1}/{2}})}{1+\frac{r^{2k}}{k}+O(k^{-3/2})}= \frac{1}{2}+O(k^{-{1}/{2}}),
    \end{align*}
    and \[\E^{\Q^{(1)}}\left[|R_k\cos(\tau_k)|^3\right] = \frac{\E\left[\exp(\frac{2r^k}{\sqrt{k}}R_k\cos(\tau_k))|R_k\cos(\tau_k)|^3\right]}{\E\left[\exp(\frac{2r^k}{\sqrt{k}}R_k\cos(\tau_k))\right]}\ll \frac{1}{1+\frac{r^{2k}}{k}+O(k^{-3/2})}\ll 1.\] Computation of moments under $\Q^{(2)}$ is similar. Using the rotational invariance of the distribution of $\tau_k$ in the third line, the fact $\E[e^{\beta \cos(\tau_k)}\sin(\tau_k)]=\E[e^{\beta \sin(\tau_k)}\cos(\tau_k)]= 0$ for any fixed $\beta$ in the fourth equality, and Lemma \ref{SZ-lemma2.2} in the fifth, we obtain
    \begin{align*}
        \E^{\Q^{(2)}}\left[R_k\cos(\tau_k)\right] &= \frac{\E\left[\exp(\frac{2r^kR_k}{\sqrt{k}}(\cos(\tau_k)+\cos(\tau_k+k\theta)))R_k\cos(\tau_k)\right]}{\E\left[\exp(\frac{2r^kR_k}{\sqrt{k}}(\cos(\tau_k)+\cos(\tau_k+k\theta)))\right]}\\
        &= \frac{\E\left[\exp(\frac{4r^k\cos(\frac{k\theta}{2})}{\sqrt{k}}R_k\cos(\tau_k+\frac{k\theta}{2}))R_k\cos(\tau_k)\right]}{\E\left[\exp(\frac{4r^k\cos(\frac{k\theta}{2})}{\sqrt{k}}R_k\cos(\tau_k+\frac{k\theta}{2}))\right]}\\
        &= \frac{\E\left[\exp(\frac{4r^k\cos(\frac{k\theta}{2})}{\sqrt{k}}R_k\cos(\tau_k))R_k\cos(\tau_k-\frac{k\theta}{2})\right]}{\E\left[\exp(\frac{4r^k\cos(\frac{k\theta}{2})}{\sqrt{k}}R_k\cos(\tau_k))\right]}\\
        &= \frac{\E\left[\exp(\frac{4r^k\cos(\frac{k\theta}{2})}{\sqrt{k}}R_k\cos(\tau_k))R_k\cos(\tau_k)\right]}{\E\left[\exp(\frac{4r^k\cos(\frac{k\theta}{2})}{\sqrt{k}}R_k\cos(\tau_k))\right]}\cos(\frac{k\theta}{2})\\
        &= \frac{2\cos^2(\frac{k\theta}{2})\frac{r^k}{\sqrt{k}}+O(k^{-1})}{1+4\cos^2(\frac{k\theta}{2})\frac{r^{2k}}{k}+O(k^{-3/2})}= (1+\cos(k\theta))\frac{r^k}{\sqrt{k}}+O(k^{-1}),       \end{align*}
        and analogously, taking $\beta = 2\cos(\frac{k\theta}{2}),~\alpha = \frac{k\theta}{2}$ in Lemma \ref{SZ-lemma2.2} (iv),\begin{align*}
        \E^{\Q^{(2)}}\left[R_k^2\cos^2(\tau_k)\right] &= \frac{\E\left[\exp(\frac{4r^k\cos(\frac{k\theta}{2})}{\sqrt{k}}R_k\cos(\tau_k))R_k^2\cos^2(\tau_k+\frac{k\theta}{2})\right]}{\E\left[\exp(\frac{4r^k\cos(\frac{k\theta}{2})}{\sqrt{k}}R_k\cos(\tau_k))\right]}\\
        &= \frac{\frac{1}{2}+O(k^{-{1}/{2}})}{1+4\cos^2(\frac{k\theta}{2})\frac{r^{2k}}{k}+O(k^{-3/2})}= \frac{1}{2}+O(k^{-{1}/{2}}),
    \end{align*}
    and using Lemma \ref{SZ-lemma2.2} (iii), (iv),
    \begin{align*}
        \E^{\Q^{(2)}}\left[R_k^2\cos(\tau_k)\cos(\tau_k+k\theta)\right] &= \frac{\cos(k\theta)-1}{2}\E^{\Q^{(2)}}\left[R_k^2\right]+\E^{\Q^{(2)}}\left[R_k^2\cos^2(\tau+\frac{k\theta}{2})\right]\\
        &= \frac{\cos(k\theta)}{2}+O(k^{-{1}/{2}}).
    \end{align*} Finally, \[\E^{\Q^{(2)}}\left[|R_k\cos(\tau_k)|^3\right] \ll \frac{\E[|R_k|^3]+O(\frac{1}{\sqrt{k}})}{1+4\cos^2(\frac{k\theta}{2})\frac{r^{2k}}{k}+O(k^{-3/2})}\ll 1.\] This completes the proof.
\end{proof}

Recall \eqref{eq:muk def} and \eqref{eq:nuk def}. 
The discrepancy of the sums over $\mu_k$ and $\nu_k$ is small due to the fluctuation of the cosine function, as pointed out in the next lemma. While the key arguments are already contained in \citep[equation (12.7)]{soundararajan2022model}, we provide a proof here for completeness.

\begin{lemma}\label{lemma:basic 2}We have for any $m\in\N$ and $\theta\in[-\pi,\pi)\setminus\{0\}$ that
    \begin{align}
        \Bigg|\sum_{e^{m-1}\leq k<e^m}\frac{r^{2k}\cos(k\theta)}{k}\Bigg|\ll \frac{\max\{1,r^{2e^m}\}}{|\theta|e^m}.\label{eq:A50}
    \end{align}
    In particular, if $r\leq 1$, for any  $\theta\in[-\pi,\pi)\setminus\{0\}$,
    \begin{align}
        \sum_{m\geq \log\frac{1}{|\theta|}}\Bigg|\sum_{e^{m-1}\leq k<e^m}(\mu_k-\nu_k)\Bigg|\ll 1.\label{eq:A52}
    \end{align}
    Moreover, if $e^m\leq K$ and $r\leq e^{1/K}$, we have
    \begin{align}
        \Bigg|\sum_{e^{m-1}\leq k<e^m}\frac{r^{2k}\cos(k\theta)}{k}\Bigg|\ll \frac{1}{|\theta|e^m}.\label{eq:A51}
    \end{align}
\end{lemma}

\begin{proof}
By summing the geometric series in the third step and using $|1-te^{i\theta}|\geq|\sin(\theta/2)|\geq |\theta|/\pi$, we have
\begin{align*}
    \Bigg|\sum_{e^{m-1}\leq k<e^m}\frac{r^{2k}\cos(k\theta)}{2k}\Bigg|&=\Bigg|\int_0^r\sum_{e^{m-1}\leq k<e^m}t^{2k-1}\cos(k\theta)\d t\Bigg|\\
    &\leq \Bigg|\int_0^r\sum_{e^{m-1}\leq k<e^m}t^{2k-1}e^{ik\theta}\d t\Bigg|\\
    &\leq \int_0^r \frac{\pi}{|\theta|}2t^{2\lceil e^{m-1}\rceil-1}\d t\ll \frac{r^{2\lceil e^{m-1}\rceil}}{|\theta|e^m}\leq \frac{\max\{1,r^{2e^m}\}}{|\theta|e^m}.
\end{align*}
This proves \eqref{eq:A50}. If $r\leq 1$, by \eqref{eq:muk def}, \eqref{eq:nuk def}, and \eqref{eq:A50},
\begin{align*}
    \sum_{m\geq \log\frac{1}{|\theta|}}\Bigg|\sum_{e^{m-1}\leq k<e^m}(\mu_k-\nu_k)\Bigg|&\ll 1+\sum_{m\geq \log\frac{1}{|\theta|}}\Bigg|\sum_{e^{m-1}\leq k<e^m}\frac{r^{2k}\cos(k\theta)}{k}\Bigg|\\
    &\ll 1+\sum_{m\geq \log\frac{1}{|\theta|}}\frac{1}{|\theta|e^m}\ll 1,
\end{align*}
proving \eqref{eq:A52}. If $e^m\leq K$ and $r\leq e^{1/K}$, we have $r^{2e^m}\leq e^2$. Inserting into \eqref{eq:A50} yields \eqref{eq:A51}.    
\end{proof}

Finally, we record the following estimate for the denominator of \eqref{eq:dq2}.
\begin{lemma}\label{lemma:basic3}For any $\theta\in[-\pi,\pi)\setminus\{0\}$, $M\geq \log(10^3/|\theta|)$, and $e^{-1/K}\le r\le e^{1/K}$ where $K\geq K_r$,
  \[\E\bigg[\exp\bigg(2\sum_{m=M+1}^{\log K_r}(Z_0(m)+Z_\theta(m))\bigg)\bigg]\asymp \frac{K_r^2}{e^{2M}}.\]
\end{lemma}

\begin{proof}
Recall in the proof of Lemma \ref{lemma-Q-computations} we computed
\begin{align*}
    \E\left[\exp\Big(\frac{2r^kR_k}{\sqrt{k}}(\cos(\tau_k)+\cos(\tau_k+k\theta))\Big)\right]&=1+4\cos^2(\frac{k\theta}{2})\frac{r^{2k}}{k}+O(k^{-3/2})\\
    &=1+\frac{2r^{2k}}{k}+\frac{2r^{2k}}{k}\cos(k\theta)+O(k^{-3/2}).
\end{align*}
Therefore,
\begin{align*}
    \E\bigg[\exp\bigg(2\sum_{m=M+1}^{\log K_r}(Z_0(m)+Z_\theta(m))\bigg)\bigg]&=\prod_{e^M\leq k<K_r}\bigg(1+\frac{2r^{2k}}{k}+\frac{2r^{2k}}{k}\cos(k\theta)+O(k^{-3/2})\bigg)\\
    &\asymp \exp\bigg(\sum_{e^M\leq k<K_r}\Big(\frac{2r^{2k}}{k}+\frac{2r^{2k}}{k}\cos(k\theta)\Big)\bigg).
\end{align*}
By Lemmas \ref{lemma:basic} and \eqref{eq:A51} of \ref{lemma:basic 2}, the desired statement follows.    
\end{proof}

\section{Deferred proofs from Section \texorpdfstring{\ref{sec: weighted partial mass}}{}}
\label{sec:deferred proofs}

\begin{proof}[Proof of Lemma \ref{lemma:sum by parts}]
Suppose first that $|\tau_1|\leq  N^{-1/2}$. Then 
\begin{align}
    \bigg|\sum_{j:|j-m|\leq N^{9/10}}e^{ij\tau_1}\,\frac{|x_1|^{j-m}}{j!/m!}r^j\bigg|\leq \sum_{j:|j-m|\leq N^{9/10}}\frac{|x_1|^{j-m}}{j!/m!}r^j\ll \sum_{j:|j-m|\leq N^{9/10}}\frac{m^{j-m}}{j!/m!}r^m.\label{eq:r^jr^m}
\end{align}
Using \eqref{eq:taylor}, the sum over $j\geq m$ can be bounded by
\begin{align}
    \sum_{j:0\leq j-m\leq N^{9/10}}\frac{m^{j-m}}{j!/m!}\ll \sum_{j'=0}^{N^{9/10}}e^{j'}\Big(\frac{m}{m+j'}\Big)^{m+j'}\leq \sum_{j'=0}^{N^{9/10}}\exp\Big(-\frac{(j')^2}{m}\Big)\ll \sqrt{N}.\label{eq:e^j^2}
\end{align}
The other sum over $j<m$ can be bounded similarly, leading to 
$$\bigg|\sum_{j:|j-m|\leq N^{9/10}}e^{ij\tau_1}\,\frac{|x_1|^{j-m}}{j!/m!}r^j\bigg|\ll \sqrt{N}r^m\ll \frac{r^m}{|\tau_1|}$$in the case $|\tau_1|\leq  N^{-1/2}$.

Now suppose that $|\tau_1|\geq N^{-1/2}$. A consequence of \eqref{eq:r^jr^m} and \eqref{eq:e^j^2} is that
$$\bigg|\sum_{j:|j-m|\leq N^{9/10}}e^{ij\tau_1}\,\frac{|x_1|^{j-m}}{j!/m!}\,r^j\bigg|\ll \bigg|\sum_{j:|j-m|\leq \sqrt{N}}e^{ij\tau_1}\,\frac{|x_1|^{j-m}}{j!/m!}\bigg|\,r^m+r^m.$$
On the other hand, summation by parts yields (without loss of generality, assume $\sqrt{N}\in\Z$)
$$\sum_{j:0\leq j-m\leq \sqrt{N}}e^{ij\tau_1}\,\frac{|x_1|^{j-m}}{j!/m!}=\frac{|x_1|^{\sqrt{N}}}{(\sqrt{N}+m)!/m!}\,\sum_{j=0}^{\sqrt{N}}e^{ij\tau_1}-\sum_{k=1}^{\sqrt{N}}\Big(\sum_{j=0}^{k-1}e^{ij\tau_1}\Big)\Big(\frac{|x_1|^k}{(k+m)!/m!}-\frac{|x_1|^{k-1}}{(k-1+m)!/m!}\Big).$$
Applying triangle inequality then leads to
\begin{align*}
    \bigg|\sum_{j:0\leq j-m\leq \sqrt{N}}e^{ij\tau_1}\,\frac{|x_1|^{j-m}}{j!/m!}\bigg|&\leq \frac{|x_1|^{\sqrt{N}}}{|\tau_1|(\sqrt{N}+m)!/m!}+\frac{1}{|\tau_1|}\sum_{k=1}^{\sqrt{N}}\bigg|\frac{|x_1|^k}{(k+m)!/m!}-\frac{|x_1|^{k-1}}{(k-1+m)!/m!}\bigg|\ll \frac{1}{|\tau_1|}\ll\sqrt{N}.
\end{align*}
The other sum over $j:-\sqrt{N}\leq j-m<0$ is similar. This finishes the proof.   
\end{proof}

\begin{proof}[Proof of Lemma \ref{lemma:sum by parts2}]
The main idea is that, under our assumptions, $e^{ij\tau}\approx e^{im\tau},~|x_1|\approx m,$ and $r^j\approx r^m$ for $|j-m|\leq \sqrt{C_4N}$, and the sum over $j$ with $|j-m|\geq\sqrt{C_4N}$ is negligible. Let us make these approximations precise. First, for $N$ large enough,
\begin{align*}
   \bigg|\sum_{\sqrt{C_4N}\leq |j-m|\leq N^{9/10}}e^{ij\tau_1}\,\frac{|x_1|^{j-m}}{j!/m!}r^j\bigg|&\leq   \sum_{\sqrt{C_4N}\leq |j-m|\leq N^{9/10}}\frac{|x_1|^{j-m}}{j!/m!}r^j\leq Lr^m\sum_{\sqrt{C_4N}\leq |j-m|\leq N^{9/10}}\frac{m^{j-m}}{j!/m!},
\end{align*}since $|\log r|=O(1/N)$. 
Applying the same argument in \eqref{eq:e^j^2} leads to 
$$\sum_{\sqrt{C_4N}\leq |j-m|\leq N^{9/10}}\frac{m^{j-m}}{j!/m!}\ll \sum_{j=\sqrt{C_4N}}^{N^{9/10}}\exp\Big(-\frac{j^2}{m}\Big)\ll e^{-3C_4}.$$
It follows that
\begin{align}
    \bigg|\sum_{\sqrt{C_4N}\leq |j-m|\leq N^{9/10}}e^{ij\tau_1}\,\frac{|x_1|^{j-m}}{j!/m!}r^j\bigg|\leq Lr^me^{-3C_4}.\label{eq:sum1}
\end{align}
Second, for $N$ large enough,
$$
\sum_{m-\sqrt{C_4N}< j< m+\sqrt{C_4N}}\frac{|x_1|^{j-m}}{j!/m!}r^j|e^{ij\tau}-e^{im\tau}|\ll r^m \sum_{j=1}^{\sqrt{C_4N}}\exp\Big(-\frac{j^2}{m}\Big)C_4^{-1/2}\ll r^m\sqrt{N}C_4^{-1/2}.$$
By the triangle inequality,
\begin{align}
   \bigg|\sum_{m-\sqrt{C_4N}< j< m+\sqrt{C_4N}}e^{ij\tau}\frac{|x_1|^{j-m}}{j!/m!}r^j-\sum_{m-\sqrt{C_4N}< j< m+\sqrt{C_4N}}e^{im\tau}\frac{|x_1|^{j-m}}{j!/m!}r^j\bigg| \leq Lr^m\sqrt{N}C_4^{-1/2}.\label{eq:sum2}
\end{align}
In addition,
\begin{align}
     \bigg|\sum_{m-\sqrt{C_4N}< j< m+\sqrt{C_4N}}e^{im\tau}\frac{|x_1|^{j-m}}{j!/m!}r^j\bigg| =\sum_{m-\sqrt{C_4N}< j< m+\sqrt{C_4N}}\frac{|x_1|^{j-m}}{j!/m!}r^j\geq \frac{1}{L} r^m\sqrt{N}.\label{eq:sum3}
\end{align}
Combining \eqref{eq:sum1}--\eqref{eq:sum3} and applying the triangle inequality, we conclude that 
$$\bigg|\sum_{j:|j-m|\leq N^{9/10}}e^{ij\tau}\,\frac{|x_1|^{j-m}}{j!/m!}r^j\bigg|\geq \frac{1}{L} r^m\sqrt{N}-Lr^me^{-3C_4}-Lr^m\sqrt{N}C_4^{-1/2}.$$
The claim then follows by picking $C_4$ large enough.
\end{proof} 
\end{document}